\documentclass[11pt]{amsart}
\usepackage{amsmath,amsfonts,amsthm, amssymb}

\usepackage[margin=1.1in]{geometry}

\usepackage{hyperref}

\usepackage{todonotes}
\usepackage{comment}

\newtheorem{theorem}{Theorem}
\newtheorem*{theorem*}{Theorem}
\newtheorem{lemma}[theorem]{Lemma}

\newtheorem{proposition}[theorem]{Proposition}

\newtheorem{corollary}[theorem]{Corollary}
\newtheorem{assumption}[theorem]{Assumption}

\newtheorem{definition}[theorem]{Definition}

\theoremstyle{remark}
\newtheorem{remark}[theorem]{Remark}

\theoremstyle{remark}
\newtheorem{example}[theorem]{Example}

\numberwithin{theorem}{section}

\newcommand{\E}{\mathbb{E}}
\renewcommand{\P}{\mathbb{P}}

\newcommand{\Ai}{\mathrm{Ai}}

\newcommand{\Cov}{\mathop{\mathrm{Cov}}}
\newcommand{\Var}{\mathop{\mathrm{Var}}}

\newcommand{\vxi}{\vec{\xi}}
\newcommand{\veta}{\vec{\eta}}

\newcommand{\mF}{\mathcal{F}}

\newcommand{\X}{\mathcal X}

\newcommand{\eps}{\varepsilon}
\newcommand{\ii}{\mathbf i}

\renewcommand{\Im}{\mathrm{Im}\,}

\renewcommand{\a}{\mathfrak a}

\newcommand{\1}{\mathbf 1}

\renewcommand{\d}{\mathrm d}

\numberwithin{equation}{section}

\title{Universal objects of the infinite beta random matrix theory}
\author{V. Gorin}

\address{University of Wisconsin -- Madison and Institute for Information Transmission Problems of Russian Academy of Sciences}
\email{vadicgor@gmail.com}

\author{V. Kleptsyn}
\address{Univ Rennes, CNRS, IRMAR - UMR 6625, F-35000 Rennes, France}
\email{victor.kleptsyn@univ-rennes1.fr}

\begin{document}
\maketitle
\begin{abstract}
 We develop a theory of multilevel distributions of eigenvalues which complements the Dyson's threefold $\beta=1,2,4$ approach corresponding to real/complex/quaternion matrices  by $\beta=\infty$ point. Our central objects are the G$\infty$E ensemble, which is a counterpart of the classical Gaussian Orthogonal/Unitary/Symplectic ensembles, and the Airy$_{\infty}$ line ensemble, which is a collection of continuous curves serving as a scaling limit for largest eigenvalues at $\beta=\infty$. We develop two points of views on these objects. The probabilistic one treats them as partition functions of certain additive polymers collecting white noise. The integrable point of view expresses their distributions through the so-called associated Hermite polynomials and integrals of the Airy function. We also outline universal appearances of our ensembles as scaling limits.
\end{abstract}

\tableofcontents

\section{Introduction}

\sloppy

\subsection{Motivations}
Traditionally, the random matrix theory\footnote{See, e.g., textbooks  \cite{AGZ,Mehta,ABF,For} for general reviews.}  deals with real, complex, and quaternion matrices, their eigenvalues and eigenvectors. Following the work of Wigner, Dyson, Mehta, and others in 1950-60s, a central role is played by Gaussian ensembles, which are defined as follows: let $X$ be an infinite $\mathbb Z_{>0}\times \mathbb Z_{>0}$ matrix with i.i.d.\ standard normal real/complex/quaternion matrix elements, normalized so that their real parts have variance $\tfrac{2}{\beta}$ with $\beta=1/2/4$, respectively. The $N\times N$ principal submatrix $M_N$ of $\tfrac{X+X^*}{2}$ is then called the Gaussian Orthogonal/Unitrary/Simplectic ensemble of rank $N$. The matrix $M_N$ is Hermitian, it has $N$ real eigenvalues $\chi_1\le \chi_2\le \dots \le \chi_N$ and their distribution is explicit. The joint density is proportional to
\begin{equation} \label{eq_GbetaE}
 \prod_{1\le i < j \le N} (\chi_j-\chi_i)^{\beta} \prod_{i=1}^N \exp\bigl( -\tfrac{\beta}{4} (\chi_i)^2 \bigr).
\end{equation}
Although originally in \eqref{eq_GbetaE} only $\beta=1,2,4$ come out, the formula suggests the possibility of taking arbitrary positive real values for $\beta$. In the terminology of statistical mechanics, such $\beta$ can be interpreted as the inverse temperature. More recently the distribution \eqref{eq_GbetaE} was found in \cite{DE_tridiag} to govern for any $\beta>0$ the eigenvalues of tridiagonal real symmetric random matrices. Multiple other reasons to be interested in the Gaussian $\beta$ ensembles \eqref{eq_GbetaE} with arbitrary $\beta>0$ are reviewed in \cite[Chapter 20 ``Beta Ensembles'']{ABF}, they include connections to the theory of Jack and Macdonald symmetric polynomials, to Coulomb log-gases, and to the Calogero--Sutherland quantum many--body system. One can go further and replace $\exp\bigl( -\tfrac{\beta}{4} (\chi_i)^2 \bigr)$ in \eqref{eq_GbetaE} by any potential $V(\chi_i)$ leading to a class of distributions known under the name $\beta$--ensembles.

Beyond $\beta=1,2,4$, there are two other special values of $\beta$ for $\beta$--ensembles. First, at $\beta=0$ the interactions between particles disappear and we link to the classical probability theory dealing with sequences of independent random variables. We are not going to consider this value here. Instead, we concentrate on $\beta=\infty$, following \cite{DE_large_beta,AKM,EPS,GM2018,VW_functional_CLT}. The point of view of \cite{DE_large_beta,EPS} is that many characteristics of the distribution \eqref{eq_GbetaE} (such as mean and variance of individual eigenvalues $x_i$ for finite $N$ and as $N\to\infty$) are well-approximated by Taylor expansions near $\beta=\infty$. In particular, their numeric simulations show a good match between the first two non-trivial asymptotic terms and exact expressions even at $\beta=1$, which seems very far from $\beta=\infty$. Our own simulations for the Gaussian ensembles of $3\times 3$ matrices are shown in Figure \ref{Fig_infinity_predictions}. We see an astonishing match between exact probability densities and their approximations from $\beta=\infty$.

 \begin{figure}[t]
\begin{center}
\includegraphics[width=0.32\linewidth]{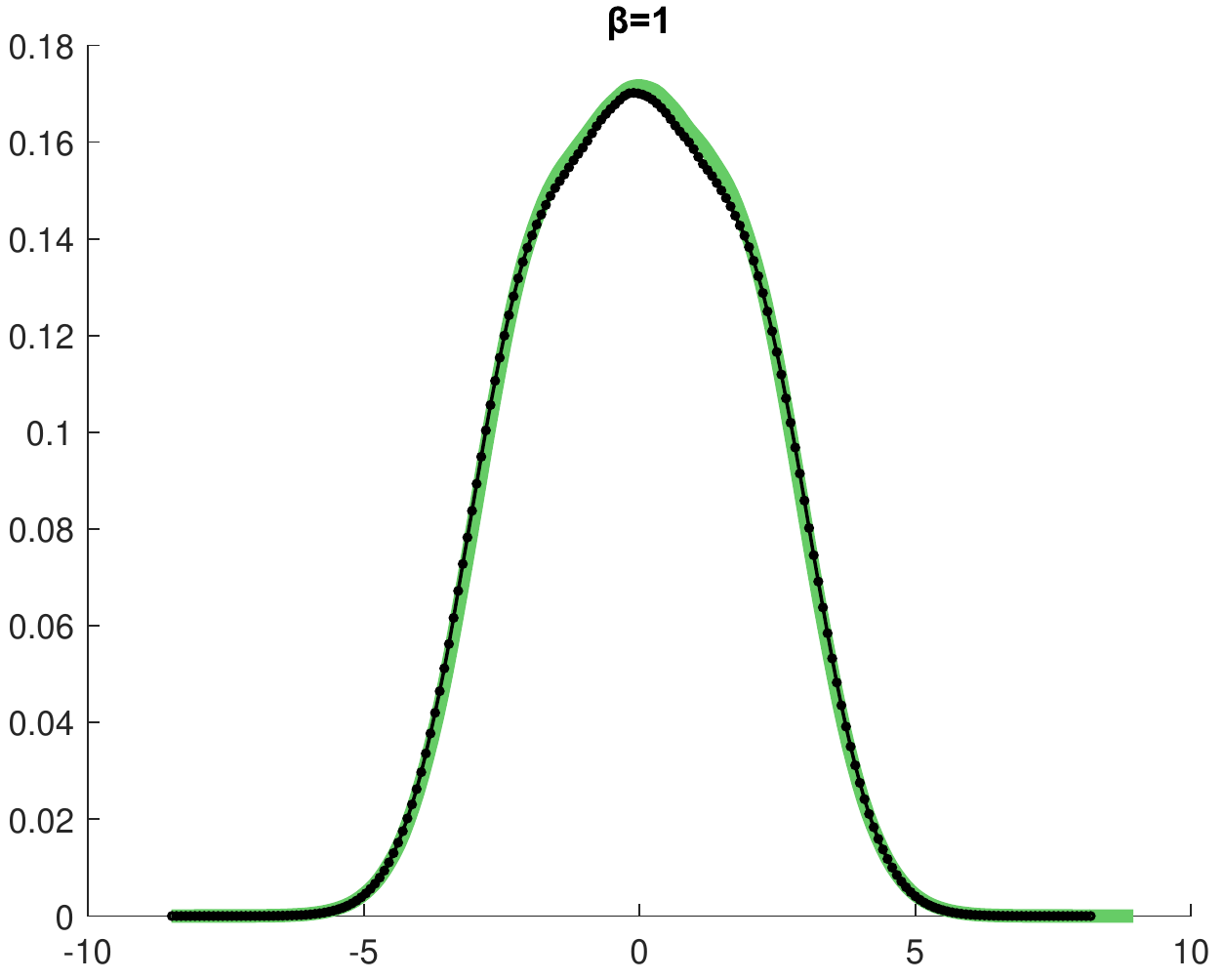}
\includegraphics[width=0.32\linewidth]{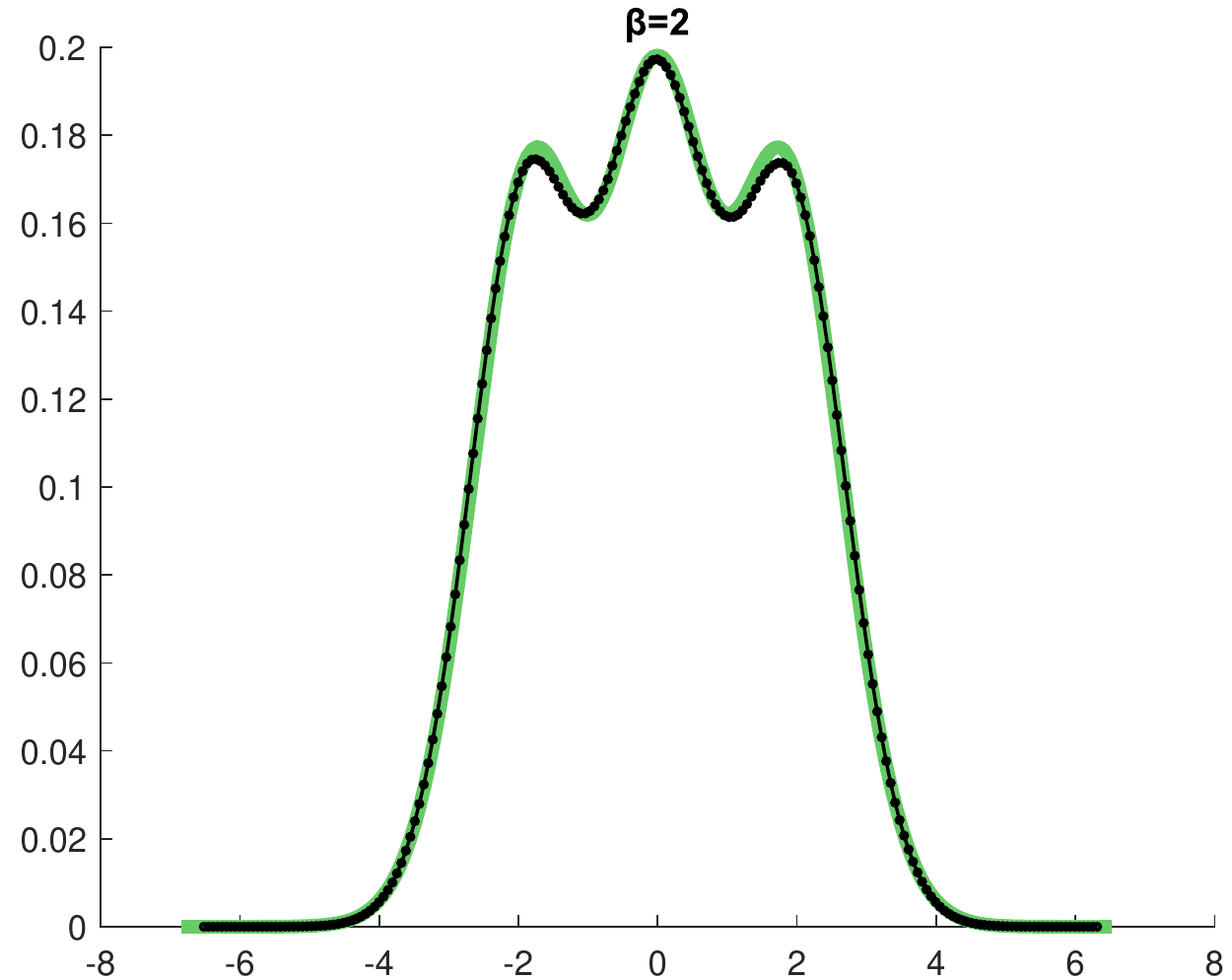}
\includegraphics[width=0.32\linewidth]{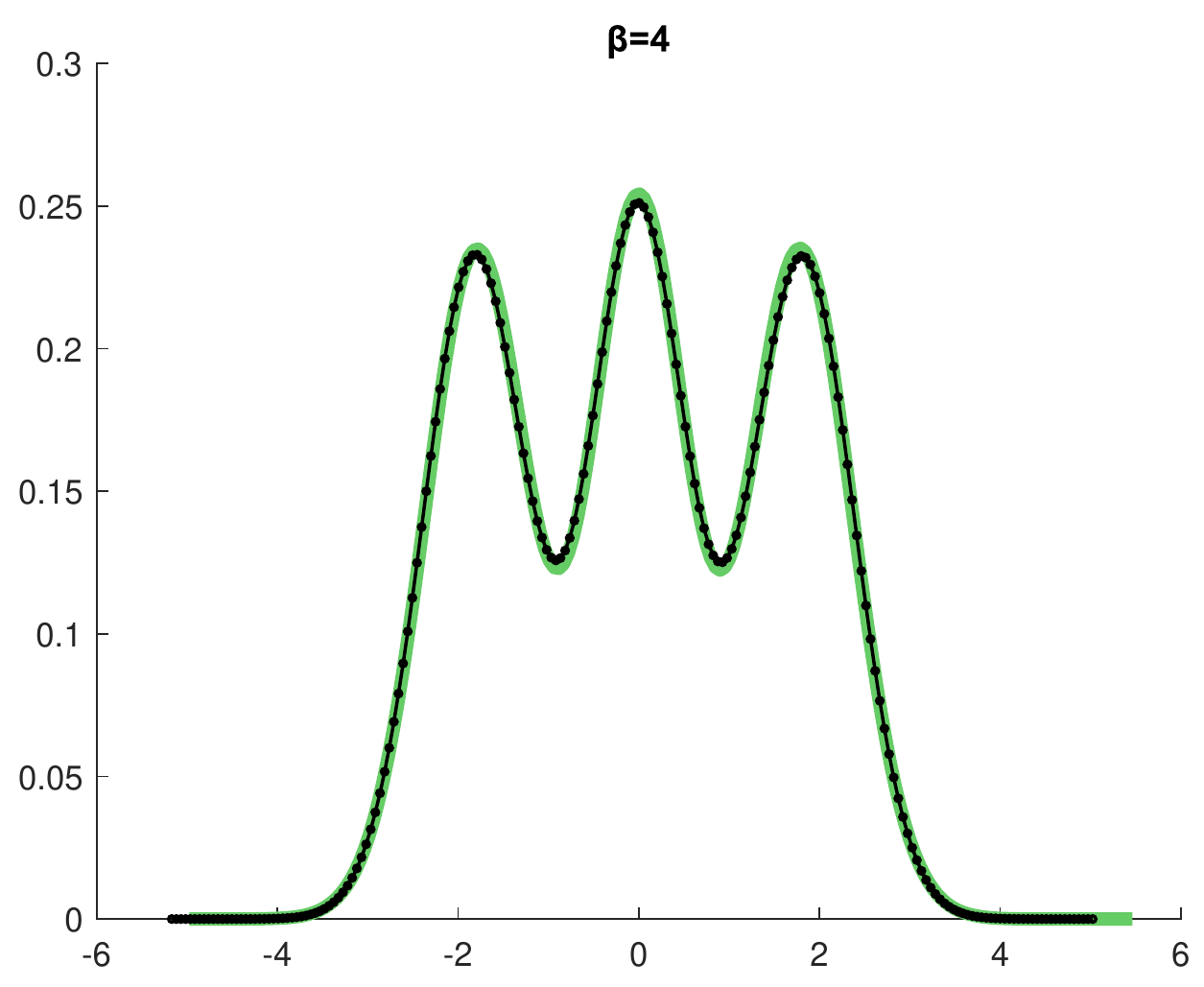}
 \caption{The figures show arithmetic mean of the probability densities (MATLAB simulation using $5\times 10^6$ samples) of $3$ eigenvalues of $3\times 3$ matrices. Light green solid lines correspond to eigenvalues sampled from G$\beta$E ensembles \eqref{eq_GbetaE} at $\beta=1/2/4$, $N=3$. Black dash-dotted lines correspond to the result of the 3--term approximation of eigenvalues of the form $\chi_i=h_i+\frac{1}{\sqrt{\beta}} \xi_i + \frac{1}{\beta} r_i$, $i=1,2,3$, where $(h_1,h_2,h_3)=(-\sqrt{3},0,\sqrt{3})$ are roots of the degree $3$ Hermite polynomial, $(\xi_1,\xi_2,\xi_3)$ is a Gaussian vector, whose study is one of our  topics, and $(r_1,r_2,r_3)$ is a deterministic vector not discussed in this text. \label{Fig_infinity_predictions}}
\end{center}
\end{figure}

The $\beta=\infty$ ensembles or, equivalently, the behavior of $\beta$--ensembles at large values of $\beta$ is the central theme of this article. As we explain in Section \ref{Section_inf_ensembles}, a $\beta=\infty$ ensemble consists of two pieces of data: The first one is a deterministic particle configuration, which is a $\beta\to\infty$ limit of $\beta$-ensembles, such as \eqref{eq_GbetaE}; the second piece is a Gaussian vector describing asymptotic fluctuations around this limit. We would like to combine large $\beta$ with large $N$. In other words, we deal with asymptotic questions about large-dimensional ensembles of $\beta=\infty$ random matrices.

We discover that the $\beta=\infty$ point possesses a lot of integrability and the asymptotic questions can be understood in precise details, going far beyond what is known for general values of $\beta>0$. This is our main message: $\beta=\infty$ is accessible to the same extent as the most well-studied point $\beta=2$.

\subsection{Second dimension and asymptotics}

\label{Section_intro_as_results}

For our asymptotic results an important role is played by an extension of $\beta$--ensembles to two-dimensional systems. In fact, there are two distinct extensions, which are both very natural. The first one originates in \cite{Dyson}, where Dyson suggested in 1960s to identify \eqref{eq_GbetaE} with a fixed time distribution of the Dyson Brownian Motion. The latter is an $N$--dimensional stochastic evolution $(X_1(t)\le \dots \le X_N(t))$, solving the SDE:
\begin{equation}
\label{eq_DBM} \d X_i(t) = \sum_{j\ne i} \frac{\d t}{X_i(t)-X_j(t)} +\sqrt{\frac{2}{\beta}}\,  \d W_i(t), \quad i=1,2,\dots,N,
\end{equation}
where $W_i(t)$ are independent standard Brownian motions. One shows that at $t=1$ the law of the solution of \eqref{eq_DBM} with zero initial condition  $X_1(0)=\dots=X_N(0)=0$ is given by \eqref{eq_GbetaE}. \cite{Dyson} constructed the evolution \eqref{eq_DBM} at $\beta=1,2,4$ as a projection onto the eigenvalues of a dynamics on Hermitian matrices in which each matrix element evolves as a Brownian motion. Yet, \eqref{eq_DBM} makes sense for any\footnote{For $\beta<1$ additional care is required, since the particles start to collide with each other, see \cite{CL}.} $\beta>0$. The Dyson Brownian Motion is a key ingredient in proofs of many recent limit theorems for random matrices and $\beta$--ensembles, see, e.g., \cite{AGZ}, \cite{EY_book}.

Another $2d$ extension is constructed by considering the joint distribution of eigenvalues of all principal top-left $N\times N$ corners of the infinite Hermitian matrix $\tfrac{X+X^*}{2}$ simultaneously for $N=1,2,3,\dots$. In this way one arrives at an array of numbers $\{\chi^{k}_i\}_{1\le i \le k}$, where $\chi^k_1\le \chi^j_2\le \dots \le \chi^j_j$ are the eigenvalues of $k\times k$ corner. The eigenvalues satisfy deterministic inequalities $\chi_i^{k+1}\le \chi_{i}^k\le \chi_{i+1}^{k+1}$ and the law of the subarray $\{\chi_i^k\}_{1\le i \le k \le N}$ has density proportional to
\begin{equation}\label{eq:beta-Gauss_corner}
  \prod_{k=1}^{N-1} \left[\prod_{1\le i<j\le k} (\chi_j^k-\chi_i^k)^{2-\beta}\right] \cdot \left[\prod_{a=1}^k \prod_{b=1}^{k+1}
 |\chi^k_a-\chi^{k+1}_b|^{\beta/2-1}\right] \cdot \prod_{i=1}^N \exp\bigl( -\tfrac{\beta}{4} (\chi_i^N)^2 \bigr).
\end{equation}
We call this distribution the \emph{Gaussian $\beta$ corners process}. Modern computations leading to \eqref{eq:beta-Gauss_corner} for $\beta=1,2,4$ can be found in \cite{Baryshnikov} and \cite{Neretin}, while the underlying ideas arose in representation theory back in 1950s, see \cite[Section 9.3]{GelfandNaimark}. The consistency between \eqref{eq:beta-Gauss_corner} and \eqref{eq_GbetaE} is automatic from the construction at $\beta=1,2,4$, but needs an additional argument at general $\beta>0$, which can be obtained either using a 100-year old integration identity from \cite{Dixon} (see also \cite{Anderson}) or as a limiting case of the branching rules for Jack and Macdonald symmetric polynomials, see \cite[Appendix]{BG_Jacobi}, \cite{GS_DBM}.

Beyond intrinsic interest, the multilevel distributions \eqref{eq:beta-Gauss_corner} were used recently to prove asymptotic theorems leading to the one-level distribution \eqref{eq_GbetaE}. The central idea here is that the multilevel distribution can be uniquely identified by some of its simple features, which \eqref{eq_GbetaE} is lacking, such as conditional uniformity at $\beta=2$ (notice that most of the factors in \eqref{eq:beta-Gauss_corner} disappear at $\beta=2$), see \cite{Gorin_ASM,Dimitrov_6v}. In wider contexts, the usefulness of similar multilevel distributions and their characteristic Gibbs properties was demonstrated, e.g., in \cite{CH,CD,CGH}.

In this text we focus our attention on the largest eigenvalues in $\beta$--ensembles and their 2d extensions. Let us state two of our main results.
 We use the notation $\Ai(x)$ for the Airy function and we let $\a_1>\a_2>\a_3>\dots$ to be its zeros.

\begin{theorem} \label{Theorem_Gcorners_limit_intro} Suppose that an infinite random array $\{\chi_i^k\}_{1\le i \le k}$ is distributed so that for each $N$ its projection onto indices $1\le i \le k \le N$ has the law \eqref{eq:beta-Gauss_corner}. In addition, for each $k=1,2,\dots$, let $x_1^k<x_2^k<\dots<x_k^k$ be the roots of the degree $k$ Hermite polynomial\footnote{Here and below we use the monic ``probabilists'' Hermite
polynomials with
weight function $e^{-x^2/2}$.}
and set
$
\kappa(t)=N+ \lfloor 2 t N^{2/3}\rfloor.
$
 Then we have the following limit in the sense of convergence of finite-dimensional distributions of the two-dimensional stochastic process:
 $$
  \lim_{N\to\infty} \lim_{\beta\to\infty} N^{1/6} \sqrt{\beta} \biggl(\chi^{\kappa(t)}_{\kappa(t)+1-i}-x^{\kappa(t)}_{\kappa(t)+1-i}\biggr)=\mathfrak Z(i,t), \quad i\in\mathbb Z_{>0}, t\in \mathbb R,
 $$
 where $\mathfrak Z(i,t)$ is a mean--zero Gaussian process with covariance
 \begin{equation}
 \label{eq_Edge_limit_covariance}
 \E \mathfrak Z(i,t) \mathfrak Z(j,s)= \frac{2}{\Ai'(\a_i) \Ai'(\a_j)} \int_0^{\infty} \Ai(\a_i+y) \Ai(\a_j+y) \exp\left(-|t-s| y\right) \frac{\d y}{y}.
 \end{equation}
\end{theorem}

Notably, for the Dyson Brownian Motion the limit turns out to be the same.  More specifically, while the $t$ parameter in Theorem \ref{Theorem_Gcorners_limit_intro} refers to the difference in the size of a submatrix, in Theorem \ref{Theorem_DBM_limit_intro} below the size of the matrix is fixed and $t$ is time in the stochastic evolution. And still we are getting the same limit behavior.\footnote{We conjecture that the same is true for each $\beta>0$: if we remove $\lim_{\beta\to\infty}$ from Theorems \ref{Theorem_Gcorners_limit_intro} and \ref{Theorem_DBM_limit_intro}, then the $N\to\infty$ limits should still coincide. Heuristically, one reason is that transition probabilities for the dynamics in both theorems can be obtained by specializations and limits from (skew) Jack polynomials, see \cite{GS_DBM}.}

\begin{theorem} \label{Theorem_DBM_limit_intro} Suppose that the $N$--dimensional dynamics $\bigl(X_i(t)\bigr)_{i=1}^N$ solves \eqref{eq_DBM} with $X_1(0)=\dots=X_N(0)=0$.
In addition, for each $k=1,2,\dots$, let $x_1^k<x_2^k<\dots<x_k^k$ be the roots of the degree $k$ Hermite polynomial and set
$
\tau(t)=1+  2t N^{-1/3}.
$
 Then we have the following limit in the sense of convergence of finite-dimensional distributions of the two-dimensional stochastic process:
 $$
  \lim_{N\to\infty} \lim_{\beta\to\infty} N^{1/6} \sqrt{\beta} \biggl(X_{N+1-i}(t)-\left(\tau(t)\tfrac{\beta}{2}\right)^{1/2} x^{N}_{N+1-i}\biggr)=\mathfrak Z(i,t), \quad i\in\mathbb Z_{>0}, t\in \mathbb R.
 $$
\end{theorem}
\begin{remark}
In both Theorems \ref{Theorem_Gcorners_limit_intro} and \ref{Theorem_DBM_limit_intro} we deal with an iterative limit, i.e.\ we first send $\beta\to\infty$ and then $N\to\infty$. One could expect  that the joint limit $N,\beta\to\infty$ is the same, yet we do not prove such results in this text.
\end{remark}

The limiting process $\mathfrak Z(i,t)$ can be defined in such a way that for each fixed $i=1,2,\dots$, it becomes an almost surely continuous function of $t$, see Section \ref{Section_continuity} for a proof and Figure \ref{Fig_limit_object} for a simulation. While we are not going to provide details in this direction, we expect that convergence in Theorems \ref{Theorem_Gcorners_limit_intro} and \ref{Theorem_DBM_limit_intro} can be upgraded to convergence in law in an appropriate space of continuous functions.

 \begin{figure}[t]
\begin{center}
\includegraphics[width=0.44\linewidth]{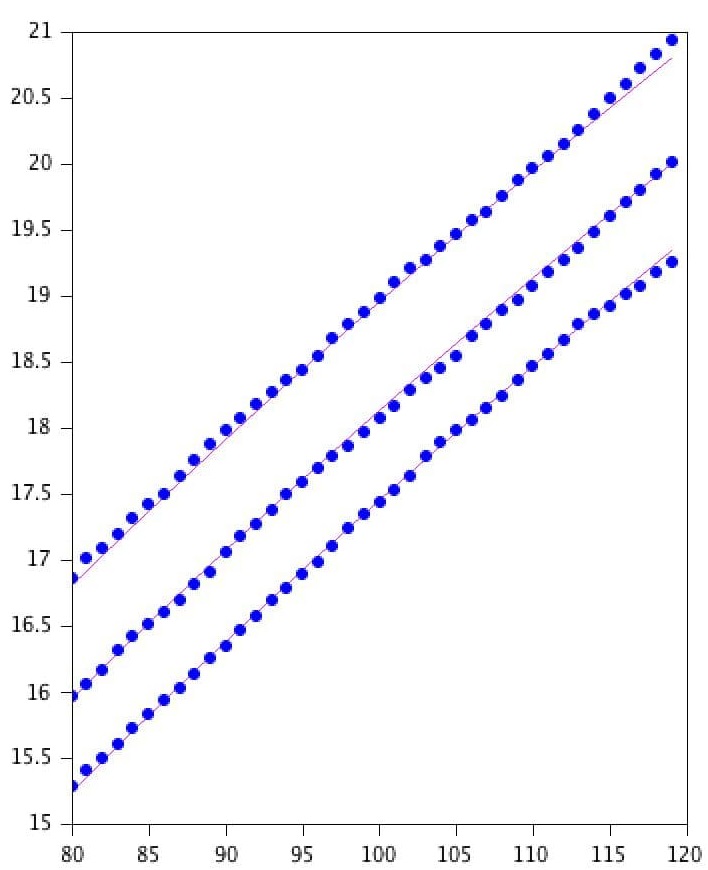}\qquad\qquad
\includegraphics[width=0.45\linewidth]{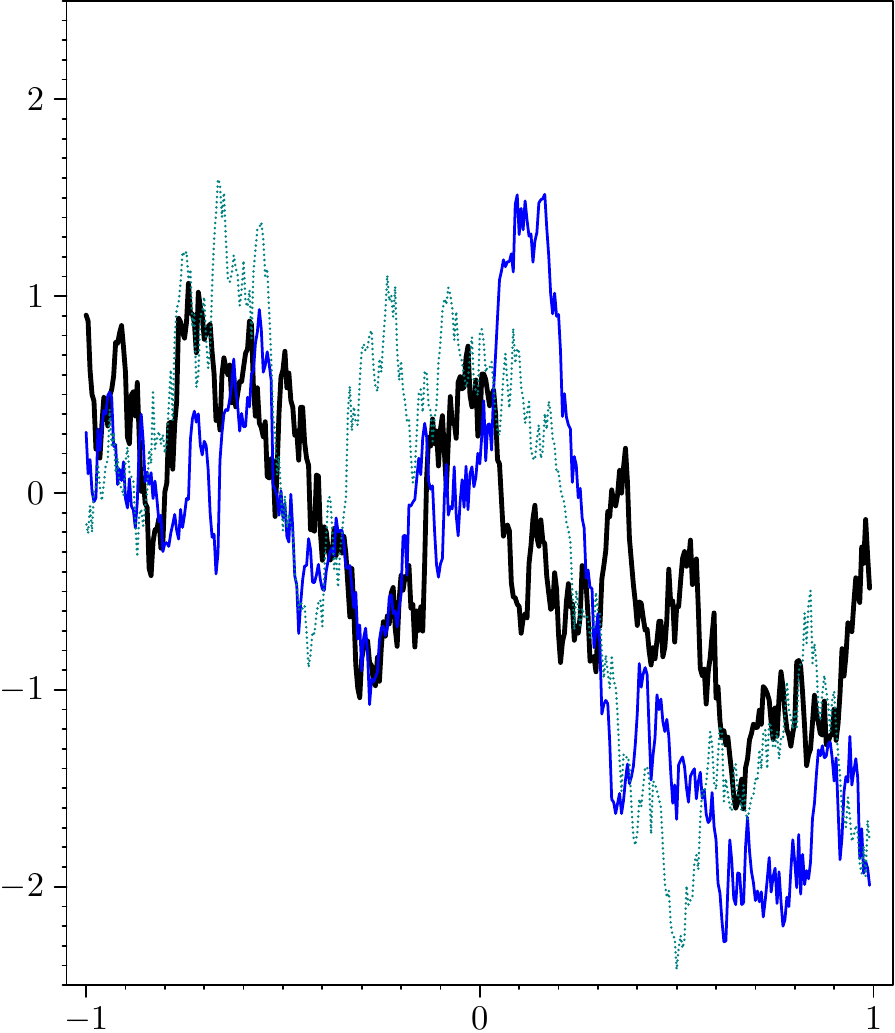}
 \caption{Left panel: Bullets show random sample of the three largest eigenvalues in the Gaussian $\beta$-corners process for corners of size $k=80,\dots,119$ and with $\beta=50$. Thin lines are corresponding roots of the Hermite polynomials. Right panel: A random sample of the limiting process $\mathfrak Z(i,t)$ for $-1\le t \le 1$; black think line for $i=1$, blue solid line for $i=2$, cyan dotted line for $i=3$. \label{Fig_limit_object}}
\end{center}
\end{figure}

In addition to the explicit formula for the covariances \eqref{eq_Edge_limit_covariance} we develop an equivalent stochastic point of view on the limiting process  $\mathfrak Z(i,t)$, $i\in\mathbb Z_{>0}$, $t\in \mathbb R$, appearing in Theorems \ref{Theorem_Gcorners_limit_intro} and \ref{Theorem_DBM_limit_intro}.
For that we consider a continuous time homogeneous Markov chain $\mathcal X_{(x_0)}(t)$, $t\ge 0$, taking values in state space $\mathbb Z_{>0}$. The initial value is $x_0\in\mathbb Z_{>0}$, i.e.\ $\mathcal X_{(x_0)}(0)=x_0$. For $i, j\in \mathbb Z_{>0}$ we define the intensity of the jump from $i$ to $j$ to be:
$$
 Q(i\to j)=\frac{2}{(\a_i-\a_j)^2}.
$$
 The transitional probabilities $P_t(i\to j)$ for this Markov chain can be expressed through integrals of the Airy function, as we explain in Section \ref{Section_random_walk_representation}.

 Next, we take a countable collection of Brownian motions $W^{(i)}(t)$, $i\in\mathbb Z_{>0}$. For each $i=1,2,\dots$ and $t\in\mathbb R$ we can identify $\mathfrak Z(i,t)$ with the following random variable:
 \begin{equation}
 \label{eq_Z_intro}
   \mathfrak Z(i,t)=2 \, \E_{\mathcal X^{(i)}(r),\, r\ge 0} \int_{r=0}^{\infty} \d W^{(\mathcal X^{(i)}(r))}(t+r).
 \end{equation}
 In words, we start the Markov chain $\mathcal X$ from $i$ at time $t$, follow its trajectory, and collect the white noises $\dot{W}^{(j)}$ along it. $\mathfrak Z(i,t)$ is the expectation over the randomness coming from $\mathcal X$; it is still a random variable with randomness coming from the Brownian motions. Alternatively, we can view $\mathfrak Z(i,t)$ as the partition function of a directed polymer in additive Gaussian noise. The form of the expression \eqref{eq_Z_intro} is a bit vague, since it is unclear how to compute the $r$-integral, as it seems to be infinite. A more mathematically precise (but, perhaps, less elegant) form is obtained by swapping the integration and expectation signs, resulting in the following expression (see Theorem \ref{Theorem_Z_as_polymer}):
 \begin{equation}
 \label{eq_Z_intro_2}
   \mathfrak Z(i,t)=2 \sum_{j=1}^{\infty} \int_{r=t}^{\infty} P_{r-t}(i\to j) \d W^{(j)}(r).
 \end{equation}
The decay of $P_{r-t}(i\to j)$ as either $r\to\infty$ or $j\to\infty$ implies that \eqref{eq_Z_intro_2} is well-defined.

Note that the representation \eqref{eq_Z_intro_2} implies that the correlations between $\mathfrak Z(i,t)$ and $\mathfrak Z(j,s)$ are always positive. This agrees with our simulation in the right panel of Figure \ref{Fig_limit_object}, which gives a feeling of attraction between the trajectories of the particles.  In contrast, for finite $\beta$ the drift of the Dyson Brownian Motion \eqref{eq_DBM} is leading to a repulsion rather than an attraction.

We call the process $\mathfrak Z(i,t)$, $i=1,2,\dots$, $t\in\mathbb R$, the \emph{Airy$_{\infty}$ line ensemble} and we treat its definition and appearance in Theorems \ref{Theorem_Gcorners_limit_intro}, \ref{Theorem_DBM_limit_intro} as the central results of our text.

\subsection{Comparison to previous results}

Most results about the asymptotic behavior of $\beta$--ensembles are available for single level ensembles as in \eqref{eq_GbetaE}. At $\beta=1,2,4$ the detailed understanding can be achieved through the theory of determinantal/Pfaffian point processes, which encode the probabilistic information in a function of two variables called a correlation kernel. This kernel is expressed through orthogonal polynomials, which makes its asymptotics accessible. In particular, the scaling limit for the largest eigenvalues of the Gaussian Orthogonal/Unitary/Symplectic ensembles, their connections to the Airy functions and Painleve equations were developed in \cite{TW-U,TW-O,For_edge}.

 At general values of $\beta>0$ the available approach is very different. It starts from the realization of the ensemble as an eigenvalue distribution of certain tridiagonal matrices, analyzes asymptotics of these matrices, and in this way identifies the scaling limits of the largest eigenvalues with (highly non-linear) functionals of Brownian motion, see \cite{DE_tridiag,ES,RRV,GS_moments} for different faces of this approach. We refer to $\beta=1,2,4$ approach as \emph{integrable} and general $\beta>0$ one as \emph{probabilistic}. To a large extent they are disjoint and many results are hard to translate from one language into another: for instance, the match between expected Laplace transform of largest eigenvalues computed in two ways in \cite{GS_moments} gave rise to a brand new distributional identity for integrated local times of the Brownian excursion. From this perspective, our $\beta=\infty$ results are an exception, since we are able to match explicit covariance \eqref{eq_Edge_limit_covariance} of $\mathfrak Z(i,t)$ with its stochastic representation \eqref{eq_Z_intro}, \eqref{eq_Z_intro_2}.

In principle, tridiagonal matrices can be used to study certain marginals of $\mathfrak Z(i,t)$. In particular, by using this approach  \cite{DE_large_beta,EPS} produced a formula for the variance of the individual components of $\mathfrak Z(i,t)$. In other words, they present\footnote{While the article formulates the statement for all $i>0$, the supporting argument is given only for $i=1,2$. On the other hand, they also analyze the limit in different order $\lim_{\beta\to\infty} \lim_{N\to\infty}$.} a one-point version of Theorems \ref{Theorem_Gcorners_limit_intro}, \ref{Theorem_DBM_limit_intro}. Interestingly, while their formula also involves an integral of Airy function, but it is of different form than $i=j$, $t=s$ specialization of \eqref{eq_Edge_limit_covariance} --- yet, numerically both formulas output the same numbers.

When it comes to the 2d extensions of $\beta$-ensembles, many results are again available at $\beta=2$. The $N\to\infty$ limiting object for the largest eigenvalues is called the Airy Line Ensemble --- it is a determinantal point process with correlation kernel expressed through the Airy functions, and it also enjoys a Brownian Gibbs resampling property, see \cite{FN,Macedo,FNH,CH}, and \cite[Section 4.4]{Ferrari}  for the analogues of our Theorems \ref{Theorem_Gcorners_limit_intro}, \ref{Theorem_DBM_limit_intro} at $\beta=2$. For $\beta=1$ the $N\to\infty$ limit of the largest eigenvalues for a common 3d extension of the corners process \eqref{eq:beta-Gauss_corner} and the Dyson Brownian Motion \eqref{eq_DBM} was computed in \cite{Sodin}.

Outside $\beta=1,2$, the available information about joint distributions of the $N\to\infty$ limit of either corners process or the Dyson Brownian Motion is very limited. Developing proper understanding of these objects remains a major open problem\footnote{On the technical side the problem stems from the fact that tridiagonal matrices (which were instrumental in understanding limits of $\beta$-ensembles) are not compatible with 2d extensions.}. One possible approach is to give a proper mathematical meaning to $N\to\infty$ limit of the Dyson Brownian Motion SDE \eqref{eq_DBM} and to the notion of its solution, see \cite{OT} and references therein. There are still technical difficulties when analyzing largest eigenvalues through this approach outside $\beta=1,2,4$.  For the bulk limits (i.e.\ for the eigenvalues in the middle of the spectrum) such an SDE point of view was put on rigorous grounds in \cite{Tsai}\footnote{One can similarly restate the corners process \eqref{eq:beta-Gauss_corner} as a Markov chain with time coordinate given by $k$. For this process the bulk limit is also available, see \cite{NV} and \cite{Huang}.}. Yet, even after we manage to convince ourselves that SDE \eqref{eq_DBM} has a proper large $N$ limit, it would still remain unclear how to solve the limiting equations. From this point of view, Theorems \ref{Theorem_Gcorners_limit_intro}, \ref{Theorem_DBM_limit_intro} are the first results computing the precise probabilistic characteristics of $N\to\infty$ limit of joint distributions of largest eigenvalues at several times or levels outside $\beta=1,2,4$.

\smallskip

One conceptual feature which unites our $\beta=\infty$ study with classical $\beta=1,2,4$ cases is that the infinitely-dimensional limiting process gets identified through a function of finitely many variables (two variables if we speak about one-level distributions as in \eqref{eq_GbetaE} or four variables if we deal with $2d$ extensions as in \eqref{eq:beta-Gauss_corner}). However, the role of this function becomes different: at $\beta=1,2,4$ the description proceeds in terms of the correlation kernels of determinantal or Pfaffian point processes, while at $\beta=\infty$ we deal with Gaussian processes uniquely fixed by their covariances.  Still, in all the situations the limiting behavior of largest eigenvalues gets expressed through the Airy functions. A vague theoretical physics analogy suggests to call $\beta=1,2,4$ results fermionic, while our $\beta=\infty$ theorems being a bosonic counterpart.

\subsection{Universality}

\label{Section_univ}

We expect that the Airy$_{\infty}$ line ensemble appears in $\beta,N\to\infty$ regime in many other problems going well beyond Theorems \ref{Theorem_Gcorners_limit_intro}, \ref{Theorem_DBM_limit_intro}. We are not going to pursue this universality direction here, let us only mention possible setups, where the appearance of the Airy$_{\infty}$ line ensemble seems plausible:

\begin{enumerate}
  \item The corners process \eqref{eq:beta-Gauss_corner} and the Dyson Brownian Motion \eqref{eq_DBM} have a common 3d extension, which is a stochastic evolution on arrays of interlacing eigenvalues constructed in \cite{GS_DBM}. We expect that the scaling limit of the largest eigenvalues in a 2d section of such evolution along a space-like path (i.e.\ along a sequence of times and corner sizes $(t_i,k_i)$, satisfying $t_1\le t_2\le t_3\le \dots$, $k_1\ge k_2\ge k_3\ge\dots$) should converge to $\mathfrak Z(i,t)$ as $\beta,N\to\infty$. Results of this type for $\beta=1,2$ were proven in \cite{Ferrari,Sodin}.
  \item One can replace $\exp\bigl( -\tfrac{\beta}{4} (\chi_i)^2 \bigr)$ in \eqref{eq_GbetaE} by a more general potential $V(\chi_i)$ and the resulting formula would give the stationary distribution for a version of the Dyson Brownian Motion with an additional drift term (see, e.g., \cite{LLX, Adhikari_Huang} and references therein for more details on the Dyson Brownian Motion with a potential). In a slightly different direction, one can also start the Dyson Brownian Motion from more complicated initial conditions than $X_1(0)=\dots=X_N(0)=0$ which we consider. One could hope that an analogue of Theorem \ref{Theorem_DBM_limit_intro} holds in such settings under mild restrictions on $V(\chi)$ and on initial conditions.
  \item One can modify the definition of the corners process \eqref{eq:beta-Gauss_corner} by replacing {$\exp\bigl( -\tfrac{\beta}{4} (\chi_i^N)^2 \bigr)$}. The most extreme case is obtained if we remove this factor altogether and instead impose deterministic equalities $\chi_i^N=y_i$, $i=1,2,\dots,N$. At $\beta=1,2,4$ this corresponds to taking an $N\times N$ Hermitian matrix with deterministic eigenvalues and uniformly random orthonormal eigenvectors and considering the law of eigenvalues of its principal corners. In contrast to \eqref{eq:beta-Gauss_corner} the definition is not going to be consistent over varying $N$ (if we replace $N$ by $N+1$, then $\chi_i^N$ become random and can no longer be deterministic), yet we can assume that $(y_1,\dots,y_N)$ changes with $N$ in a regular way as $N\to\infty$ and then analyze the behavior of the largest eigenvalues of corners of size $\approx N\alpha$ for some $0<\alpha<1$. We expect an analogue of Theorem \ref{Theorem_Gcorners_limit_intro} to hold in such setting and present a partial result in this direction in Theorem \ref{Theorem_edge_Airy}.
\end{enumerate}

There is also a universality of a different kind, namely, the Gaussian $\beta$ corners process \eqref{eq:beta-Gauss_corner} and its $\beta=\infty$ counterpart appear as scaling limits in various setups. Let us explain this by starting from the real $\beta=1$ example. Consider a uniformly random point $(v_1,\dots,v_N)$ on the unit sphere $\mathbb S^{N-1}$ in $\mathbb R^N$. A direct computation shows that each individual squared coordinate $v_i^2$ is distributed as Beta random variable $B(\tfrac{1}{2}, \frac{N-1}{2})$, which can be then used to show that $\E(v_i)^2=\tfrac{1}{N}$, $\E(v_i)^4=\tfrac{3}{N(N+2)}$, $\E (v_i)^2 (v_j)^2=\tfrac{1}{N(N+2)}$.
 Now take an $N\times N$ Hermitian matrix $\Lambda$ with deterministic eigenvalues $\lambda_1,\dots,\lambda_N$ and uniformly random eigenvectors. The top-left matrix element $\Lambda_{11}$ can be written as
$$
 \lambda_1 (v_1)^2+ \lambda_2 (v_2)^2+\dots+\lambda_N (v_N)^2,\qquad (v_1,\dots,v_N) - \text{ uniformly random vector on } S^{N-1}.
$$
Computing the mean and variance of $\Lambda_{11}$ using the above moments of $(v_i)^2$ and using additional arguments to show the asymptotic Gaussianity, one proves the distributional convergence
$$
 \Lambda_{11}-\frac{\lambda_1+\dots+\lambda_N}{N}\approx \sqrt{\frac{1}{N+2}\left( \frac{\sum_{i=1}^N (\lambda_i)^2}{N}- \frac{\left(\sum_{i=1}^N \lambda_i\right)^2}{N^2}\right)} \cdot \mathcal N(0,2) ,\qquad N\to\infty.
$$
This result should be treated as convergence of recentered and rescaled $1\times 1$ corner of the matrix to the $1\times 1$ Gaussian Orthogonal Ensemble, whose eigenvalues are given by \eqref{eq:beta-Gauss_corner} with $\beta=1$. The procedure can be generalized in two directions: instead of $1\times 1$ we can consider arbitrary $n\times n$ corners and instead of $\beta=1$ we can consider arbitrary $\beta>0$. The result remains the same: the scaling limit is always given by the Gaussian $\beta$ corners process \eqref{eq:beta-Gauss_corner}, see \cite{Cuenca} and \cite{MM}.

Section \ref{Section_corners_limits} contains a $\beta=\infty$ version of such results. It starts from the observation of \cite{GM2018} that the process formed by eigenvalues of corners of a $N\times N$ Hermitian matrix with fixed spectrum and uniformly random eigenvectors admits a non-degenerate $\beta\to\infty$ scaling limit. This limit is an interesting $N(N-1)/2$--dimensional Gaussian process, whose components are attached to the lattice of all zeros of all derivatives of a degree $N$ real--valued polynomial. The next step is to send $N\to\infty$ and Theorem \ref{Theorem_Gaussian_limit} shows that under very mild restrictions the limit (which is a counterpart of the eigenvalue process for fixed size corners of a large matrix from the previous paragraph) is universally given by the $\beta=\infty$ version of Gaussian $\beta$ corners process \eqref{eq:beta-Gauss_corner}.

\subsection{Our methods} For the proofs we start from the computation of $\beta\to\infty$ fixed $N$ limit in \eqref{eq:beta-Gauss_corner}, following \cite{GM2018}. In the first order, individual eigenvalues at level $k$ converge to the roots of the degree $k$ Hermite polynomial, $\lim_{\beta\to\infty} \chi^k_i=x^k_i$, and we are led to study the fluctuations around these roots:
$$
 \zeta^k_i=\lim_{\beta\to\infty} \sqrt{\beta}(\chi^k_i-x^k_i).
$$
While the $N(N-1)/2$ dimensional process $\{\zeta^k_i\}_{1\le i \le k \le N}$ is Gaussian and has an explicit density (see Section \ref{Section_GinftyE}), computing its $N\to\infty$ limit is very far from being obvious: each coordinate of this process interacts with many others in a non-trivial way.

An important ingredient underlying all our results is identification of $\zeta^k_i$ with a partition function of a directed additive polymer obtained by running a random walk on roots of the Hermite polynomials and collecting white noises along the trajectories. This is a discrete version of the representation \eqref{eq_Z_intro} for $\mathfrak Z(i,t)$. Thus, our asymptotic problems are now reduced to the study of this random walk. In one time step the walker jumps from a root of the degree $k$ Hermite polynomial to a root of the degree $k+1$ Hermite polynomial with probability of a jump from $x$ to $y$ being equal to $\frac{1}{(k+1)(x-y)^2}$.

Our next step is to diagonalize the transition semigroup of the random walk. It turns out that for each $j\le k$ the transition probabilities preserve the space of polynomials of degree $\le j$ and, moreover, are explicitly diagonalized in the basis of certain polynomials $Q^{(k)}_m(z)$, $0\le m <k$. We further give two descriptions of polynomials $Q^{(k)}_m(z)$. On one hand, for fixed $k$, these are the $k$ first monic orthogonal polynomial with respect to the discrete uniform weight on the roots of the degree $k$ Hermite polynomial $H_k(z)$. On the other hand, these are the \emph{associated Hermite polynomials} first studied in \cite{AW}. The three-term recurrence (in $m$) satisfied by these polynomials is the same as the recurrence of the Hermite polynomials, but read in the opposite order.\footnote{In terminology of \cite{Boor-Saff} and \cite{Vinet-Z} $Q^{(k)}_m(z)$ are dual polynomials to $H_k(z)$.}

The formula \eqref{eq_Edge_limit_covariance} eventually arises as a limit of the expression for the covariance of $\zeta^k_i$ through polynomials $Q^{(k)}_m(z)$. In order to compute this limit, we need to compute the asymptotic of polynomials $Q^{(k)}_m(z)$ at the locations of the largest roots of the Hermite polynomials $H_k(z)$. We remark that while the asymptotic behavior of orthogonal polynomials supported on discrete sets has been studied in great detail, one typically assumes that the support of the weight function locally looks like a lattice, see, e.g., \cite{BKMM} for such results. However, in our case the largest roots of the Hermite polynomials approximate zeros of the Airy function, which are very far from forming a lattice. Hence, the type of the asymptotic of $Q^{(k)}_m(z)$ that we develop seems to be new, see Theorem \ref{Theorem_Q_to_Airy} for the exact statement and proof.

For the Dyson Brownian Motion of Theorem \ref{Theorem_DBM_limit_intro} the story is similar: again the polynomials $Q^{(k)}_m(z)$ and their asymptotic behavior play a crucial role.

\smallskip

Let us outline the directions in which our approach might generalize. The representation of the $\beta\to\infty$ limit of the corners process through a random walk collecting noises exists not only for the Gaussian ensemble \eqref{eq:beta-Gauss_corner}, but also for the process formed by the $\beta$ version of the operation of cutting corners from a Hermitian matrix with fixed spectrum and uniformly random eigenvectors discussed in the previous section. However, the general situation is complicated by two features. First, the variance of the noise becomes inhomogeneous. Second, we do not know any reasonable identification for the polynomials diagonalizing the random walk transition matrix, in particular, it is unclear, whether they are orthogonal with respect to some natural weight. On the other hand, since we already know the answers from Theorems \ref{Theorem_Gcorners_limit_intro}, \ref{Theorem_DBM_limit_intro}, it might be possible to show that they remain valid in the such a more general setting by arguing directly and probabilistically in terms of the random walk --- this would be a step toward the universality of the previous section. Simultaneously, we also expect that our representation through the random walk should be helpful in studying other joint limits as $\beta,N\to\infty$, such bulk local limits or global fluctuations of the spectra.

\bigskip

Finally, let us mention two other texts which appeared almost simultaneously with our paper\footnote{The three groups of authors were working independently and without knowing about each other's projects.}. Both texts deal with the Dyson Brownian Motion \eqref{eq_DBM}. \cite{L} proves an existence theorem for the edge limit at finite values of $\beta>1$ (as in Theorem \ref{Theorem_DBM_limit_intro}, but with $\beta$ staying finite) and shows that the limit can be thought of as a solution to an $N=\infty$ version of \eqref{eq_DBM}. The approach of \cite{L} does not give explicit formulas  for the edge limit and it is unclear whether our $\mathfrak Z(i,t)$ can be identified directly by sending $\beta\to\infty$ in the results of \cite{L} --- this is an interesting open question. \cite{AHV} computes the \emph{fixed} time edge limit of the $\beta=\infty$ Dyson Brownian Motion providing a different approach to the asymptotic results of \cite{DE_large_beta, EPS}; in other words, \cite{AHV} covers the intersection of Theorems \ref{Theorem_Gcorners_limit_intro} and \ref{Theorem_DBM_limit_intro} corresponding to the  $t=0$ marginal. The associated Hermite polynomials also appear in \cite{AHV}, but in a different way: in our work they diagonalize transition matrices, while in \cite{AHV} they are eigenfunctions of fixed time covariance matrices. We also remark that \cite[Section 6]{AHV} makes a step in the universality direction of Section \ref{Section_univ} by analyzing the $N,\beta\to\infty$ limits of the Laguerre ensemble which can be obtained from \eqref{eq_GbetaE} by replacing  $\exp\bigl( -\tfrac{\beta}{4} (\chi_i)^2 \bigr)$ with another weight function.

\subsection*{Acknowledgements} We are grateful to Alan Edelman for fruitful discussions about $\beta$--ensembles and to Alexei Zhedanov for bringing \cite{Vinet-Z} to our attention. We thank two anonymous referees for helpful comments.
  The work of V.G.\ was partially supported by NSF Grants DMS-1664619, DMS-1949820,  by BSF grant 2018248, and by the Office of the Vice Chancellor for Research and Graduate Education at the University of Wisconsin--Madison with funding from the Wisconsin Alumni Research Foundation.

The work of V.K. was partially supported by ANR Gromeov (ANR-19-CE40-0007), by Centre Henri Lebesgue (ANR-11-LABX-0020-01), as well as by the Laboratory of Dynamical Systems and Applications NRU HSE, of the Ministry of science and higher education of the RF grant ag.\ No.\ 075-15-2019-1931.

\section{$\beta=\infty$ multilevel ensembles}

\label{Section_inf_ensembles}

The goal of this section is to define the $\beta\to\infty$ fixed $N$ limits of the multidimensional objects of general $\beta$ random matrix theory: $\beta$--corners processes and the Dyson Brownian Motion.

\subsection{$\infty$-corners process} \label{Section_infty_corners}

Take an $N\times N$ random Hermitian matrix with fixed spectrum $x_1^N,\dots,x_N^N$ and uniformly
random eigenvectors\footnote{Equivalently, we deal with the uniform measure on all Hermitian matrices with fixed spectrum $x_1^N,\dots,x_N^N$.}. Let $x_i^k$, $i\le k \le N-1$, be the $i$th eigenvalue of the $k\times k$
top--left corner of this matrix. This procedure can be done for real, complex, or quaternion matrix
elements (corresponding to $\beta=1,2,4$, respectively, see \cite{Neretin} for the modern proof), resulting in the joint laws for the array
$\{\chi_i^k\}_{1\le i\le k\le N-1}$ given by the density with respect to the  Lebesgue measure
\begin{equation}\label{eq:beta-corner}
\frac{1}{Z_{N,\beta}}
  \prod_{k=1}^{N-1} \left[\prod_{1\le i<j\le k} (\chi_j^k-\chi_i^k)^{2-\beta}\right] \cdot \left[\prod_{a=1}^k \prod_{b=1}^{k+1}
 |\chi^k_a-\chi^{k+1}_b|^{\beta/2-1}\right],
\end{equation}
where $Z_{N,\beta}$ is the normalizing constant, and the eigenvalues $\chi_i^k$ satisfy the
deterministic inequalities $\chi_i^{k+1}\le \chi_{i}^k\le \chi_{i+1}^{k+1}$ for all $1\le i \le k \le N-1$.

While our ultimate interest is in $N\to\infty$ asymptotics of \eqref{eq:beta-corner}, it was
noticed in \cite{GM2018} that a simpler object can be obtained if we first send $\beta\to\infty$
while keeping $N$ fixed. Namely, as $\beta\to\infty$, the values $\{\chi_i^k\}$ become deterministic
(``crystallize''), tending to an array $\{x_i^k\}$. The latter can be computed recursively using the
relation $P_{k-1}(x)= \frac{1}{k} P_k'(x)$, where $P_k(x)=\prod_{j=1}^k (x-x_j^k) $ is the
characteristic polynomial for the \emph{limiting} level~$k$ eigenvalues\footnote{Thus, the polynomials $P_k(x)$ form an Appell sequence.}.  Recentering around these
limiting values and renormalizing by $\sqrt{\beta}$ we arrive at the $\infty$--corners
process. This is a Gaussian process
$$\{\xi_i^k\}_{1\le i \le k \le N}=\lim_{\beta\to\infty} \Bigl\{ \sqrt{\beta} (\chi_j^k-x_j^k)\Bigr\}_{1\le i \le k \le N},$$ where
$\xi_1^N=\xi_2^N=\dots=\xi^N_N=0$, and the other coordinates (see~\cite[Eq.~(11)]{GM2018}) have the
common density proportional to
\begin{equation}\label{eq:xi}
 \exp\left(\sum_{k=1}^{N-1} \left[ \sum_{1\le i <j \le k} \frac{ (\xi_i^k-\xi_j^k)^2}{2 (x_i^k-x_j^k)^2}
 - \sum_{a=1}^k \sum_{b=1}^{k+1} \frac{ (\xi_a^k-\xi_b^{k+1})^2 }{4 (x_a^k-x_b^{k+1})^2}
 \right]\right).
\end{equation}

\subsection{Gaussian $\infty$--corners process} \label{Section_GinftyE}

A special role in our exposition is played by the Gaussian $\infty$--corners process\footnote{Note the double meaning of the word Gaussian here. The process is a Gaussian vector and it also arises as a limit of eigenvalues of Gaussian matrices.}, in which
the polynomials $P_k(x)=H_k(x)$ are the Hermite polynomials and the top row
$\xi_1^N,\xi_2^N,\dots,\xi^N_N$ is also random rather than deterministically vanishing. This object
can be obtained as $\beta\to\infty$ limit of the corners process constructed from the Gaussian
$\beta$--ensemble, which is a distribution on arrays $\{\chi_i^k\}_{1\le i\le k\le N}$, obtained from \eqref{eq:beta-corner} by making the top row random and distributed according to the Gaussian $\beta$--ensemble \eqref{eq_GbetaE}. The distribution of the full array $\{\chi_i^k\}_{1\le i\le k\le N}$ was given in \eqref{eq:beta-Gauss_corner}. Recentering $\chi_i^k$ around the zeros of the Hermite polynomials, multiplying by $\sqrt{\beta}$ and sending $\beta\to\infty$ we get the Gaussian $\infty$--corners process. For one level the link to the zeros of the Hermite polynomials is classical, see \cite[Section 6.7]{Szego}, \cite{K-LLN}, while the second order Gaussianity was investigated in \cite{DE_large_beta}. The multilevel result is obtained through a straightforward Taylor expansion of \eqref{eq:beta-Gauss_corner} near its maximum given by the roots of the Hermite polynomials, cf.\  \cite[Theorem 1.6]{GM2018}.

Recasting the result of $\beta\to\infty$ limit transition, we deal with an infinite-dimensional centered Gaussian vector $\zeta_i^j$, $1\le i\le j$,
such that for each fixed $N=1,2,\dots,$ the $N(N+1)$--dimensional marginal $\{\zeta_i^j\}_{1\le i
\le j \le N}$ has density proportional to:

\begin{equation}\label{eq:zeta}
 \exp\left(-\sum_{i=1}^N \frac{(\zeta_i^N)^2}{4}+\sum_{k=1}^{N-1} \left[ \sum_{1\le i <j \le k} \frac{ (\zeta_i^k-\zeta_j^k)^2}{2 (x_i^k-x_j^k)^2}
 - \sum_{a=1}^k \sum_{b=1}^{k+1} \frac{ (\zeta_a^k-\zeta_b^{k+1})^2 }{4 (x_a^k-x_b^{k+1})^2}
 \right]\right),
\end{equation}
where $x_i^k$ is the $i$--th root ($i=1$ means the smallest) of the degree $k$ Hermite polynomial
$H_k$.

\begin{proposition} \label{Proposition_GInftyE}
 The definition \eqref{eq:zeta} is consistent: restricting $\{\zeta_i^j\}_{1\le i \le j \le N}$ to $k(k+1)/2$ coordinates $\{\zeta_i^j\}_{1\le i \le j \le k}$ gives the object
 of the same type. Further, restriction of $\{\zeta_i^j\}_{1\le i \le j \le N}$ onto $N$ particles
 $\zeta_1^N,\zeta_2^N,\dots,\zeta_N^N$ has the density proportional to
 \begin{equation}
\label{eq:zeta_projection}
 \exp\left(-\sum_{i=1}^N \frac{(\zeta_i^N)^2}{4} -\sum_{1\le i <j \le N} \frac{ (\zeta_i^N-\zeta_j^N)^2}{2 (x_i^N-x_j^N)^2}
 \right).
 \end{equation}
\end{proposition}
\begin{proof}
 Following \cite{GM2018}, the formula \eqref{eq:zeta} is obtained as $\beta\to \infty$ limit of the density of the Gaussian $\beta$ corners process of  \cite[Definition 1.1]{GS_DBM}
 at  $t=\tfrac{2}{\beta}$  and the consistency becomes the corollary of the consistency of the latter Definition
 1.1. Similarly, \eqref{eq:zeta_projection} is $\beta\to\infty$ limit of the density of the Gaussian
 $\beta$ ensemble; it is a projection of \eqref{eq:zeta} as $\beta\to\infty$ limit of the fact that the
 Gaussian $\beta$ corners process projects to the Gaussian $\beta$ ensemble, which can be found in \cite[Corollary
 5.4]{GS_DBM}.
\end{proof}

\subsection{Dyson Brownian Motion at $\beta=\infty$} \label{Section_DBM_infty} Recall that the Dyson Brownian Motion (see, e.g., \cite[Chapter 9]{Mehta}, \cite[Section 4.3]{AGZ}) is an $N$--dimensional stochastic process with coordinates
$X_1(t)\le X_2(t)\le \dots \le X_N(t)$, $t\ge 0$, defined as a solution to the system of SDEs
\begin{equation}
\label{eq_DBM_main} \d X_i(t) = \sum_{j\ne i} \frac{\d t}{X_i(t)-X_j(t)} + \sqrt{\frac{2}{\beta}}\, \d W_i(t), \quad i=1,2,\dots,N,\quad t\ge 0,
\end{equation}
where $W_1(t),\dots, W_N(t)$ is a collection of independent standard Brownian motions. The evolution \eqref{eq_DBM_main} should be supplied with initial conditions and in this text we are going to only consider the case $X_1(0)=X_2(0)=\dots=X_N(0)=0$. In this situation the distribution of the solution to \eqref{eq_DBM} at a fixed time $t$ is (a rescaled version of) the Gaussian $\beta$ ensemble of density
\begin{equation} \label{eq_GbetaE_t}
 \prod_{1\le i < j \le N} (\chi_j-\chi_i)^{\beta} \prod_{i=1}^N \exp\bigl( -\tfrac{\beta }{4t} (\chi_i)^2 \bigr).
\end{equation}
Since we are ultimately interested in $\beta\to\infty$ limit, we can assume $\beta\ge 1$; in this situation  \eqref{eq_DBM_main} has a unique strong solution, see \cite[Section 4.3]{AGZ}. Hence, we deal with a pair of $N$--dimensional stochastic processes $( X_i(t); W_i(t))_{i=1}^N$, $t\ge 0$, such that $(W_i(t))_{i=1}^N$ is the standard Brownian motion, for each $t>0$ the law of $(X_i(t))_{i=1}^N$ is given by \eqref{eq_GbetaE_t} (in particular $X_i(0)=0$), and $(X_i(t))_{i=1}^N$ is the unique strong solution to \eqref{eq_DBM} on $t\in [0,+\infty)$ time interval.

\begin{theorem} \label{Theorem_DBM_beta_limit}
 Fix $N$ and let $X_1(t)\le X_2(t)\le \dots \le X_N(t)$ be the solution to \eqref{eq_DBM_main} with $X_1(0)=X_2(0)=\dots=X_N(0)=0$ and let $x_1^N< x_2^N<\dots<x_N^N$ be the roots of the degree $N$ Hermite polynomial. Define
 \begin{equation}
 \label{eq_DBM_large_beta_limit}
  \zeta^N_i(t)=\lim_{\beta\to\infty} \sqrt{\beta}\left(X_i(t)-\sqrt{t}\, x_i^N\right).
 \end{equation}
 Then the $N$--dimensional (Gaussian) vector $(\zeta^N_1(t), \dots, \zeta_N^N(t))$ solves a linear SDE
 \begin{equation}
 \label{eq_DBM_infinity}
  \d \zeta^N_i(t) = - \sum_{j \ne i} \frac{ \zeta^N_i(t)- \zeta^N_j(t)}{t(x_i^N-x_j^N)^2} \d t + \sqrt{2}\, \d W_i(t), \quad t\ge 0,
 \end{equation}
 with initial condition $\zeta^N_1(0)=\dots=\zeta^N_N(0)=0$. The convergence in \eqref{eq_DBM_large_beta_limit} is in law in the space of $N$--dimensional continuous functions on each interval $t\in [t_1,t_2]$ with $0<t_1<t_2$, and joint with the law of $W_i(t)$, $t \ge 0$, $1\le i \le N$ (the latter does not depend on $\beta$).
\end{theorem}

Before coming to the proof of Theorem \ref{Theorem_DBM_beta_limit} let us look at the limiting SDE \eqref{eq_DBM_infinity}.

\begin{lemma} \label{Lemma_zeta_SDE}
 Let $\bigl(W_i(t)\bigr)_{i=1}^N$, $t\ge 0$ be a standard Brownian motion. There exists a unique stochastic process $\bigl(\zeta^N_i(t)\bigr)_{i=1}^N$, $t\ge 0$, such that for each $\eps>0$, $\bigl(\zeta^N_i(t)\bigr)_{i=1}^N$ is a strong solution to \eqref{eq_DBM_infinity} on the interval $t\in[\eps,+\infty)$ and
 $$
  \lim_{t\to 0} \zeta^N_i(t)=0, \quad \text{ in probability for each }i=1,2,\dots,N.
 $$
\end{lemma}

We prove Lemma \ref{Lemma_zeta_SDE} in Section \ref{Section_Inf_DBM_sol}; the solution is expressed there as a sum involving Ito integrals and orthogonal polynomials. This solution is the limiting process in Theorem \ref{Theorem_DBM_beta_limit}.

\smallskip

We expect that convergence in Theorem \ref{Theorem_DBM_beta_limit} can be upgraded to almost sure uniform convergence on each interval $t\in[0,T]$, $T>0$. Such an upgrade would need careful analysis at $t=0$, where both \eqref{eq_DBM_main} and \eqref{eq_DBM_infinity} are singular. Because eventually our interest is in large $t$ (as in Theorem \ref{Theorem_DBM_limit_intro}), we decided not to pursue this analysis here and to phrase Theorem \ref{Theorem_DBM_beta_limit} in the way avoiding $t=0$. A variant of Theorem \ref{Theorem_DBM_beta_limit} for a different initial condition can be found in \cite{VW_functional_CLT}. We also give a proof here in order to be self-contained.


\begin{proof}[Proof of Theorem \ref{Theorem_DBM_beta_limit}]
 We start by computing the first order limit $y_i(t):=\lim_{\beta\to\infty} X_i(t)$. There are several ways to do it. First, looking at \eqref{eq_GbetaE_t} we conclude that $y_1(t)<\dots<y_N(t)$ should solve the variational problem
 \begin{equation}
 \label{eq_Hermite_variational}
   \prod_{1\le i <j \le N} (y_j-y_i) \prod_{i=1}^N \exp\left(-\frac{1}{4t} (y_i)^2 \right) \to \max.
 \end{equation}
 The latter is known to be solved by rescaled zeros of the Hermite polynomials: $y_i(t)=\sqrt{t} x_i^N$. Such a variational characterization of roots
 dates back to the work of T.~Stieltjes, cf.\ \cite[Section 6.7]{Szego}, \cite{K-LLN}. We can also send $\beta\to\infty$ directly in \eqref{eq_DBM_main} concluding that $y_i(t)$ should solve
\begin{equation}
\label{eq_DBM_Hermite} \d y_i(t) = \sum_{j\ne i} \frac{\d t}{y_i(t)-y_j(t)}, \quad i=1,2,\dots,N,\quad t\ge 0; \qquad y_1(0)=\dots=y_N(0)=0.
\end{equation}
 The fact that  $y_i(t)=\sqrt{t} x_i^N$ solve \eqref{eq_DBM_Hermite} would follow once we show that
 \begin{equation}
 \label{eq_Hermite_var_eq}
  \frac{1}{2} x^N_i=\sum_{j\ne i} \frac{1}{x^N_i-x^N_j}, \quad i=1,2,\dots,N.
 \end{equation}
 The latter identity is equivalent to the vanishing of the logarithmic derivatives in each $y_i$ of \eqref{eq_Hermite_variational} at $t=1$ for the maximizing configuration $y_i=x_i^N$.

 \smallskip

 Next, let us compute the centered fixed $t$ limit of $X_i(t)$ as $\beta\to\infty$. For that we Taylor expand the (logarithm of the) density \eqref{eq_GbetaE_t} around the $N$--tuple $(\sqrt{t} x_i^N)_{i=1}^N$.
 In the same way as in Proposition \ref{Proposition_GInftyE}, this results in a limiting relation involving a rescaled version of \eqref{eq:zeta_projection}:
 \begin{equation}
 \label{eq_DBM_one_point_limit}
  \lim_{\beta\to\infty}\sqrt{\beta} \biggl(  X_i(t)- \sqrt{t} x_i^N\biggr)_{i=1}^N \stackrel{d}{=} \bigl(\sqrt{t} u_i\bigr)_{i=1}^N,
 \end{equation}
 where $(u_1,\dots,u_N)$ is a Gaussian vector with density proportional to
 $$
   \exp\left[- \sum_{1\le i<j\le N} \tfrac{1}{2 (x^N_i-x^N_j)^2} \bigl(u_i-u_j\bigr)^2-\sum_{i=1}^N \tfrac{1}{4} \bigl(u_i\bigr)^2    \right].
 $$
 Let us emphasize that \eqref{eq_DBM_one_point_limit} is a distributional limit at a fixed time $t$. In order to deduce the multi-time limit, we further write
 $$
  X_i(t)=\sqrt{t} x_i^N + \frac{1}{\sqrt{\beta}} \eta_i(t)
 $$
 and plug this into \eqref{eq_DBM_main} getting
 \begin{equation}
\label{eq_DBM_main_large_beta} \frac{1}{2\sqrt t} x_i^N \d t+  \frac{1}{\sqrt{\beta}} \d \eta_i(t)  = \sum_{j\ne i} \frac{\d t}{\sqrt{t} x_i^N -\sqrt{t} x_j^N  + \frac{1}{\sqrt{\beta}} \eta_i(t)- \frac{1}{\sqrt{\beta}} \eta_j(t) } + \sqrt{\frac{2}{\beta}}\, \d W_i(t).
\end{equation}
Further, Taylor expanding the $\d t$ term in the right-hand side in small parameter $\frac{1}{\sqrt{\beta}}$  we get
 \begin{multline*}
\label{eq_DBM_main_large_beta_2} \frac{1}{2\sqrt t} x_i^N \d t+  \frac{1}{\sqrt{\beta}} \d \eta_i(t)  \\= \sum_{j\ne i} \frac{\d t}{\sqrt{t} x_i^N -\sqrt{t} x_i^N } + \frac{1}{\sqrt{\beta}} \sum_{j\ne i} \frac{\d t(\eta_j(t)-\eta_i(t))}{t( x_i^N -x_j^N)^2}  + \sqrt{\frac{2}{\beta}}\, \d W_i(t) + O\left(\frac{1}{\beta}\right).
\end{multline*}
Using \eqref{eq_Hermite_var_eq} to cancel the first terms in the right-hand and left-hand sides, multiplying by $\sqrt{\beta}$, and sending $\beta\to\infty$ we get \eqref{eq_DBM_infinity}.

Now choose $\eps>0$. For $t\ge \eps$, the $\beta\to\infty$ convergence of the SDE that $\eta_i(t)$ satisfies towards \eqref{eq_DBM_infinity}, together with \eqref{eq_DBM_one_point_limit}, implies that $\bigl(\zeta^N_i(t)\bigr)_{i=1}^N=\lim_{\beta\to\infty} \bigl(\eta_i(t)\bigr)_{i=1}^N$,  is the solution of \eqref{eq_DBM_infinity} on time interval $t\in [\eps,+\infty)$ with initial condition given by $(\sqrt{\eps} u_i)_{i=1}^N$, cf.\ \cite[Proof of Theorem 2.2]{VW_functional_CLT} for some details. Note that such solution is unique by general theorems on SDEs with Lipshitz coefficients (see, e.g., \cite[Theorem 21.3]{Kal}).

Clearly, the initial condition $\zeta^N_i(\eps)\stackrel{d}{=}\sqrt{\eps} u_i$ for each $i$ converges to $0$ as $\eps\to 0$ in distribution and, hence, also in probability. We conclude that the limiting process $(\zeta^N_i(t))_{i=1}^N$, $t\ge 0$, is the object of Lemma \ref{Lemma_zeta_SDE}.
\end{proof}

\subsection{Asymptotic results for the corners processes}
\label{Section_corners_limits}

We presented the $N\to\infty$ asymptotic results about the Gaussian $\infty$--corners process of Section \ref{Section_GinftyE} and the $\beta=\infty$ Dyson Brownian Motion of Section \ref{Section_DBM_infty} in Theorems \ref{Theorem_Gcorners_limit_intro} and \ref{Theorem_DBM_limit_intro}, respectively. In this section we present several $N\to\infty$ asymptotic results dealing with the $\beta=\infty$--corners process $\{\xi^k_i\}$ of Section \ref{Section_infty_corners}.

The definition of the process $\xi^k_i$ relies on the (deterministic) configuration of points
$x^k_i$. Recall that we start from an
$N$--tuple $y_1\le y_2\le \dots\le y_N$, and define the monic polynomials:
\begin{equation}
 P_N(x)=\prod_{i=1}^N (x-y_i),\qquad P_k(x)=\frac{1}{N(N-1)\dots (N-k+1)} \left(\frac{\partial}{\partial
 x}\right)^{N-k} P_N(x). \label{eq_Appell_sequence}
\end{equation}
The points $x_1^k \le x_2^k \le \dots x_k^k$ are defined as $k$ (real) roots of $P_k(x)$. We study the points $x_i^k$ in three scaling regimes, which are schematically shown in Figure \ref{Fig_local_limits}

 \begin{figure}[t]
\begin{center}
\includegraphics[width=0.8\linewidth]{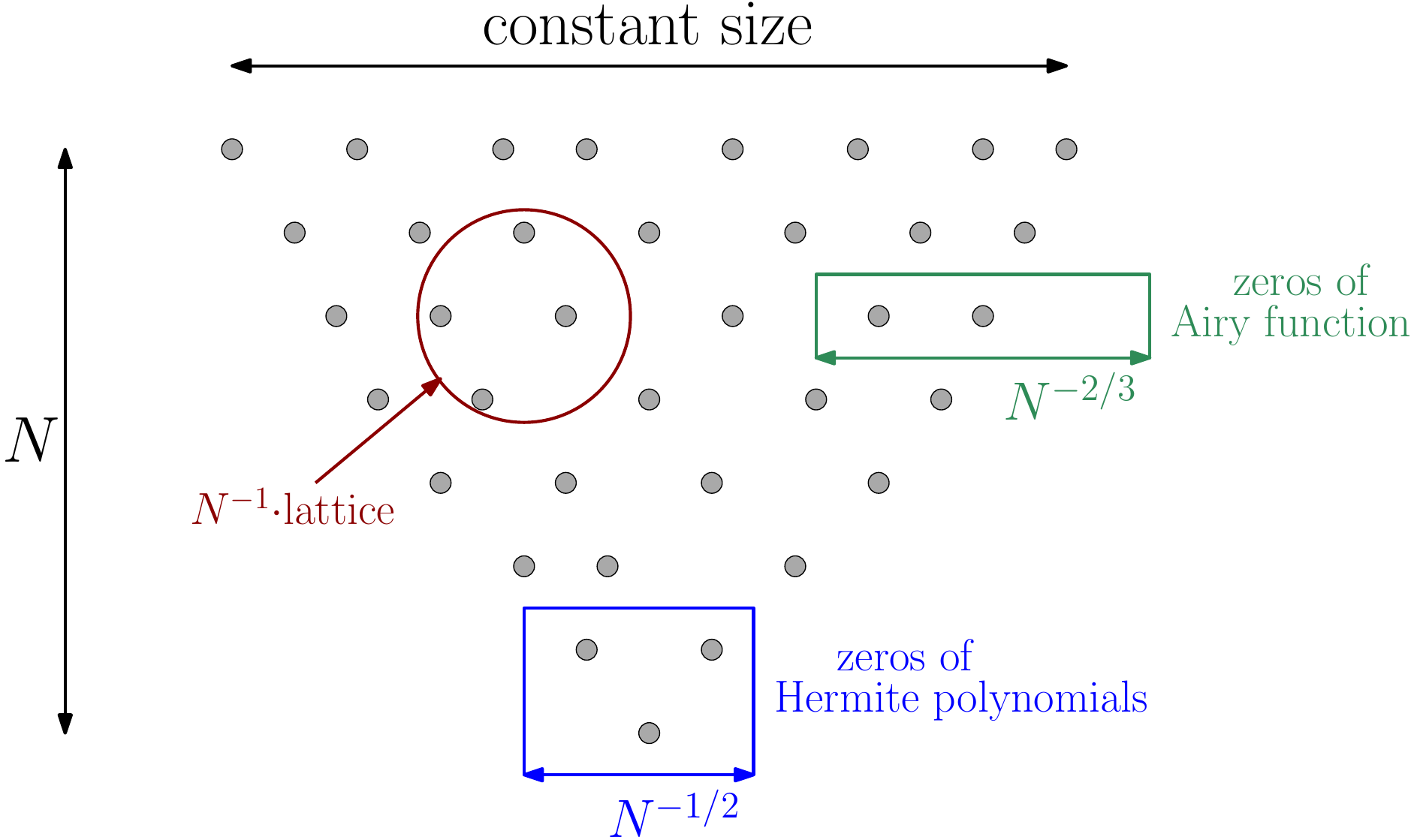}
 \caption{Three scaling regimes and limiting objects for the grid formed by zeros of derivatives of $P_N(x)$.\label{Fig_local_limits}}
\end{center}
\end{figure}

For $N$--tuples $y_1\le y_2\le \dots\le y_N$ (with each $y_i=y_i(N)$ depending on $N$, although we
omit this dependence from the notations) we introduce various quantities describing it:
\begin{itemize}
\item (Centered) moments:
$$
 \mu_N=\frac{1}{N}\sum_{i=1}^N y_i, \quad (\sigma_N)^2 =\frac{1}{N}\sum_{i=1}^{N} (y_i-\mu_N)^2,\quad
 (\kappa_N)^3=\frac{1}{N}\sum_{i=1}^N |y_i-\mu_N|^3.
$$
\item Empirical measures:
$$
 \rho_N=\frac{1}{N}\sum_{i=1}^N \delta_{y_i}.
$$
\end{itemize}

We would like to have asymptotic control on $y_i$ and for different applications we use different
topologies summarized in the following three assumptions:

\begin{assumption}\label{Assumption_Hermite}
 We have
\begin{equation}
 \lim_{N\to\infty} \frac{\kappa_N}{\sigma_N}N^{-1/6} =0.
\end{equation}
\end{assumption}
\begin{remark}
 A typical situation is that both $\sigma_N$ and $\kappa_N$ stay bounded away from $0$ and $\infty$, in which case the assumption holds automatically.
\end{remark}

\begin{assumption}\label{Assumption_weak}
 As $N\to\infty$ the measures $\rho_N$ weakly converge to a compactly supported probability measure $\rho$.
\end{assumption}

\begin{assumption}\label{Assumption_strong}
 As $N\to\infty$:
  \begin{enumerate}
 \item
  The measures $\rho_N$ weakly converge to a compactly supported probability measure
 $\rho$;
 \item The supremum of the support of $\rho$ is $B$ and
 $
  \lim_{N\to\infty} y_N=B;
 $
  \item For a constant $\vartheta>0$, which does not depend on $N$, we have $y_{N+1-i}-y_{N-i}> \vartheta/N$ for all $1\le i \le \vartheta N$;
  \item $\rho$ has a density $\rho(x)$ on $[B-\vartheta,B]$, which satisfies $\rho(x)\ge \vartheta (B-x)$ on this segment.
  \end{enumerate}
 \end{assumption}
 \begin{remark}
  The conditions in Assumption \ref{Assumption_strong} are tuned so that to guarantee the convergence in Theorem \ref{Theorem_edge_Airy} of the largest points $x_{k+1-i}^k$ to the roots of the Airy function for all the range of ratios $0<\tfrac{k}{N}<1$; these conditions will be used in Lemma \ref{Lemma_critical_point_edge}. If we only aim at small values of the ratio $\tfrac{k}{N}$, then the conditions can be significantly weakened: small $k$ has a smoothing role, which leads automatically to the necessary edge behavior.

  If we are interested in the smallest points  $x_i^k$ (rather than the largest), then we need to use similar conditions with $N+1-i$ indices replaced by $i$ and with supremum of the support $B$ replaced by the infinum $A$.
 \end{remark}

The first two results of this section explain the prominent role of the Gaussian $\infty$--corners process as a scaling limit.

\begin{theorem}\label{Theorem_Hermite_limit}
 Let $\{x_i^k\}_{1\le i \le k}$ be the roots of $P_k(x)$ as in \eqref{eq_Appell_sequence}. Under Assumption \ref{Assumption_Hermite} for each fixed $1\le i \le k$
 $$
  \lim_{N\to\infty} \frac{\sqrt{N}}{\sigma_N} \bigl(x_i^k-\mu_N\bigr)=h_i^k,
 $$
 where $h_1^k,h_2^k,\dots, h_k^k$ are $k$ roots of the degree $k$ Hermite polynomial $H_k(x)$.
\end{theorem}
\begin{remark}
 For a particular case when $x_i^N$, $i=1,2,\dots,N$, are i.i.d.\ random variables, a result similar to Theorem \ref{Theorem_Hermite_limit} can be found in \cite{HS}.
\end{remark}
\begin{example}
 Suppose that $N$ is even, $N=2M$, and
 $$
  P_N(x)=P_{2M}(x)=(x+1)^M (x-1)^M=x^{2M}-M x^{2M-2}+ \frac{M(M-1)}{2} x^{2M-4}+\dots
 $$
  In this situation $\mu_N=0$, $\sigma_N^2=1$, and $\kappa_N^3=1$. Hence, Theorem \ref{Theorem_Hermite_limit} applies. Let us check its conclusion directly for $k=3$. Indeed,
  $$
   P_3(x)=\frac{1}{2M(2M-1)\cdots 4} \frac{\partial^{2M-3}}{\partial x^{2M-3}} P_{2M}(x)= x^3- \frac{6 M}{2M (2M-1)} x=x^3-\frac{3}{2M-1} x.
  $$
  We see that as $M\to\infty$
  \begin{equation}
  \label{eq_x37}
   (2M)^{3/2}  P_3\left(\frac{x}{\sqrt{2M}}\right)\to x^3-3x.
  \end{equation}
  Because $x^3-3x$ is the degree three Hermite polynomial, \eqref{eq_x37} agrees with Theorem \ref{Theorem_Hermite_limit}.
\end{example}

\begin{theorem} \label{Theorem_Gaussian_limit} For each $N=1,2,\dots,$ take
$N$--tuple of reals $y_1\le y_2\le \dots\le y_N$ and let $\{\xi_i^k(N)\}_{1\le i \le k \le N}$ be a Gaussian vector distributed as the
$\infty$--corners process \eqref{eq:xi} with top level $x_i^N=y_i$, $i=1,\dots,N$.
Under Assumption \ref{Assumption_Hermite} for each fixed $K=1,2,\dots$, we have convergence in distribution
$$
 \lim_{N\to\infty} \frac{\sqrt{N}}{\sigma_N}\bigl\{ \xi_i^k(N) \bigr\}_{1\le i \le k \le K}=\bigl\{\zeta_{i}^k\bigr\}_{1\le i \le k \le
 K},
$$
where $\{\zeta_i^k\}$ is the Gaussian $\infty$--corners process of Section \ref{Section_GinftyE}.
\end{theorem}

For the next results, we need to introduce an equation on an unknown variable $z$, with parameters
$1\le k \le N-1$ and $x\in \mathbb R$
\begin{equation}
\label{eq:critical_equation}
 \frac{1}{N}\cdot \frac{P_N'(z+x)}{P_N(z+x)}=\frac{N-k+1}{N} \cdot \frac{1}{z}, \quad z\in\mathbb C.
\end{equation}
In our approach this equation arises as a critical point condition $G'(z)=0$ with
\begin{equation}
\label{eq_G_definition}
  G(z):=\frac{1}{N}\ln\bigl( P_N(z+x)\bigr) - \frac{N-k+1}{N} \ln z.
\end{equation}
\begin{lemma}
\label{Lemma_complex_roots}
 Either all roots of \eqref{eq:critical_equation} are real, or it has a unique pair of complex
 conjugate roots.
\end{lemma}
\begin{proof}
  Let us first assume that all $y_i$ are distinct.
  After clearing the denominators, \eqref{eq:critical_equation} is a polynomial equation of degree
  $N$. Hence, it has at most $N$ roots. On the other hand, \eqref{eq:critical_equation} can be rewritten as
  \begin{equation}
  \label{eq_x6}
   \frac{1}{N} \sum_{i=1}^N \frac{1}{z-(y_i-x)}- \frac{N-k+1}{N}\cdot \frac{1}{z}=0.
  \end{equation}
 Let us look at $N-1$ segments $(y_i-x,y_{i+1}-x)$, $1\le i \le N-1$ on the real axis. The point $0$ belongs to at most one of them. For the remaining $N-2$ segments, the function in the left-hand side of \eqref{eq_x6} is continuous and changes its sign from positive at $z=y_i-x+0$ to negative at $z=y_{i+1}-x-0$. Therefore, each such segment has a root of \eqref{eq:critical_equation} and we found $N-2$ real roots. Hence, there are at most two complex
  roots.

  For the case when some $y_i$ are allowed to coincide, the argument remains the same with the only difference being that the polynomial equation now has degree ``number of distinct values of $y_i$'' rather than $N$.
\end{proof}

Whenever \eqref{eq:critical_equation} has two complex roots, we say that $(x, \frac{k}{N})$ belongs
to the \emph{liquid region} (sometimes also called the \emph{band}) and denote through $z_c$ the
corresponding root in the upper half-plane. Otherwise, we say that $(x, \frac{k}{N})$ belongs to
the void region.

 \begin{figure}[t]
\begin{center}
{\scalebox{0.8}{\includegraphics{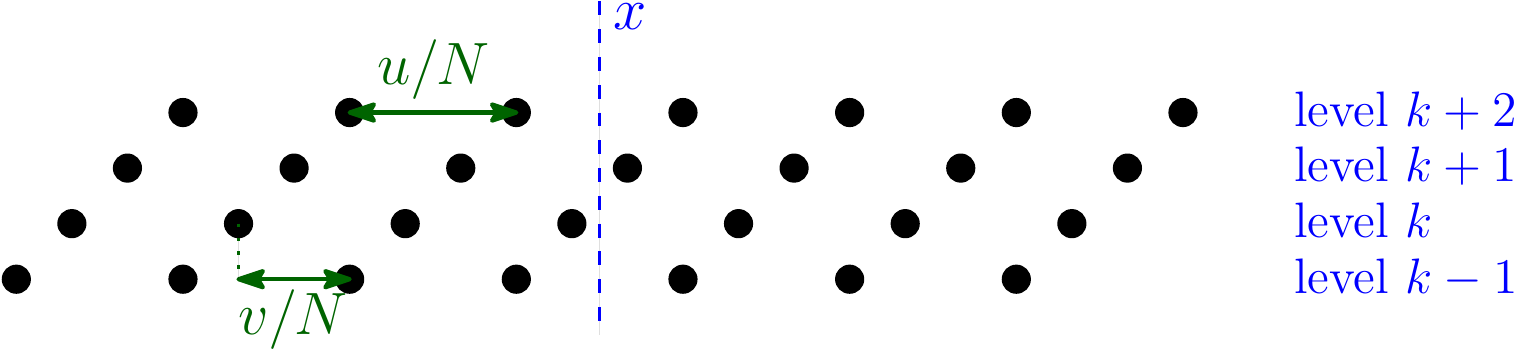}}}
 \caption{Particles near a point $x$ in the bulk resemble a lattice with spacings proportional to $\tfrac{1}{N}$. \label{Fig_lattice}}
\end{center}
\end{figure}

\begin{theorem} \label{Theorem_bulk_lattice}
 Under Assumption \ref{Assumption_weak} choose $(x,\frac{k}{N})$ in the liquid region in such a way that as
 $N\to\infty$, $\frac{k}{N}$ is bounded away from $0$ and $1$ and $z_c$ stays bounded away from the
 real axis and from $\infty$. Then, zooming in near $x$, the point configurations
 $$
  \{N (x_i^k -x) \}
 $$
 asymptotically form a lattice (cf.\ Figure \ref{Fig_lattice}) with fixed spacing $u=\lim_{N\to\infty} (x_{i+1}^k-x_i^k)$ and fixed spacing
 $v=\lim_{N\to\infty}(x_i^k-x_{i}^{k+1})$ (satisfying $0<v<u$), such that
$$
 u=\pi \left(\frac{N-k+1}{N} \Im \frac{1}{z_c} \right)^{-1}, \qquad v=u \cdot \frac{1}{\pi} \arg (z_c).
$$
\end{theorem}
\begin{remark}
Results of this type are known in the literature, see, e.g., \cite{FR}.
\end{remark}
\begin{remark}
 When all the roots of \eqref{eq:critical_equation} are real, we expect  to observe no points from $\{x^k_i\}$ near $(x,\frac{k}{n})$, hence, the name ``void region''. We do not prove such a statement here, but it probably can be proven by the same methods which we use in the Appendix.
\end{remark}

Looking carefully into the argument of Lemma \ref{Lemma_complex_roots}, one can notice that for a very large (positive or negative) $x$ all roots of \eqref{eq:critical_equation} are real and such $x$ belongs to the void region. If we start decreasing $x$ from $+\infty$, then at some point we eventually reach the liquid region. This transition point is the right \emph{edge} of the liquid region. Note that at this point the complex conjugate roots $z_c$ and $\overline z_c$ merge together, forming a double root of \eqref{eq:critical_equation}.

Let $\a_1>\a_2>\a_3>\dots$ be the zeros of the Airy function $\Ai(x)$.
\begin{theorem} \label{Theorem_edge_Airy}
 Under Assumption \ref{Assumption_strong}, as $N\to\infty$ and with $k$ varying in such a way that for $k/N$ stays bounded away from $0$ and from $1$,   let $x=x(N,k)$ be the largest real number, such that \eqref{eq:critical_equation} has a double root, and let $z_c\in\mathbb R$ denote the location of this root.
 Then for each $i=1,2,\dots$,
 $$
  \lim_{N\to\infty} N^{2/3} \, \frac{x_{k+1-i}^k-x}{\sigma}= \a_i,
 $$
 where $\a_i$ is the $i$th largest zero of the Airy function and, using $G(z)$ given by \eqref{eq_G_definition}, we have
 $$
   \sigma=z_c^2 \left(\frac{G'''(z_c)}{2}\right)^{1/3} \frac{N}{N-k+1}.
 $$
\end{theorem}
\begin{remark}
  A very similar statement holds for the smallest points $x_i^k$, $i=1,2,\dots$, with the difference being that $x$ is replaced by the smallest real number for which \eqref{eq:critical_equation} has a double root. Note that $G'''(z_c)>0$ when we deal with the largest points $x_{k+1-i}^k$ and $G'''(z_c)<0$ when we deal with the smallest points $x_i^k$.
\end{remark}

\begin{remark}
 One can expect that in the setting of Theorem \ref{Theorem_edge_Airy}, the two-dimensional process $(i,t)\to c_2 N^{2/3}\xi^{k+c_1 N^{2/3} t}_{k+c_1 N^{2/3} t-i}$ converges to $\mathfrak Z(i,t)$ after proper choice of the deterministic constants $c_1,c_2>0$. This should be viewed as a (conjectural) extension of Theorem \ref{Theorem_Gcorners_limit_intro}.
\end{remark}

Proofs of Theorems \ref{Theorem_Hermite_limit}, \ref{Theorem_bulk_lattice}, \ref{Theorem_edge_Airy} are based on the steepest descent analysis of the contour integrals, they are given in Appendix (Section \ref{Section_steepest_descent}). Proof of Theorem \ref{Theorem_Gaussian_limit} is in Section \ref{Section_Gaussian_limit_proof}.


\section{Innovations and the jumping process}

\label{Section_innovations}

%
%

Our approach to the asymptotic theorems for $\{\xi_i^k\}$ and $\{\zeta_i^k\}$ is based on their representations as
partition functions of directed polymers (with heavy-tailed jumps) collecting \emph{additive}
independent Gaussian noises. In this section we introduce such representations.

\smallskip

As before, we start from a collection $\{x_i^k\}$ of roots of an Appell sequence of polynomials \eqref{eq_Appell_sequence}. We define a collection of numbers $\alpha_{a,b}^k$,  through
\begin{equation}
\label{eq:transition_coeffs_def} \alpha_{a,b}^k =
\dfrac{(x_a^k-x_b^{k+1})^{-2}}{\sum\limits_{b'=1}^{k+1} (x_a^k-x_{b'}^{k+1})^{-2}}, \qquad 1\le a \le k,\quad  1\le b \le k+1.
\end{equation}
The definition readily implies that $\alpha^k_{a,b}$ form a stochastic  matrix:
\begin{equation}\label{eq:stoch}
\forall a, b \quad \alpha_{a,b}^k>0, \quad \text{and} \quad \forall a \quad  \sum_{b=1}^{k+1}
\alpha_{a,b}^k =1.
\end{equation}
We also define a linear operator $A_k$ with matrix $(\alpha_{a,b}^k)_{a=1,\dots,k,\, b=1,\dots,k+1}$: it maps $(k+1)$--dimensional space to $k$--dimensional space.
\begin{remark} $A_k$ can be interpreted as the differential of the $k$--dimensional vector of  roots of the
derivative $P_{k+1}'$ as a function of $k+1$ roots of $P_{k+1}$. In this
interpretation, the identity $\sum_{b=1}^{k+1} \alpha_{a,b}^k =1$ becomes a corollary of an observation that shifting all the roots of a polynomial by a constant~$\eps$ we also shift
every root of its derivative by the same constant~$\eps$.
\end{remark}

\begin{definition}
The \emph{jumping process} is a Markov process with the set of allowed states
$\X_k:=\{x_a^k\}_{a=1,\dots,k}$ at the time $k$, and with the transition probabilities given by
\eqref{eq:transition_coeffs}
$$
\P(x_a^k \to x_b^{k+1}) = \alpha_{a,b}^k.
$$
\end{definition}
The product of matrices in~\eqref{eq:vxi} then becomes its diffusion kernel:
\begin{definition}
\label{Definition_dif_kernel}
The \emph{diffusion kernel} $K^{k,\ell}(a\to b)$ is defined as the (transition) probability that the
jumping process, starting at $x_a^k$ at time $k$, at time $\ell>k$ ends up at~$x_b^\ell$. Formally,
$$K^{k,\ell}(a\to b)=(A_k\dots A_{\ell-1})_{a,b}.$$
\end{definition}

\begin{theorem} \label{Theorem_inftyCorners_RW}
The process $\{\xi_i^k\}_{1\le i\le k\le N}$ of Section \ref{Section_infty_corners} can be represented as
\begin{equation}\label{eq:xi-K}
\xi_a^k = \sum_{\ell=k}^{N-1}  \sum_{b=1}^{\ell} K^{k,\ell}(a\to b)  \cdot \eta^\ell_b,
\end{equation}
where $\eta_b^\ell$ are independent Gaussian random variables with the variance
\begin{equation}\label{eq:var-eta}
\Var \eta_b^\ell = \frac{2}{\sum_{b=1}^{k+1} (x_b^\ell-x_{c}^{\ell+1})^{-2}} = -2\,
\frac{P_{\ell+1}(x_b^\ell)}{P_{\ell+1}''(x_b^\ell)}.
\end{equation}
We also have
\begin{equation}\label{eq:cov-sum}
\Cov (\xi_{a_1}^{k_1}, \xi_{a_2}^{k_2}) = \sum_{\ell=\max(k_1,k_2)}^{N-1} \left( \sum_{b=1}^{\ell}
K^{k_1,l}(a_1\to b) K^{k_2,\ell}(a_2\to b) \cdot \Var\eta^\ell_b \right).
\end{equation}
\end{theorem}

\begin{theorem} \label{Theorem_GinftyE_RW} The process $\{\zeta_i^k\}_{1\le i\le k}$ of Section \ref{Section_GinftyE} can be represented as
\begin{equation}\label{eq:zeta-K}
\zeta_a^k = \sum_{\ell=k}^{\infty}  \sum_{b=1}^{\ell} K^{k,\ell}(a\to b)  \cdot \eta^\ell_b,
\end{equation}
where $\eta_b^\ell$ are independent Gaussian random variables with the variance
\begin{equation}\label{eq:var-eta_G}
\Var \eta_b^\ell =\frac{2}{\ell+1}.
\end{equation}
We also have
\begin{equation}\label{eq:cov-sum_G}
\Cov (\zeta_{a_1}^{k_1}, \xi_{a_2}^{k_2}) = \sum_{\ell=\max(k_1,k_2)}^{\infty} \left( \sum_{b=1}^{\ell}
\frac{2}{\ell+1} K^{k_1,l}(a_1\to b) K^{k_2,\ell}(a_2\to b) \right).
\end{equation}
\end{theorem}
\begin{remark}
 Let us emphasize that $K^{k,\ell}(a\to b)$ depends on the array $\{x_i^j\}$. In particular, in Theorem \ref{Theorem_GinftyE_RW} the diffusion kernel is constructed using roots of the Hermite polynomials, while in Theorem \ref{Theorem_inftyCorners_RW} more general configurations are allowed.
\end{remark}
In words, Theorems \ref{Theorem_inftyCorners_RW} and \ref{Theorem_GinftyE_RW} say that $\{\xi_i^k\}$ and $\{\zeta_i^k\}$ are averages over the trajectories of the jumping process of the sums of independent Gaussian noises collected by this process. In the rest of the section we prove these theorems.

\smallskip

Consider the process $\{\xi_i^k\}_{1\le i\le k\le N}$ of Section \ref{Section_infty_corners}  as a vector-valued process $\{\vxi_k\}_{k=1,\dots,N}$, where $\vxi_k=(\xi_1^k,\dots,\xi_k^k)$.
It is immediate to see from~\eqref{eq:xi} that this process is Markovian: conditionally on any $\vxi_{k_0}$, the values of $\xi_k$ with $k<k_0$ are independent from those with $k>k_0$.

Now, let us compute the conditional distribution of $\vxi_k$ given $\vxi_{k+1}$. One way to do it
is by sending $\beta\to\infty$ in the similar finite $\beta$ conditional distribution, computed in
\cite[(1.6)]{GS_DBM} or \cite[(56)]{GM2018}. The computations result in the density of the
conditional distribution of $\vxi_{k}$ given $\vxi_{k+1}$ being  proportional to
\begin{equation}\label{eq:conditional}
 \exp\left( - \sum_{a=1}^k \sum_{b=1}^{k+1} \frac{ (\xi_a^k-\xi_b^{k+1})^2 }{4 (x_a^k-x_b^{k+1})^2}
\right) = \prod_{a=1}^k  \exp\left( - \sum_{b=1}^{k+1} \frac{ (\xi_a^k-\xi_b^{k+1})^2 }{4 (x_a^k-x_b^{k+1})^2}
\right).
\end{equation}
Completing the squares in the last formula, we rewrite it as
\begin{equation}\label{eq:conditional_complete_square}
  C\cdot \prod_{a=1}^k  \exp\left( - \frac{1}{4}\left[ \sum_{b=1}^{k+1} \frac{1}{ (x_a^k-x_b^{k+1})^2}\right] \left(\xi_a^k-\sum_{b=1}^{k+1} \xi_b^{k+1}\dfrac{(x_a^k-x_b^{k+1})^{-2}}{\sum\limits_{b'=1}^{k+1} (x_a^k-x_{b'}^{k+1})^{-2}} \right)^2 \right),
\end{equation}
where $C$ is a constant which does not depend on $\xi_1^k,\xi_2^k,\dots,\xi_k^k$.
The conditional expectation $\E(\vxi_k\mid \vxi_{k+1})$, thus, can be written as
\begin{equation}
\label{eq:transition_coeffs} \E(\xi_a^k\mid \vxi_{k+1}) = \sum_{b=1}^{k+1} \alpha_{a,b}^k
\xi^{k+1}_b, \, \text{ where } \, \alpha_{a,b}^k =
\dfrac{(x_a^k-x_b^{k+1})^{-2}}{\sum\limits_{b'=1}^{k+1} (x_a^k-x_{b'}^{k+1})^{-2}}.
\end{equation}

We write $\vxi_k$ as a sum of this conditional expectation and of the \emph{innovations vector}
$\veta_k=\vxi_k-\E (\vxi_k \mid \vxi_{k+1} )$. From~\eqref{eq:conditional_complete_square} we see that~$\veta_k$
has independent components with the variance
\begin{equation}\label{eq:var-eta_second}
\Var \eta_a^k = \frac{2}{\sum_{b=1}^{k+1} (x_a^k-x_{b}^{k+1})^{-2}} = -2\,
\frac{P_{k+1}(x_a^k)}{P_{k+1}''(x_a^k)},
\end{equation}
where the second equality comes from differentiating the relation $\frac{P_{k+1}'(y)}{P_{k+1}(y)}=
\sum\limits_{b=1}^{k+1} \frac{1}{y-x_b^{k+1}}$, substituting $y=x_a^k$, and using $P'_{k+1}(x^k_a)=0$.

Now, let us iterate the representation
$$
\vxi_k = A_k \vxi_{k+1} + \veta_k,
$$
going from an arbitrary level $k$ all the way to the top level $N$. Since $\eta^N_i=0$, $1\le i \le
N$, we get
\begin{equation}\label{eq:vxi}
\vxi_k= \veta_k + A_k \veta_{k+1} + A_k A_{k+1} \veta_{k+2} + \dots + A_k A_{k+1} \dots A_{N-2} \veta_{N-1},
\end{equation}
which is precisely \eqref{eq:xi-K}. The identity \eqref{eq:cov-sum} directly follows from \eqref{eq:xi-K} and independence of $\eta_a^k$, thus, finishing the proof of Theorem \ref{Theorem_inftyCorners_RW}.

\smallskip

Let us now develop a similar representation for the Gaussian $\infty$--corners process
$\zeta_i^k$ of \eqref{eq:zeta}. In this particular case $P_k(X)=H_k(x)$ are the Hermite polynomials and
they satisfy the differential equation
\begin{equation}
\label{eq:He}
H_k''(x) - x H_k'(x) + k H_k (x)=0.
\end{equation}
Thus, at every root $y$ of $H_k = \frac{1}{k+1} H_{k+1} '$ one has
$\frac{H_{k+1}(y)}{H''_{k+1}(y)}=-\frac{1}{k+1}$. Hence, $\Var \eta^k_b=\frac{2}{k+1}$ for all~$b$.

Another distinction is that $\zeta_i^N$ no longer vanishes and \eqref{eq:xi-K} gets modified to:
\begin{equation}\label{eq:zeta-K_2}
\zeta_a^k = \sum_{\ell=k}^{N-1} \sum_{b=1}^{\ell} K^{k,\ell}(a\to b)  \cdot \eta^\ell_b+ \sum_{b=1}^N
K^{k,N}(a\to b) \zeta^N_b
\end{equation}
Since $N>k$ is arbitrary in \eqref{eq:zeta-K_2}, we can also take $N=\infty$, getting
\begin{equation}\label{eq:zeta-K_infty}
\zeta_a^k = \sum_{\ell=k}^{\infty} \sum_{b=1}^{\ell} K^{k,\ell}(a\to b)  \cdot \eta^\ell_b,
\end{equation}
which is the same as \eqref{eq:zeta-K}.

\begin{remark}  The series \eqref{eq:zeta-K_infty} is almost surely convergent, as follows (by the Kolmogorov's three series theorem, see, e.g., \cite[Theorem 2.5.8]{Durrett}) from the independence of the terms $\eta^\ell_n$ and convergence of the series defining the variance of $\zeta_a^k$, i.e.\
$$
 \sum_{\ell=k}^{\infty} \sum_{b=1}^{\ell} \left(K^{k,\ell}(a\to b)\right)^2  \cdot \frac{2}{\ell+1} <\infty.
$$
The last inequality is implied by the upper bound $K^{k,\ell}(a\to b)\le \tfrac{k}{\ell}$ of Lemma \ref{Lemma_kernel_upper_bound}.
\end{remark}

\begin{remark}
 For the transition from \eqref{eq:zeta-K_2} to \eqref{eq:zeta-K_infty}, one should additionally check that
 \begin{equation}
 \label{eq_remainder_zero}
  \lim_{N\to\infty} \sum_{b=1}^N
K^{k,N}(a\to b) \zeta^N_b =0, \qquad \text{in probability.}
 \end{equation}
For that, let us note that, by construction,  vectors $( \zeta^N_b)_{b=1}^N$ and $(\eta^\ell_b)_{1\le b \le \ell <N}$ in \eqref{eq:zeta-K_2} are uncorrelated with each other. Hence,
\begin{equation}
 \Var(\zeta_a^k)= \Var\left(\sum_{\ell=k}^{N-1} \sum_{b=1}^{\ell} K^{k,\ell}(a\to b)  \cdot \eta^\ell_b\right)+ \Var\left( \sum_{b=1}^N
K^{k,N}(a\to b) \zeta^N_b\right).
\end{equation}
 Sending $N\to\infty$ in the last identity, \eqref{eq_remainder_zero} would follow if we manage to prove that
\begin{equation}
\label{eq_x38}
 \Var(\zeta_a^k)= \Var\left(\sum_{\ell=k}^{\infty} \sum_{b=1}^{\ell} K^{k,\ell}(a\to b)  \cdot \eta^\ell_b\right).
\end{equation}
 This identity will be established in Corollary \ref{Corollary_two_forms_of_covariance} by relying on the representation of  $K^{k,\ell}(a\to b)$ in terms of orthogonal polynomials.\footnote{Before reaching Corollary \ref{Corollary_two_forms_of_covariance}, the reader might assume that we deal with the process of \eqref{eq:zeta-K} whenever we mention $\zeta_a^k$.}
\end{remark}

\smallskip

Using independence of $\eta_a^k$, the representation $\eqref{eq:zeta-K}$ implies \eqref{eq:cov-sum_G}. The proof of Theorem \ref{Theorem_GinftyE_RW} is finished.


\section{Random walks through orthogonal polynomials}

\label{Section_orthogonal}

The aim of this section is to diagonalize the stochastic matrices $A_k$ from \eqref{eq:stoch} using a special class of orthogonal polynomials.

\subsection{Preservation of polynomials} Let us choose a sequence of polynomials $P_k(x)$, such that $P_k$ is a monic polynomial of degree $k$ and $P_{k-1}(x)=\frac{1}{k} P_{k}'(x)$ for each $k=1,2,\dots$. Each polynomial $P_k$ is further assumed to have $k$ distinct real roots, which constitute the set $\X_k$.

\begin{definition}
 $\mF_k$ is the $k$--dimensional space of functions on $\X_k$.
\end{definition}

We further define $D_k$ to be the dual operator to $A_k$:
\begin{definition}
 The operator $D_k$ maps $\mF_k$ to $\mF_{k+1}$ through:
 $$
  [D_k f](x)=\sum_{y\in \X_k}\left[ \frac{f(y)}{(x-y)^2} \Biggl(\sum_{x'\in \X_{k+1}} \frac{1}{(x'-y)^2}\Biggr)^{-1}\right], \qquad x\in \X_{k+1}.
 $$
\end{definition}

We are going to mostly concentrate on the action of $D_k$ on polynomial functions. It is important to note that since $\mF_k$ is finite-dimensional, monomials $x^n$, $n=0,1,2,\dots$ are linearly dependent. Hence, there can be several representations of $D_k$, whose equivalence is sometimes non-evident.
\begin{proposition}
\label{Proposition_polynomial_preservation}
 For each $m=0,1,\dots,k-1$, linear operator $D_k$ preserves the space of polynomials of degree at most $m$. In more details,
 $$
  D_k x^m = \left(1 -\tfrac{m+1}{k+1}\right) x^m + (\text{polynomial of degree at most } m-1).
 $$
\end{proposition}
In the proof we rely on the following  identity.
\begin{lemma}\label{Lemma_sum_squares_as_second_derivative}
 For $y\in \X_{k}$ we have
 \begin{equation}
 \label{eq_sum_squares_as_second_derivative}
  \sum_{x\in \X_{k+1}} \frac{1}{(x-y)^2}= - \frac{P''_{k+1}(y)}{P_{k+1}(y)}.
 \end{equation}
\end{lemma}
\begin{proof}  This is a reformulation of the second equality in \eqref{eq:var-eta}.\end{proof}

\begin{proof}[Proof of Proposition \ref{Proposition_polynomial_preservation}]
 We are going to use two integral representations for the action of the operator $D_k$ on polynomial functions. First,
 \begin{equation}
 \label{eq_integral_representation_1}
  [D_k f](x)=-\frac{1}{2\pi \ii} \oint_{\X_k} f(z)\cdot  \frac{P_{k+1}(z)}{P'_{k+1}(z)} \cdot \frac{dz}{(z-x)^2},
 \end{equation}
 where the integration contour is positively (i.e.\ counter clock-wise) oriented and includes all poles at points of $\X_k$, but does not include $x$. Indeed, taking into account \eqref{eq_sum_squares_as_second_derivative}, the sum of the residues of \eqref{eq_integral_representation_1} at points $y\in \X_k$ matches the sum in the definition of $D_k$. Second, for $x\in \X_{k+1}$, using $P_{k+1}(x)=0$, we can deform the integration contour in \eqref{eq_integral_representation_1} through the simple pole at $z=x$ picking up the residue $f(x)$ there and get
 \begin{equation}
 \label{eq_integral_representation_2}
  [D_k f](x)=f(x)-\frac{1}{2\pi \ii} \oint_{\infty} f(z)\cdot  \frac{P_{k+1}(z)}{P'_{k+1}(z)} \cdot \frac{dz}{(z-x)^2},
 \end{equation}
 where the integration now goes in positive direction over a very large contour enclosing all singularities of the integrand. Let us emphasize that equality between \eqref{eq_integral_representation_1} and \eqref{eq_integral_representation_2} holds only for $x\in \X_{k+1}$. We now specialize to $f(x)=x^m$ and compute the integral in \eqref{eq_integral_representation_2} as a residue at $\infty$. For that we expand for large $z$
 \begin{equation}
 \label{eq_x9}
  \frac{1}{(z-x)^2}=\frac{1}{z^2}+2 \frac{x}{z^3}+ 3\frac{x^2}{z^4} + 4 \frac{x^3}{z^5}+\dots.
 \end{equation}
 Note that $z^m\cdot  \frac{P_{k+1}(z)}{P'_{k+1}(z)}$ grows in the leading order as $\frac{z^{m+1}}{k+1}$. Hence, only the first $m+1$ terms in \eqref{eq_x9}, which are
 $$
  \frac{1}{z^2}+\dots+ (m+1) \frac{x^m}{z^{m+2}},
 $$
 contribute to the residue. We conclude that this residue is a degree $m$ polynomial of the form $\frac{m+1}{k+1} x^m + \dots$.
\end{proof}

\subsection{Lattices with $3$--term recurrence}

Our next task is to introduce a basis in $\mF_k$, such that the action of $D_k$ is diagonal with respect to this basis. We were unable to present a satisfactory definition for generic choices of $P_k$ and need to restrict ourselves to the following class\footnote{As of 2021, we do not know other classes of $P_k$ leading to explicit identification of a basis. A possible another good case for future investigations is $\beta=\infty$ version of the ergodic measures on eigenvalues of corners of general $\beta$--random matrices of infinite size, see, e.g.,  \cite{AN} and  \cite[Section 4.4]{BCG} for discussions about these measures}:

\begin{definition} \label{Definition_classical}
 We say that polynomials $P_k(z)$ are \emph{classical} if
 \begin{equation}
 \label{eq_3_term_relation}
  P''_{k}(z) \alpha_k(z) + P'_k(z) \beta_k(z) + P_k(z)=0,
 \end{equation}
 where $\alpha_k(z)$ is a polynomial of degree at most $2$ and $\beta_k(z)$ is a polynomial of degree at most $1$.
\end{definition}
Examples are given by classical orthogonal polynomials, see, e.g., \cite{Maroni} and  Section \ref{Section_Hermite_Laguerre_Jacobi}.

\begin{definition} \label{Definition_new_ortho}
 Fix $k$ and equip $\X_k$ with the weight
 \begin{equation}
 \label{eq_weight_def}
  w_k(y)= - \frac{1}{k(k+1)} \cdot \frac{P_{k+1}(y)}{P_{k-1}(y)}= - \frac{P_{k+1}(y)}{P''_{k+1}(y)}.
 \end{equation}
Consider a scalar product on $\mF_k$:
\begin{equation}
 \langle f, g\rangle_k = \sum_{y\in \X_k} f(y) g(y) w_k(y).
\end{equation}
Define $Q^{(k)}_m(x)$, $m=0,1,2,\dots,{k-1}$, to be the monic orthogonal polynomials with respect to this scalar product.
\end{definition}
\begin{remark}
 Due to interlacing between the roots of $P_{k+1}$ and its derivative, the weight $w_k(y)$, $y\in \X_k$ is positive.
\end{remark}
\begin{remark}
 For each $y\in \X_k$, due to \eqref{eq_3_term_relation} and vanishing of $P_k(y)$, we have $w_k(y)=\alpha_{k+1}(y)$.
\end{remark}


\begin{theorem} \label{Theorem_diagonalization_of_transition} Suppose that polynomials $P_k(z)$ are classical. Then for  $0\le m \le k-1$ we have
\begin{equation}
 D_k Q^{(k)}_m = \left(1-\tfrac{m+1}{k+1}\right) Q^{(k+1)}_m.
\end{equation}
\end{theorem}
\begin{proof}
 Proposition \ref{Proposition_polynomial_preservation} implies that $D_k Q^{(k)}_m$ is a degree $m$ polynomial with leading coefficient $\left(1-\tfrac{m+1}{k+1}\right)$. Hence, it remains to prove that:
 \begin{equation}
  \langle D_k Q^{(k)}_m, x^j \rangle_{k+1} \stackrel{?}{=} 0, \quad 0\le j \le m-1.
 \end{equation}
We are going to use the following contour integral representation of the scalar product $\langle f,g\rangle_k$ for polynomial functions $f$ and $g$:
\begin{equation}
\label{eq_scalar_as_contour}
  \langle f,g\rangle_k= - \frac{1}{k+1} \cdot \frac{1}{2\pi \ii} \oint_{\X_k} f(z) g(z) \frac{P_{k+1}(z)}{P_k(z)} dz,
\end{equation}
where the integration contour is counter-clockwise oriented and encloses all singularities of the integrand, which has $k$ poles at the points of $\X_k$ --- roots of $P_k(z)$. Sum of the residues at these poles matches the definition of scalar product. The formula \eqref{eq_scalar_as_contour} remains valid even for non-polynomial functions $f$ and $g$ as long as these functions have an analytic continuation to a small complex neighborhood of $\X_k$; in this situation the integration contour should be a union of small loops around points of $\X_k$.

Combining \eqref{eq_integral_representation_1} with \eqref{eq_scalar_as_contour}, we need to prove:
\begin{equation}
\label{eq_x10}
\oint_{\X_{k+1}}   \left[\oint_{\X_k} Q^{(k)}_m(z)\cdot  \frac{P_{k+1}(z)}{P_{k}(z)} \frac{dz}{(z-u)^2}\right] u^j \frac{P_{k+2}(u)}{P_{k+1}(u)} du \stackrel{?}{=}0.
\end{equation}
Note that the internal integral might fail to be a polynomial as a function of $u$. The $u$--integral in \eqref{eq_x10} is over a union of $k+1$ small loops around points of $\X_{k+1}$ and the $z$--integral is over a union of $k$ small loops around points of $\X_k$

We would like to deform the $u$--contour in \eqref{eq_x10} to make it a large circle. In this deformation we encounter singularities at the double pole $u=z$ resulting (up to $2\pi \ii$ factor, which we omitted) in an additional residue term given by the following integral
\begin{multline}
\label{eq_x11}
 \oint_{\X_k} Q^{(k)}_m(z)\cdot  \frac{P_{k+1}(z)}{P_{k}(z)} \frac{\partial}{\partial z}\left( z^j \frac{P_{k+2}(z)}{P_{k+1}(z)}\right) dz
 =
  \oint_{\X_k} Q^{(k)}_m(z)\cdot  \frac{P_{k+1}(z)}{P_{k}(z)}  j z^{j-1} \frac{P_{k+2}(z)}{P_{k+1}(z)} dz\\+
  \oint_{\X_k} Q^{(k)}_m(z)\cdot  \frac{P_{k+1}(z)}{P_{k}(z)}  z^j (k+2) dz-
 \oint_{\X_k} Q^{(k)}_m(z)\cdot  \frac{P_{k+1}(z)}{P_{k}(z)}  z^j \frac{ P_{k+2}(z) (k+1) P_k(z)}{ (P_{k+1}(z))^2}dz.
\end{multline}
Let us show that each of the integrals in the right-hand side of \eqref{eq_x11} vanishes. In the last one the factor $P_k(z)$ cancels out and there are no singularities inside the integration contour. The middle integral is a scalar product of $Q^{(k)}_m$ and $z^j (k+2)$ and, thus, vanishes. For the remaining first integral we use the three-term relation \eqref{eq_3_term_relation}:
\begin{multline}
  j\oint_{\X_k} Q^{(k)}_m(z)  z^{j-1}  \cdot  \frac{P_{k+2}(z)}{P_{k}(z)}  dz=
  j\oint_{\X_k} Q^{(k)}_m(z)  z^{j-1}  \cdot  \frac{(k+2)(k+1) P_k(z) \alpha_{k+2}(z)}{P_{k}(z)}  dz\\+
    j\oint_{\X_k} Q^{(k)}_m(z)  z^{j-1}  \cdot  \frac{(k+2)P_{k+1}(z) \beta_{k+2}(z)}{P_{k}(z)}  dz.
\end{multline}
For the last two integrals, the first one has integrand with no singularities, hence, it vanishes\footnote{Note that this is the only place where $\alpha_{k+2}(z)$ appears and we do not need it to be a polynomial in order for this argument to work. Yet, it is unclear, whether this observation can be used to add any generality to the theorem that we are proving.}. The second integral is a scalar product of $Q^{(k)}_m$ with polynomial $z^{j-1} (k+2) \beta_{k+2}(z)$ of degree at most $j$, hence, it also vanishes.

Now \eqref{eq_x10} got converted into
\begin{equation}
\label{eq_x12}
\oint_{\infty}   \left[\oint_{\X_k} Q^{(k)}_m(z)\cdot  \frac{P_{k+1}(z)}{P_{k}(z)} \frac{dz}{(z-u)^2}\right] u^j \frac{P_{k+2}(u)}{P_{k+1}(u)} du \stackrel{?}{=}0.
\end{equation}
Let us integrate in $u$ first by computing the $u$--residue at $\infty$. For that we expand $1/(z-u)^2$ in $1/u$ power series. Since $ u^j \frac{P_{k+2}(u)}{P_{k+1}(u)}$ grows as $(k+2) u^{j+1}$, we only need terms up to $1/u^{j+2}$ in the expansion, i.e.\ we need
$$
 \frac{1}{(z-u)^2} = \frac{1}{u^2}+ 2 \frac{z}{u^3} + \dots + (j+1)\frac{z^j}{u^{j+2}} + (\cdots),
$$
where the $(\cdots)$ terms can be ignored. We conclude that the $u$--integral is a polynomial in $z$ of degree at most $j$. Hence, the $z$--integral becomes a scalar product of $Q^{(k)}_m$ with this polynomial and vanishes.
\end{proof}

\subsection{Hermite, Laguerre, and Jacobi examples}

\label{Section_Hermite_Laguerre_Jacobi}

In this section we list the classical polynomials, for which \eqref{eq_3_term_relation} is satisfied. We take the formulas directly from \cite{KS}. 

First, the (monic) Hermite polynomials form an Appell sequence, $H_k'(z)=k H_{k-1}(z)$, and also satisfy a differential equation
$$
 H''_{k}(z)-z H'_{k}(z)+ k H_{k}(z)=0.
$$
Hence, they fit into Definition \ref{Definition_classical}. The weight is constant in this case:
\begin{equation}
\label{eq_Hermite_ortho_weight}
w_k(y)= \frac{1}{k+1}, \qquad y\in \X_k.
\end{equation}

The second example is given by the generalized Laguerre polynomials $L^{(\alpha)}_k(z)$, which solve the following second order differential equation:
\begin{equation}
\label{eq_x13}
 z f''(z)+(\alpha+1-z) f'(z)+ k f(z)=0, \qquad k=0,1,2,\dots.
\end{equation}
The leading coefficient of $L^{(\alpha)}_k(z)$ is usually chosen to be $\frac{(-1)^k}{k!}$ and in this normalization they satisfy the relation
$$
 \frac{\partial}{\partial z} L^{(\alpha)}_k(z)= - L^{(\alpha+1)}_{k-1}(z).
$$
Hence, the polynomials
$$
 P_k(z)= (-1)^k k!\cdot L^{(\alpha-k)}_k (z), \qquad k=0,1,2,\dots,
$$
are monic, form an Appell sequence, and fit into Definition \ref{Definition_classical}. The weight is linear in this case:
$$
 w_k(y)= \frac{y}{k+1},\qquad y\in \X_k.
$$

The third example is given by the Jacobi polynomials $J^{(\alpha,\beta)}_k(z)$, which solve the following second order differential equation:
\begin{equation}
\label{eq_x14}
 \bigl(1-z^2\bigr) f''(z)+\bigl(\beta-\alpha-(\alpha+\beta+2)z\bigr) f'(z)+ k(k+\alpha+\beta+1) f(z)=0.
\end{equation}
If we use the normalization of \cite{KS}, then the leading coefficient is
$$
 2^{-m}\frac{\Gamma(\alpha+\beta+2k+1)}{\Gamma(k+1) \Gamma(\alpha+\beta+k+1)}
$$
and the polynomials satisfy the relation:
$$
 \frac{\partial}{\partial z} J^{(\alpha,\beta)}_k(z) = \frac{k+\alpha+\beta+1}{2} J^{(\alpha+1,\beta+1)}_{k-1}(z).
$$
Hence, the polynomials
\begin{equation}
\label{eq_monic_Jacobi}
 P_k(z)= 2^m \frac{\Gamma(k+1) \Gamma(\alpha+\beta+k+1)}{\Gamma(\alpha+\beta+2k+1)} J^{(\alpha-k,\beta-k)}_{k}(z), \qquad k=0,1,2,\dots
\end{equation}
are monic, form an Appell sequence, and fit into Definition \ref{Definition_classical}. The weight is quadratic:
\begin{equation}
\label{eq_Jacobi_dual_weight}
 w_k(y)= \frac{1-y^2}{(k+1)(k+\alpha+\beta+2)},\qquad y\in \X_k.
\end{equation}

We remark that if $\alpha>-1$ and $\beta>-1$, then Jacobi and Laguerre polynomials are orthogonal (with respect to weights $(1-z)^{\alpha}(1+z)^{\beta}$ on $[-1,1]$ and $x^{\alpha} e^{-x}$ on $[0,+\infty)$, respectively), yet, this restriction on the parameters is not necessary for the polynomials to be well-defined and for the above identities to hold. Note, however, that we need the polynomials to be real--rooted, which is always true for $\alpha>-1$, $\beta>-1$, but fails for some values of $\alpha\le -1$, $\beta\le -1$, see, e.g., \cite{Law,Ber}.

\subsection{Consequences of orthogonality}

Our main motivation for the introduction of the orthogonal polynomials $Q^{(k)}_j$ is that they are helpful in analyzing the covariance \eqref{eq:cov-sum}.

\begin{theorem}  Suppose that polynomials $P_k(z)$ are classical and let $Q^{(k)}_m$ be as in Definition \ref{Definition_new_ortho}. Then the stochastic process $\{\xi^k_a\}_{1\le a \le k\le N}$ admits the following formula for the covariance:
\begin{multline}
\Cov (\xi_{a_1}^{k_1}, \xi_{a_2}^{k_2})= 2 w_{k_1}(x^{k_1}_{a_1}) w_{k_2}(x^{k_2}_{a_2}) \sum_{\ell=\max(k_1,k_2)}^{N-1} \sum_{m=0}^{\min(k_1,k_2)-1}  Q^{(k_1)}_{m}(x^{k_1}_{a_1}) \, Q^{(k_2)}_{m}(x^{k_2}_{a_2}) \\ \times  \frac{\langle Q^{(\ell)}_{m}, Q^{(\ell)}_{m} \rangle_\ell}{\langle Q^{(k_1)}_{m}, Q^{(k_1)}_{m} \rangle_{k_1} \langle Q^{(k_2)}_{m}, Q^{(k_2)}_{m} \rangle_{k_2}}
 \prod_{j=k_1}^{\ell-1} \left(1-\tfrac{m+1}{j+1}\right) \prod_{j=k_2}^{\ell-1} \left(1-\tfrac{m+1}{j+1}\right).
\end{multline}
Further, if $P_k(z)$ are the Hermite polynomials and we deal with $\{\zeta^k_a\}_{1\le a \le k}$, then
\begin{multline}
\label{eq_covariance_zeta}
\Cov (\zeta_{a_1}^{k_1}, \zeta_{a_2}^{k_2})= 2 w_{k_1}(x^{k_1}_{a_1}) w_{k_2}(x^{k_2}_{a_2}) \sum_{\ell=\max(k_1,k_2)}^{\infty} \sum_{m=0}^{\min(k_1,k_2)-1}  Q^{(k_1)}_{m}(x^{k_1}_{a_1}) \, Q^{(k_2)}_{m}(x^{k_2}_{a_2}) \\ \times  \frac{\langle Q^{(\ell)}_{m}, Q^{(\ell)}_{m} \rangle_\ell}{\langle Q^{(k_1)}_{m}, Q^{(k_1)}_{m} \rangle_{k_1} \langle Q^{(k_2)}_{m}, Q^{(k_2)}_{m} \rangle_{k_2}}
 \prod_{j=k_1}^{\ell-1} \left(1-\tfrac{m+1}{j+1}\right) \prod_{j=k_2}^{\ell-1} \left(1-\tfrac{m+1}{j+1}\right).
\end{multline}
\end{theorem}

\begin{proof} The diffusion kernel of Definition \ref{Definition_dif_kernel} admits a spectral representation. Using the notation $\1_{x^k_a}$ for the delta-function at $x^k_a$, we have:
\begin{multline}\label{eq_Diffusion_spectral}
K^{k,\ell}(a\to b)= [D_{\ell-1} \cdots D_{k+1} D_k\1_{x^k_a}] (x^\ell_b)=
\sum_{m=0}^{k-1} \frac{\langle \1_{x^k_a}, Q^{(k)}_m \rangle_k}
{\langle Q^{(k)}_m, Q^{(k)}_m \rangle_k} [D_k D_{k+1}\cdots D_{\ell-1} Q^{(k)}_m ](x^\ell_b)]
\\=
 w_k(x^k_a) \sum_{m=0}^{k-1} \frac{Q^{(k)}_m(x^k_a) Q^{(\ell)}_m(x^\ell_b)}{\langle Q^{(k)}_m, Q^{(k)}_m \rangle_k} \prod_{j=k}^{\ell-1} \left(1-\tfrac{m+1}{j+1}\right).
\end{multline}
Using \eqref{eq:cov-sum} and $\Var\eta^\ell_b=2 w_\ell(x^\ell_b)$, we further write
\begin{multline}
\Cov (\xi_{a_1}^{k_1}, \xi_{a_2}^{k_2}) = 2 \sum_{\ell=\max(k_1,k_2)}^{N-1}
\langle K^{k_1,\ell}(a_1\to \cdot ), K^{k_2,\ell}(a_2\to \cdot ) \rangle_\ell
\\= 2 w_{k_1}(x^{k_1}_{a_1}) w_{k_2}(x^{k_2}_{a_2})\sum_{\ell=\max(k_1,k_2)}^{N-1} \sum_{m_1=0}^{k_1-1} \sum_{m_2=0}^{k_2-1} Q^{(k_1)}_{m_1}(x^{k_1}_{a_1}) Q^{(k_2)}_{m_2}(x^{k_2}_{a_2}) \frac{\langle Q^{(\ell)}_{m_1}, Q^{(\ell)}_{m_2} \rangle_\ell}{\langle Q^{(k_1)}_{m_1}, Q^{(k_1)}_{m_1} \rangle_{k_1} \langle Q^{(k_2)}_{m_2}, Q^{(k_2)}_{m_2} \rangle_{k_2}}
\\ \times \prod_{j=k_1}^{\ell-1} \left(1-\tfrac{m_1+1}{j+1}\right) \prod_{j=k_2}^{\ell-1} \left(1-\tfrac{m_2+1}{j+1}\right).
\end{multline}
Orthogonality implies $m_1=m_2$ and the last expression simplifies to
\begin{multline}
2 w_{k_1}(x^{k_1}_{a_1}) w_{k_2}(x^{k_2}_{a_2})\sum_{\ell=\max(k_1,k_2)}^{N-1} \sum_{m=0}^{\min(k_1,k_2)-1}  Q^{(k_1)}_{m}(x^{k_1}_{a_1}) Q^{(k_2)}_{m}(x^{k_2}_{a_2}) \frac{\langle Q^{(\ell)}_{m}, Q^{(\ell)}_{m} \rangle_\ell}{\langle Q^{(k_1)}_{m}, Q^{(k_1)}_{m} \rangle_{k_1} \langle Q^{(k_2)}_{m}, Q^{(k_2)}_{m} \rangle_{k_2}}
\\ \times \prod_{j=k_1}^{\ell-1} \left(1-\tfrac{m+1}{j+1}\right) \prod_{j=k_2}^{\ell-1} \left(1-\tfrac{m+1}{j+1}\right).
\end{multline}
For $\{\zeta^k_a\}_{1\le a \le k}$ the argument is the same.
\end{proof}

\subsection{Duality property}

\label{Section_duality}

In previous subsection we explained how $\{\xi^k_a\}_{1\le a \le k\le N}$ can be analyzed using orthogonal polynomials $Q^{(k)}_m(z)$ of Definition \ref{Definition_new_ortho}. Our next aim is to collect the necessary tools for obtaining the asymptotic theorems about these polynomials.

Although polynomials $Q^{(k)}_m(z)$ are not well-known, but they have appeared in the literature previously. Some of their properties are explained in \cite{Vinet-Z} with certain elements of the constructions going back to \cite{Boor-Saff}, \cite{Borodin-OP} and others being rooted in classical orthogonal polynomial topics: associated polynomials (we rely on \cite{AW}), quadrature formulas, and Christoffel numbers. Let us present a general framework.

Suppose that we are given a sequence of monic orthogonal polynomials\footnote{We do NOT assume these polynomials to form an Appell sequence.} $\mathcal P_n(x)$, $n=0,1,2,\dots$ satisfying a three-term recurrence:
\begin{equation}
\label{eq_tridiag}
 \mathcal P_{n+1}(x)+ b_n \mathcal P_n(x) + u_n \mathcal P_{n-1}(x)=x \mathcal P_n(x)
\end{equation}
with an initial condition
$$
 \mathcal P_0(x)=1,\qquad \mathcal P_1(x)=x-b_0.
$$
One way to think about \eqref{eq_tridiag} is by considering a tridiagonal matrix of the form
\begin{equation}
\label{eq_tridiag_matrix}
 \begin{pmatrix}
  b_0 & u_1 & 0 &\dots \\
  1& b_1 & u_2 & 0 & \dots\\
  0 & 1 & b_2\\
  \vdots & & &\ddots
 \end{pmatrix}.
\end{equation}
Then \eqref{eq_tridiag} says that the operator of multiplication by $x$ is given by the matrix \eqref{eq_tridiag_matrix} in the basis of orthogonal polynomials $\mathcal P_0(x), \mathcal P_1(x),\dots$. Simultaneously, denoting through $\mathcal M_n$ the top-left $n\times n$ corner of \eqref{eq_tridiag_matrix}, we see that the recurrence \eqref{eq_tridiag} is solved by
\begin{equation}
\label{eq_characteristic_solution}
\mathcal P_n(x)=\det(x-\mathcal M_n).
\end{equation}
Fix $N>0$ and define \emph{dual polynomials} $\mathcal Q_n(x)$, $n=0,1,\dots,N-1$ through the dual recurrence:
\begin{equation}
\label{eq_tridiag_dual}
 \mathcal Q_{n+1}(x)+ b_{N-n-1} \mathcal Q_n(x) + u_{N-n} \mathcal Q_{n-1}(x)=x \mathcal Q_n(x)
\end{equation}
with the initial condition
$$
 \mathcal Q_0(x)=1,\qquad \mathcal Q_1(x)=x-b_{N-1}.
$$
In other words, the $N\times N$ tridiagonal matrices corresponding to \eqref{eq_tridiag} and \eqref{eq_tridiag_dual} differ by reflection with respect to the $\diagup$ diagonal.

It turns out that polynomials $\mathcal Q_n$ have an explicit orthogonality measure, which is supported on the $N$ roots of $P_N$ and has weight:
\begin{equation}
\label{eq_weight_Zhedanov}
 w^*(x)= \frac{\mathcal P_{N-1}(x)}{\mathcal P'_N(x)}, \qquad \text{for } x \text{ such that } P_N(x)=0.
\end{equation}
\cite[(1.20)]{Vinet-Z} explains that
\begin{equation}
 \sum_{x\mid \mathcal P_N(x)=0} w^*(x) \mathcal Q_m(x) \mathcal Q_n(x)= \1_{n=m} \cdot h_n, \quad 0\le n,m \le N-1.
\end{equation}
Let us compare the weight $w^*(x)$ of \eqref{eq_weight_Zhedanov} with $w_k(x)$ of Definition \ref{Definition_classical}. In general, the formulas are different, however, it is important to recall that we actually deal with classical polynomials.
Indeed, \cite{AC} suggested to \emph{define} classical orthogonal polynomials as those satisfying a relation
\begin{equation}
\label{eq_diff_relation_AC}
  \pi (x)\mathcal P'_n (x) = \left( {\alpha _n x + \beta _n } \right)\mathcal P_n (x) + \gamma_n \mathcal P_{n-1} (x),\quad n \geq 1,
\end{equation}
where $\pi(x)$ is a polynomial (which then has to be of degree at most $2$). The relation \eqref{eq_diff_relation_AC} readily implies that $w^*(x)$ is a polynomial of degree at most $2$ (and the latter fact can be used as yet another definition of classical orthogonal polynomials, see \cite{Vinet-Z}), matching the examples of Section \ref{Section_Hermite_Laguerre_Jacobi}. In particular, for the monic Jacobi polynomials \eqref{eq_monic_Jacobi} the relation \eqref{eq_diff_relation_AC} takes the form
\begin{multline*}
   (x^2-1)\mathcal P'_n (x) = \left( n x + n\frac{\beta^2-\alpha^2}{(\alpha+\beta)(2n+\alpha+\beta)}  \right)\mathcal P_n (x) \\ -\frac{4 n (n+\alpha+\beta)(n+\alpha)(n+\beta)}{(2n+\alpha+\beta-1)(2n+\alpha+\beta)^2} \mathcal P_{n-1} (x),
\end{multline*}
giving the match between $w^*(x)$ and $w_k(x)$ of \eqref{eq_Jacobi_dual_weight} up to a constant factor. Hence, monic orthogonal polynomials with respect to these weights coincide.
\bigskip

We also rely on a link between dual and \emph{associated} polynomials. Fix a parameter $c=0,1,2,\dots,$ and define the associated polynomials $\mathcal P_n^{(c)}(x)$ as a solution to the three-term recurrence:
\begin{equation}
\label{eq_tridiag_associated}
 \mathcal P_{n+1}^{(c)}(x)+ b_{n+c} \mathcal P_n^{(c)}(x) + u_{n+c} \mathcal P_{n-1}^{(c)}(x)=x \mathcal P_n^{(c)}(x)
\end{equation}
and the initial condition
$$
 \mathcal P_0^{(c)}(x)=1,\qquad \mathcal P_1^{(c)}(x)=x-b_c.
$$
In terms of the tridiagonal matrix \eqref{eq_tridiag_matrix} we deleted the first $c$ rows and the first $c$ columns.

Then, either using  \cite[Theorem 1]{Vinet-Z} or comparing \eqref{eq_characteristic_solution} for dual and associated polynomials, one identifies
\begin{equation}
 Q_n(x)=\mathcal P_n^{(N-n)}(x), \quad 0 \le n \le N.
\end{equation}
In particular, $Q_N=\mathcal P_N^{(0)}=\mathcal P_N$.

\bigskip

For us the most important case is when $\mathcal P_k(x)$ are the Hermite polynomials. In this situation, we saw in Section \ref{Section_Hermite_Laguerre_Jacobi} that $w_k(x)=\tfrac{1}{k+1}$. On the other hand, $\mathcal P'_n(x)= n \mathcal P_{n-1}(x)$ and, therefore, $w^*(x)$ is also a constant. Taking into account the three--term relation for the Hermite polynomials
$$
 H_{n+1}(x)+n H_{n-1}(x)= x H_n(x)
$$
and for the associated version
$$
 H_{n+1}^{(c)}(x)+(n+c) H_{n-1}^{(c)}(x)= x H_n^{(c)}(x),
$$
we record the conclusion:
\begin{proposition}
 Let $P_k(z)$, $k=0,1,2,\dots$ be the Hermite polynomials $H_k(z)$. Then orthogonal polynomials $Q^{(k)}_m(z)$ of Definition \ref{Definition_new_ortho} satisfy the three-term recurrence:
 \begin{equation}
\label{eq_tridiag_dual_Hermite}
 Q_{m+1}^{(k)}(z) + (k-m) Q_{m-1}^{(k)}(z)=z Q_m^{(k)}(z), \quad 0\le m \le k-1
\end{equation}
and the initial conditions
\begin{equation}
\label{eq_dual_Hermite_initial}
 Q_0^{(k)}(z)=1,\qquad Q_1^{(k)}(z)=z.
\end{equation}
We also have an identity with the associated Hermite polynomials:
\begin{equation}\label{eq_dual_as_associated}
 Q_m^{(k)}(z)= H^{(k-m)}_m(z), \quad 0\le m \le k.
\end{equation}
\end{proposition}
\begin{corollary} \label{Corollary_Q_norm}We have
 \begin{equation}
 \label{eq_Q_norm}
  \langle Q_m^{(k)},Q_m^{(k)}\rangle_k=\frac{k (k-1)(k-2)\cdots(k-m)}{k+1}.
 \end{equation}
\end{corollary}
\begin{proof}
 For any sequence orthogonal polynomials satisfying a three-term recurrence of the form \eqref{eq_tridiag}, the ratio of the norm of the $m$--th polynomial and the norm of the $0$th polynomial is $u_1 u_2 \cdots u_m$.
\end{proof}

Here is one more ingredient that we need.

\begin{proposition} The associated Hermite polynomials have an explicit generating function:
\begin{equation}
\label{eq_Hermite_gen}
 \sum_{n=0}^{\infty} v^n \frac{H_n^{(c)}(x)}{(c+1)_n}= c v^{-c}\exp\left(-\tfrac{v^2}{2}+xv\right) \int_0^{v} u^{c-1} \exp\left(\tfrac{u^2}{2}-xu\right) du,
\end{equation}
which can be rewritten using \eqref{eq_dual_as_associated} as a contour integral
\begin{equation}
\label{eq_Q_contour_statement}
 Q^{(k)}_m(x)=\frac{(k-m)_{m+1}}{2\pi \ii} \oint_0 v^{-(k-m)}\exp\left(-\tfrac{v^2}{2}+x v\right) \left[\int_0^v u^{k-m-1}\exp\left(\tfrac{u^2}{2}-xu\right) \d u\right] \frac{\d v}{v^{m+1}}.
\end{equation}
\end{proposition}
\begin{proof}
 See \cite[(4.14)]{AW}, but note a different definition of the Hermite polynomials used there --- they are orthogonal with respect to $\exp(-x^2)$ in \cite{AW} rather than $\exp(-x^2/2)$ used here.
\end{proof}
\begin{remark}
 One can directly check the that the right-hand side of \eqref{eq_Q_contour_statement} satisfies the relations \eqref{eq_tridiag_dual_Hermite} and \eqref{eq_dual_Hermite_initial}.
\end{remark}

\section{G$\infty$E limit: Proof of Theorem \ref{Theorem_Gaussian_limit}}

\label{Section_Gaussian_limit_proof}

The proof relies on several lemmas. We use the notations of Section \ref{Section_innovations}. As before, for $1\le k\le N$, $x_i^k$ are the roots of $P_k(x)\sim \left(\frac{\partial}{\partial x}\right)^{N-k} P_N(x)$ and $K^{k,\ell}(a\to b)$ are diffusion kernels of Definition \ref{Definition_dif_kernel}.

\begin{lemma}\label{Lemma_kernel_upper_bound} The matrix elements of the diffusion kernel of Definition \ref{Definition_dif_kernel} satisfy
\begin{equation}
 K^{k,\ell}(a\to b)\le \frac{k}{\ell},\qquad \ell>k.
\end{equation}
\end{lemma}
\begin{proof} Applying Proposition \ref{Proposition_polynomial_preservation} with $m=0$, we get for each $b\in \{1,2,\dots,\ell\}$:
$$
 \sum_{a=1}^k K^{k,\ell}(a\to b)= \left(1-\frac{1}{k+1}\right)\cdot \left(1-\frac{1}{k+2}\right) \cdots \left(1-\frac{1}{\ell}\right)=\frac{k}{\ell}.
$$
In words, the above formula says that the uniform measure on $\X_k$ is mapped to the uniform measure on $\X_l$ by our diffusion.
It remains to use the non-negativity of $K^{k,\ell}(a\to b)$.
\end{proof}

\begin{lemma}\label{Lemma_Variance_sum}
 For each $1\le k < N$ we have
 \begin{equation}
   \sum_{i=1}^k \Var \left(\eta_i^k\right) = \frac{2}{k+1} \left(\frac{1}{k+1}\sum\limits_{i=1}^{k+1} (x_i^{k+1})^2-\left(\frac{1}{k+1} \sum\limits_{i=1}^{k+1} x_i^{k+1}\right)^2\right).
 \end{equation}
\end{lemma}
\begin{proof}
 We write using \eqref{eq:var-eta}:
 \begin{equation}
     \sum_{i=1}^k \Var \left(\eta_i^k\right)=-2\sum_{i=1}^k \frac{P_{k+1}(x^k_i)}{P''_{k+1}(x^k_i)}=-2\sum_{x:\, P'_{k+1}(x)=0} \frac{P_{k+1}(x)}{P''_{k+1}(x)}=-\frac{1}{\pi \ii} \oint_{\infty}  \frac{P_{k+1}(z)}{P'_{k+1}(z)}dz,
 \end{equation}
 where the integration goes over a large positively-oriented contour enclosing all singularities of the integrand. We further compute the last integral as the coefficient of $1/z$ in the following power series expansion at $z=\infty$:
 \begin{multline*}
  \frac{P_{k+1}(z)}{P'_{k+1}(z)}=\left( \sum_{i=1}^{k+1} \frac{1}{z-x^{k+1}_i}\right)^{-1}=z
   \left( \sum_{i=1}^{k+1} \frac{1}{1-x^{k+1}_i/z}\right)^{-1}\\= \frac{z}{k+1} \left(1+ \frac1{k+1}\sum_{i=1}^{k+1} \frac{x^{k+1}_i}{z}+\frac1{k+1}  \sum_{i=1}^{k+1} \left(\frac{x^{k+1}_i}{z}\right)^2 +O(z^{-3}) \right)^{-1}\\=
   \frac{z}{k+1}- \frac1{(k+1)^2}\sum_{i=1}^{k+1} x^{k+1}_i +\frac1{z(k+1)} \left( \left(\frac{1}{k+1}\sum_{i=1}^{k+1} x^{k+1}_i \right)^2-\frac{1}{k+1}\sum_{i=1}^{k+1} (x^{k+1}_i)^2 \right) +O(z^{-2}).
 \end{multline*}
 The coefficient of $\frac{1}{z}$ in the last expression matches the desired formula.
\end{proof}

\begin{lemma} \label{Lemma_squares_transition}
 If $\sum\limits_{i=1}^N x_i^N=0$  and $\frac{1}{N}\sum\limits_{i=1}^N (x_i^N)^2=\sigma^2$,
 then for all $1\le k \le N$ we have
 $$
   \sum_{i=1}^k x_i^k=0\quad \text{ and }\quad \frac{1}{k}\sum_{i=1}^k (x_i^k)^2=\frac{k-1}{N-1} \sigma^2.
 $$
\end{lemma}
\begin{proof}
 We proceed by induction in $(N-k)$ with the base case $N-k=0$ being obvious. Suppose that the statement is true for some $k$. Then
 \begin{multline*}
  P_{k}(z)=\prod_{i=1}^{k} (z-x^{k}_i)=z^{k}-\left(\sum_{i=1}^{k} x^{k}_i \right)\cdot z^{k-1}+ \left(\sum_{i<j} x^{k}_i x^{k}_j \right)\cdot z^{k-2}-\dots
  \\= z^{k}-0 \cdot z^{k-1}+ \left( \frac{1}{2}\left(\sum_{i=1}^{k} x^{k}_i\right)^2-\frac{1}{2}\sum_{i=1}^{k} (x^{k}_i)^2  \right)\cdot z^{k-2}-\dots
  \\=z^{k}-0 \cdot z^{k-1}-\frac{1}{2}\left(\sum_{i=1}^{k} (x^{k}_i)^2  \right)\cdot z^{k-2}-\dots.
 \end{multline*}
Differentiating, we get
$$
 P_{k-1}(z)=\frac{1}{k} \frac{\partial}{\partial z} P_{k}(z)=z^{k-1}-0 \cdot z^{k-2}-\frac{k-2}{2 k} \left(\sum_{i=1}^{k} (x^{k}_i)^2  \right)\cdot z^{k-3}-\dots.
$$
Comparing the coefficient of $z^{k-2}$ with the expansion of $P_{k-1}(z)=\prod_{i=1}^{k-1}(z-x^{k-1}_i)$, we conclude that $\sum_{i=1}^{k-1} x^{k-1}_i=0$. Then comparing the coefficient of $z^{k-3}$  and dividing by $(k-1)$ we deduce
$$
 \frac{1}{k-1} \sum_{i=1}^{k-1} (x^{k-1}_i)^2=\frac{k-2}{k-1} \cdot  \frac{1}{k} \sum_{i=1}^{k} (x^{k}_i)^2.\qedhere
$$
\end{proof}

\begin{proof}[Proof of Theorem \ref{Theorem_Gaussian_limit}] We are going to assume that $\mu_N=0$ and $\sigma_N=\sqrt{N}$. All other cases can be obtained by shifting and rescaling the relevant variables. Theorem \ref{Theorem_Hermite_limit} then implies the convergence of $x_i^k$, $i=1,\dots,k$, towards the roots $h_i^k$ of the Hermite polynomial $H_k$.

We further use the expansions \eqref{eq:xi-K} and \eqref{eq:zeta-K_infty}. We have
\begin{equation}\label{eq:xi-K_2}
\xi_a^k = \sum_{\ell=k}^{N-1}  \sum_{b=1}^{\ell} K^{k,\ell}(a\to b)  \cdot \eta^\ell_b,
\end{equation}
where $\eta^{\ell}_b$ are independent centered Gaussians with variances \eqref{eq:var-eta}. Also
\begin{equation}\label{eq:zeta-K_infty_2}
\zeta_a^k = \sum_{\ell=k}^{\infty} \sum_{b=1}^{\ell} \tilde K^{k,\ell}(a\to b)  \cdot \tilde \eta^\ell_b,
\end{equation}
where the variances of the noises $\tilde \eta^\ell_n$ and kernels $\tilde K^{k,\ell}(a\to b)$ are now constructed using the roots $h_i^k$ of the Hermite polynomials instead of $x_i^k$.

Convergence of $x_i^k$ towards $h_i^k$ readily implies that the expansion \eqref{eq:xi-K_2} converges to \eqref{eq:zeta-K_infty_2} term by term. It remains to produce a tail bound showing that the terms with large $\ell$ do not contribute to \eqref{eq:xi-K_2} (and similar argument would work for \eqref{eq:zeta-K_infty_2}).

For that we write using Lemmas \ref{Lemma_kernel_upper_bound}, \ref{Lemma_Variance_sum}, \ref{Lemma_squares_transition}:
\begin{multline}
 \Var \left( \sum_{\ell=L}^{N-1}  \sum_{b=1}^{\ell} K^{k,\ell}(a\to b)  \cdot \eta^\ell_b\right)=
  \sum_{\ell=L}^{N-1}  \sum_{b=1}^{\ell} \bigl(K^{k,\ell}(a\to b)\bigr)^2  \cdot \Var(\eta^\ell_b)\\ \le
  \sum_{\ell=L}^{N-1}  \left(\max_{1\le b \le \ell} K^{k,\ell}(a\to b)\right)^2  \cdot  \sum_{b=1}^{\ell} \Var(\eta^\ell_b)\le   \sum_{\ell=L}^{N-1}  \frac{k^2}{\ell^2}  \cdot \frac{2}{\ell+1} \cdot \frac{\ell}{N-1}\cdot  N \le 4 k^2 \sum_{\ell=L}^{N} \frac{1}{\ell^2},
\end{multline}
which converges (uniformly in $N$) to zero as $L\to\infty$.
\end{proof}

\section{Edge limit: proof of Theorem \ref{Theorem_Gcorners_limit_intro} and properties of $\mathfrak Z(i,t)$.}

This section has four parts. First, we analyze orthogonal polynomials $Q^{(k)}_{m}(z)$ in the \mbox{asymptotic} regime relevant to Theorems \ref{Theorem_Gcorners_limit_intro} and \ref{Theorem_DBM_limit_intro}. Then we prove Theorem \ref{Theorem_Gcorners_limit_intro}. In the third subsection we explain how the limiting object (Airy$_{\infty}$ line ensemble) can be identified with a partition function of a polymer whose trajectories travel over the roots of the Airy function. Finally, in the last subsection we apply Kolmogorov continuity theorem to deduce the regularity of the trajectories of $\mathfrak Z(i,t)$.

\subsection{Asymptotic theorem for polynomials $Q^{(k)}_{m}(z)$}

Recall that the Airy function $\Ai(z)$ is defined as a solution to the  differential equation
\begin{equation}
 \label{eq_Airy_DE}
 \Ai''(z)=z \Ai(z),
\end{equation}
given explicitly by the contour integral
\begin{equation}\label{eq_Ai_function}
\Ai(z) := \frac{1}{2\pi \ii}\int{\exp\left( \frac{v^3}{3} - zv \right) dv},
\end{equation}
where the contour in the integral is the upwards-directed contour which is the union of the lines $\{ e^{-i\pi/3}t : t \geq 0\}$ and $\{ e^{i\pi/3}t : t \geq 0\}$.

\begin{theorem} \label{Theorem_Q_to_Airy} Let polynomials $Q^{(k)}_m$ be as in Definition \ref{Definition_classical} for $P_k$ being the Hermite polynomials $H_k$. Let $x^{k+1-i}_k$ be the $i$th largest root of $H_k$.  Then for each fixed $i=1,2,\dots$, as $k\to\infty$ we have
\begin{equation}
\label{eq_Q_to_Airy}
 k^{-1/3} \frac{Q^{(k)}_m(x_{k+1-i}^k)}{\sqrt{\langle Q^{(k)}_m, Q^{(k)}_m\rangle_k}}= \frac{\Ai\left(\a_i+ \frac{m}{k^{1/3}}\right)}{\Ai'(\a_i)} (1+o(1)),
\end{equation}
where $\a_i$ is the $i$th largest real zero of the Airy function and convergence is uniform over $m$ such that the ratio $\tfrac{m}{k^{1/3}}$ belongs to compact subsets of $[0,+\infty)$. In addition, there exists $C>0$, such that we have a uniform bound
\begin{equation}
\label{eq_Q_uniform_bound}
 \left| k^{-1/3} \frac{Q^{(k)}_m(x_{k+1-i}^k)}{\sqrt{\langle Q^{(k)}_m, Q^{(k)}_m\rangle_k}} \right| < C \left(1+ \frac{m}{k^{1/3}}\right)^{-1}, \quad 0\le m \le k-1, \quad k=1,2,\dots.
\end{equation}
\end{theorem}
We present two proofs of Theorem \ref{Theorem_Q_to_Airy}. The first one shows that the relation \eqref{eq_tridiag_dual_Hermite} after proper rescaling of variables converges to the Airy differential equation \eqref{eq_Airy_DE}. This is how we first arrived at the asymptotic statement \eqref{eq_Q_to_Airy}. In principle, the convergence of the equations should imply the desired convergence of their solutions, yet, additional technical efforts are needed (the Airy differential equation has a second solution, which is explosive at $+\infty$ and may potentially lead to large errors in approximations). Simultaneously with our work (and independently) Theorem \ref{Theorem_Q_to_Airy} was obtained in \cite{AHV}: they also rely on \eqref{eq_tridiag_dual_Hermite} and use several clever analytic tricks to show convergence of its solution to the Airy function.

In our second proof we provide a very different argument and arrive at an integral representation for the right-hand side of \eqref{eq_Q_to_Airy} (different from \eqref{eq_Ai_function}) by applying the steepest descent analysis to the generating function of \eqref{eq_Hermite_gen}.

 \begin{figure}[t]
\begin{center}
{\scalebox{0.8}{\includegraphics{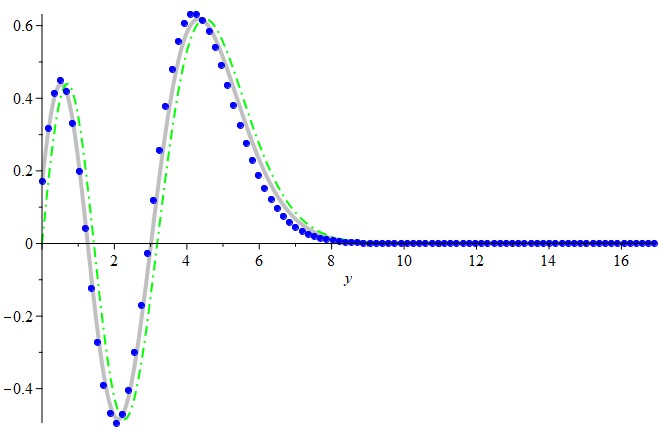}}}
 \caption{Blue points are $\left(\frac{m}{k^{1/3}}, k^{-1/3} \frac{Q^{(k)}_m(x_{k}^k)}{\sqrt{\langle Q^{(k)}_m, Q^{(k)}_m\rangle_k}}\right)$ for $k=200$ and $m=0,1,\dots,100$. Gray thick line is the graph of $\frac{\Ai\left(\a_1+ y +k^{-1/3}\right)}{\Ai'(\a_1)}$, green dash-dotted line is the graph of $\frac{\Ai\left(\a_1+ y \right)}{\Ai'(\a_1)}$.    \label{Fig_Airy_conv}}
\end{center}
\end{figure}

\begin{remark}
 While it does not matter for the validity of the statement, but from the numeric point of view, we found that if we replace the right-hand side of \eqref{eq_Q_to_Airy} with
$$
 \frac{\Ai\left(\a_i+ \frac{m+1}{k^{1/3}}\right)}{\Ai'(\a_i)},
$$
then we get a  better agreement for the finite values of $k$, see Figure \ref{Fig_Airy_conv}.
\end{remark}

\begin{remark} \label{Remark_Airy_integral}
 Here is a way to check normalizations in \eqref{eq_Q_to_Airy}. Note that the matrix
 $$
  \left[\frac{1}{\sqrt{k+1}} \frac{Q_m^{(k)}(x_{k+1-i}^k)}{\sqrt{\langle Q_{m}^{(k)},Q_{m}^{(k)}\rangle_k}} \right]_{1\le i \le k,\, 0\le m \le k-1}
 $$
 is orthogonal. Hence,
 \begin{equation}
 \label{eq_Q_dual_norm}
  \sum_{m=0}^{k-1} \frac{1}{k+1} \cdot  \frac{ \left(Q_m^{(k)}(x_{k+1-i}^k)\right)^2}{\langle Q_{m}^{(k)},Q_{m}^{(k)}\rangle_k}=1.
 \end{equation}
 As $k\to\infty$ the sum becomes integral. Hence, if the normalization in \eqref{eq_Q_to_Airy} is correct, then we should have
 \begin{equation*}
 \label{eq_Airy_norm}
  \int_0^{\infty}\left(\frac{\Ai\left(\a_i+ y\right)}{\Ai'(\a_i)}\right)^2 dy=1.
 \end{equation*}
 But indeed, integrating by parts, using $\Ai(\a_i)=0$ and the Airy differential equation, we have
 \begin{multline} \label{eq_squared_integral}
  \int_{\a_i}^{\infty} \Ai^2(y) \d y = -2 \int_{\a_i}^{\infty} \Ai'(y) \Ai(y) y \d y=-2\int_{\a_i}^{\infty} \Ai'(y) \Ai''(y) \d y\\= (\Ai'(y))^2\Bigr|^{\a_i}_{+\infty}= (\Ai'(\a_i))^2.
 \end{multline}
 The same orthogonality implies that we should also have
 \begin{equation*}
 \label{eq_Airy_orthogonality}
  \int_0^{\infty}\left(\frac{\Ai\left(\a_i+ y\right)}{\Ai'(\a_i)}\right)\left(\frac{\Ai\left(\a_j+ y\right)}{\Ai'(\a_j)}\right) dy=0, \qquad i\ne j.
 \end{equation*}
 And indeed,
 \begin{multline*}
  \frac{\partial}{\partial y} \left[\Ai(\a_i+y) \Ai'(\a_j+y)-\Ai'(\a_i+y) \Ai(\a_j+y)\right]\\ =(\a_j+y)\Ai(\a_i+y) \Ai(\a_j+y)- (\a_i+y)\Ai(\a_i+y) \Ai(\a_j+y)= (\a_j-\a_i) \Ai(\a_i+y) \Ai(\a_j+y).
 \end{multline*}
 Hence, $\frac{1}{\a_j-\a_i}   \left[\Ai(\a_i+y) \Ai'(\a_j+y)-\Ai'(\a_i+y) \Ai(\a_j+y)\right]$ is an antiderivative of $\Ai\left(\a_i+ y\right)\Ai\left(\a_j+ y\right)$, which implies \eqref{eq_Airy_orthogonality}.
\end{remark}

\begin{proof}[Sketch of the first proof of Theorem \ref{Theorem_Q_to_Airy}]
 We start by noting that as $k\to\infty$
 \begin{equation}
 \label{eq_x_edge}
  x_{k+1-i}^k= 2\sqrt{k}+ k^{-1/6} \a_i (1+o(1)),
 \end{equation}
 as follows from the Plancherel-Rotach asymptotics (going back to \cite{PR}) for the Hermite polynomials $H_k(x)$ for $x$ close to $2\sqrt{k}$.  Using \eqref{eq_Q_norm} we transform \eqref{eq_tridiag_dual_Hermite} into
 \begin{multline}
 \label{eq_Q_recurrence_normalized}
 \sqrt{k-m-1}\frac{Q_{m+1}^{(k)}(x_{k+1-i}^k)}{\sqrt{\langle Q_{m+1}^{(k)},Q_{m+1}^{(k)}\rangle_k}} + \sqrt{k-m} \frac{Q_{m-1}^{(k)}(x_{k+1-i}^k)}{\sqrt{\langle Q_{m-1}^{(k)},Q_{m-1}^{(k)}\rangle_k}}\\=\bigl( 2\sqrt{k}+ k^{-1/6} \a_i (1+o(1))\bigr) \frac{Q_m^{(k)}(x_{k+1-i}^k)}{\sqrt{\langle Q_{m}^{(k)},Q_{m}^{(k)}\rangle_k}}.
 \end{multline}
 Dividing \eqref{eq_Q_recurrence_normalized} by $\sqrt{k}$ and Taylor-expanding square roots using $\sqrt{1-q}= 1-\frac{q}{2}+o(q)$, we get
 \begin{multline}
 \label{eq_Q_recurrence_normalized_scaled}
 \frac{Q_{m+1}^{(k)}(x_{k+1-i}^k)}{\sqrt{\langle Q_{m+1}^{(k)},Q_{m+1}^{(k)}\rangle_k}} -2 \frac{Q_m^{(k)}(x_{k+1-i}^k)}{\sqrt{\langle Q_{m}^{(k)},Q_{m}^{(k)}\rangle_k}} + \frac{Q_{m-1}^{(k)}(x_{k+1-i}^k)}{\sqrt{\langle Q_{m-1}^{(k)},Q_{m-1}^{(k)}\rangle_k}}\\= k^{-2/3} \a_i (1+o(1)) \frac{Q_m^{(k)}(x_{k+1-i}^k)}{\sqrt{\langle Q_{m}^{(k)},Q_{m}^{(k)}\rangle_k}} + \frac{m}{2 k}(1+o(1))\left(\frac{Q_{m+1}^{(k)}(x_{k+1-i}^k)}{\sqrt{\langle Q_{m+1}^{(k)},Q_{m+1}^{(k)}\rangle_k}} +\frac{Q_{m-1}^{(k)}(x_{k+1-i}^k)}{\sqrt{\langle Q_{m-1}^{(k)},Q_{m-1}^{(k)}\rangle_k}}\right).
 \end{multline}
 Next, let $y=\frac{m}{k^{1/3}}$ be finite.  Then in the leading order \eqref{eq_Q_recurrence_normalized_scaled} becomes
\begin{multline}
 \label{eq_Q_recurrence_approximating}
  k^{2/3}\left(\frac{Q_{m+1}^{(k)}(x_{k+1-i}^k)}{\sqrt{\langle Q_{m+1}^{(k)},Q_{m+1}^{(k)}\rangle_k}} -2 \frac{Q_m^{(k)}(x_{k+1-i}^k)}{\sqrt{\langle Q_{m}^{(k)},Q_{m}^{(k)}\rangle_k}} + \frac{Q_{m-1}^{(k)}(x_{k+1-i}^k)}{\sqrt{\langle Q_{m-1}^{(k)},Q_{m-1}^{(k)}\rangle_k}}\right)\\\approx  \a_i  \frac{Q_m^{(k)}(x_{k+1-i}^k)}{\sqrt{\langle Q_{m}^{(k)},Q_{m}^{(k)}\rangle_k}}+  \frac{y}{2} \left(\frac{Q_{m+1}^{(k)}(x_{k+1-i}^k)}{\sqrt{\langle Q_{m+1}^{(k)},Q_{m+1}^{(k)}\rangle_k}} +\frac{Q_{m-1}^{(k)}(x_{k+1-i}^k)}{\sqrt{\langle Q_{m-1}^{(k)},Q_{m-1}^{(k)}\rangle_k}}\right).
\end{multline}
If we now treat $\frac{Q_m^{(k)}(x_{k+1-i}^k)}{\sqrt{\langle Q_{m}^{(k)},Q_{m}^{(k)}\rangle_k}}$ as a function of $y$, then \eqref{eq_Q_recurrence_approximating} is precisely a finite-difference approximation of the differential equation \eqref{eq_Airy_DE} upon identification $z=y+\a_i$.

It remains to match the boundary conditions and normalization. Note that the right-hand side of \eqref{eq_Q_to_Airy} as a function of $y$ has value $0$ and derivative $1$ at $y=0$. For the left-hand side, $Q_0^{(k)}(x^k_{k+1-i})=1$, and, therefore, as $k\to\infty$
\begin{equation}
\label{eq_Q_boundary}
  k^{-1/3} \frac{Q^{(k)}_0(x_{k+1-i}^k)}{\sqrt{\langle Q^{(k)}_0, Q^{(k)}_0\rangle_k}}= k^{-1/3} \sqrt{\frac{k+1}{k}} \to 0.
\end{equation}
On the other hand, $Q_1^{(k)}(z)=z$ and its norm is $\frac{k(k-1)}{k+1}$ according to \eqref{eq_Q_norm}. Hence,
\begin{multline} \label{eq_Q_first_diff}
 k^{-1/3}\left( \frac{Q^{(k)}_1(x_{k+1-i}^k)}{\sqrt{\langle Q^{(k)}_0, Q^{(k)}_0\rangle_k}}-\frac{Q^{(k)}_0(x_{k+1-i}^k)}{\sqrt{\langle Q^{(k)}_0, Q^{(k)}_0\rangle_k}}\right)\\=
 k^{-1/3}\left( 2\sqrt{k}(1+o(1)) \sqrt{\frac{k+1}{k(k-1)}}-\sqrt{\frac{k+1}{k}}\right)
  = k^{-1/3} (1+o(1)).
\end{multline}
This means that $k^{-1/3} \frac{Q_m^{(k)}(x^{k+1-i}_k)}{\sqrt{\langle Q_{m}^{(k)},Q_{m}^{(k)}\rangle_k}}$ as a function of $y$ grows by $k^{-1/3}$ when $y$ is increased by $k^{-1/3}$ (near $y=0$). Thus, we have a match with unit derivative at $y=0$.
\end{proof}

\begin{proof}[Second proof of Theorem \ref{Theorem_Q_to_Airy}] The proof splits into two parts. First, we explain how to find the leading asymptotics giving the answer for a fixed $y=\frac{m}{k^{1/3}}\in (0,+\infty)$ and then we explain how to achieve the desired uniformity over all $y\in [0,+\infty)$.

\noindent {\bf Part 1.} We use the  contour integral representation \eqref{eq_Q_contour_statement} written as:
\begin{equation}
\label{eq_Q_contour}
 Q^{(k)}_m(x)=\frac{(k-m)_{m+1}}{2\pi \ii} \oint_0 v^{-k-1}\exp\left(-\tfrac{v^2}{2}+x v\right) \left[\int_0^v u^{k-m-1}\exp\left(\tfrac{u^2}{2}-xu\right) \d u\right] \d v.
\end{equation}
Throughout the proof we always assume that $x=x^{k}_{k+1-i}$ for some $i=1,2,\dots$.
 Note that
\begin{equation}
\label{eq_zero_integral}
 \frac{k!}{2\pi \ii} \oint_0 v^{-k-1}\exp\left(-\tfrac{v^2}{2}+x v\right) \d v= H_k(x)=0, \text{ at } x=x^{k}_{k+1-i}.
\end{equation}
Thus, the lower limit of the $u$--integral can be changed from $0$ to any other point without changing the value of the double integral. Let us change this limit to $1$ and then integrate by parts in \eqref{eq_Q_contour}. We get:
\begin{equation}
\label{eq_Q_contour_parts}
 Q^{(k)}_m(x)=-\frac{(k-m)_{m+1}}{2\pi \ii} \oint_0 \left[\int_{1}^v u^{-k-1}\exp\left(-\tfrac{u^2}{2}+x u\right) \d u \right] v^{k-m-1}\exp\left(\tfrac{v^2}{2}-xv\right)\, \d v.
\end{equation}
The transition from \eqref{eq_Q_contour} to \eqref{eq_Q_contour_parts} uses the fact that the internal $u$-integral is a meromorphic single-valued function of $v$, which follows from the independence of the value of the integral from the choice of integration path implied by \eqref{eq_zero_integral} (otherwise, integration by parts would have led to the appearance of an additional term).

The lower limit $1$ of the $u$--integral in \eqref{eq_Q_contour_parts} again can be changed to any other point (this time, because of $v^{k-m-1}\exp\left(\tfrac{v^2}{2}-xv\right)$ having no singularities in the complex plane leading to vanishing of its contour integrals). It is convenient for us to change this point to $-\infty$, leading to the final expression:
\begin{equation}
\label{eq_Q_contour_final}
 Q^{(k)}_m(x)=-\frac{(k-m)_{m+1}}{2\pi \ii} \oint_0 \left[\int_{-\infty}^v u^{-k-1}\exp\left(-\tfrac{u^2}{2}+x u\right) \d u \right] v^{k-m-1}\exp\left(\tfrac{v^2}{2}-xv\right)\, \d v.
\end{equation}
Next, we apply a version of the steepest descent method to the integral \eqref{eq_Q_contour_final}. This method guides us to deform the integration contour to pass through the critical points of the integrand and to localize the integration to neighborhoods of these points.

\begin{figure}[t]
\begin{center}
\includegraphics[width=0.47\linewidth]{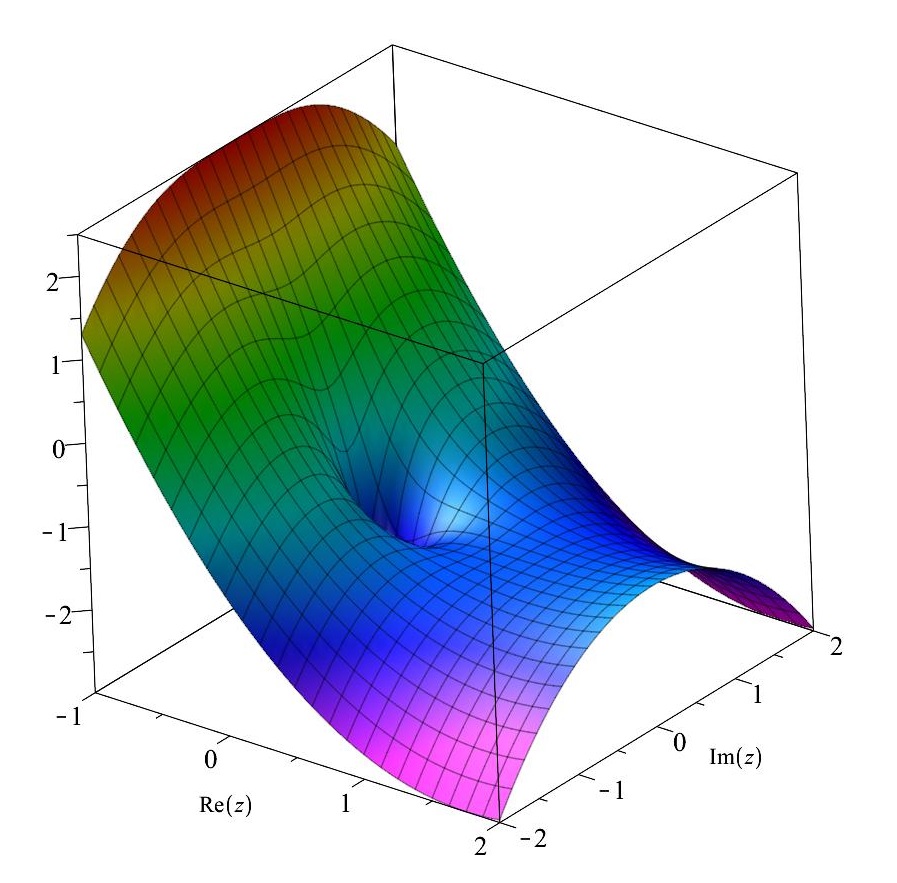}\quad
\includegraphics[width=0.47\linewidth]{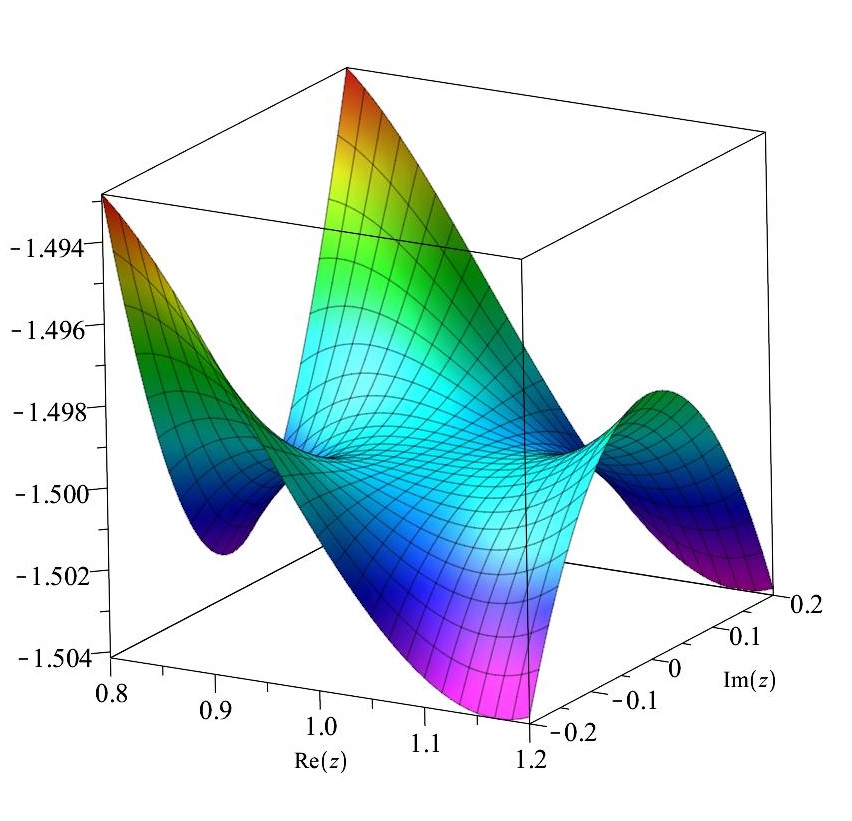}
 \caption{The graph of $\Re \hat F(z)$ of \eqref{eq_G_transformed} globally on the left panel and locally near the double critical point at $z=1$ on the right panel. \label{Fig_real_graph}}
\end{center}
\end{figure}

Denote
$$
 F(v):=-\ln\left(v^{-k} \exp\left(-\tfrac{v^2}{2}+2\sqrt{k} v\right)\right)=k \ln(v)+\tfrac{v^2}{2}-2\sqrt{k} v.
$$
Then using the asymptotic expansion \eqref{eq_x_edge} for $x$, the $u$--dependent part of the integrand in \eqref{eq_Q_contour_final} becomes
$$
 \frac{1}{u} \exp(-F(u))   \cdot \exp(k^{-1/6}(\a_i+o(1))u),
$$
and the remaining explicitly depending on $v$ factors in \eqref{eq_Q_contour_final} admit a similar representation in terms of $F(v)$.
While it might seem that $F$ changes with $k$, but, in fact, the dependence on $k$ is very simple:
\begin{equation}
\label{eq_G_transformed}
 F(v)=k \hat F (\hat v) + k\ln(\sqrt{k}), \quad  \hat F(\hat v)=\ln(\hat v)+ \frac{\hat v^2}{2} - 2 \hat v,\quad \hat v=\frac{v}{\sqrt k}.
\end{equation}
Thus, all the properties of $F(v)$ can be read from analyzing a single explicit function $\hat F(\hat v)$. Further, notice
$$
 F'(v)=\frac{k}{v}+v - 2\sqrt{k},\qquad F''(v)=-\frac{k}{v^2}+1, \qquad F'''(v)=2\frac{k}{v^3}.
$$
Hence, $v=\sqrt{k}$ is a double critical point of the function $F(v)$.
We are going to deform the $v$--integration contour to pass near this point, so that the asymptotic of the integral is given by the contribution of a small neighborhood of the critical point. It is helpful to take a look at the graph of $\Re F(v)$ before explaining the geometry of the contours and we refer to Figure \ref{Fig_real_graph}.

 \begin{figure}[t]
\begin{center}
{\scalebox{0.8}{\includegraphics{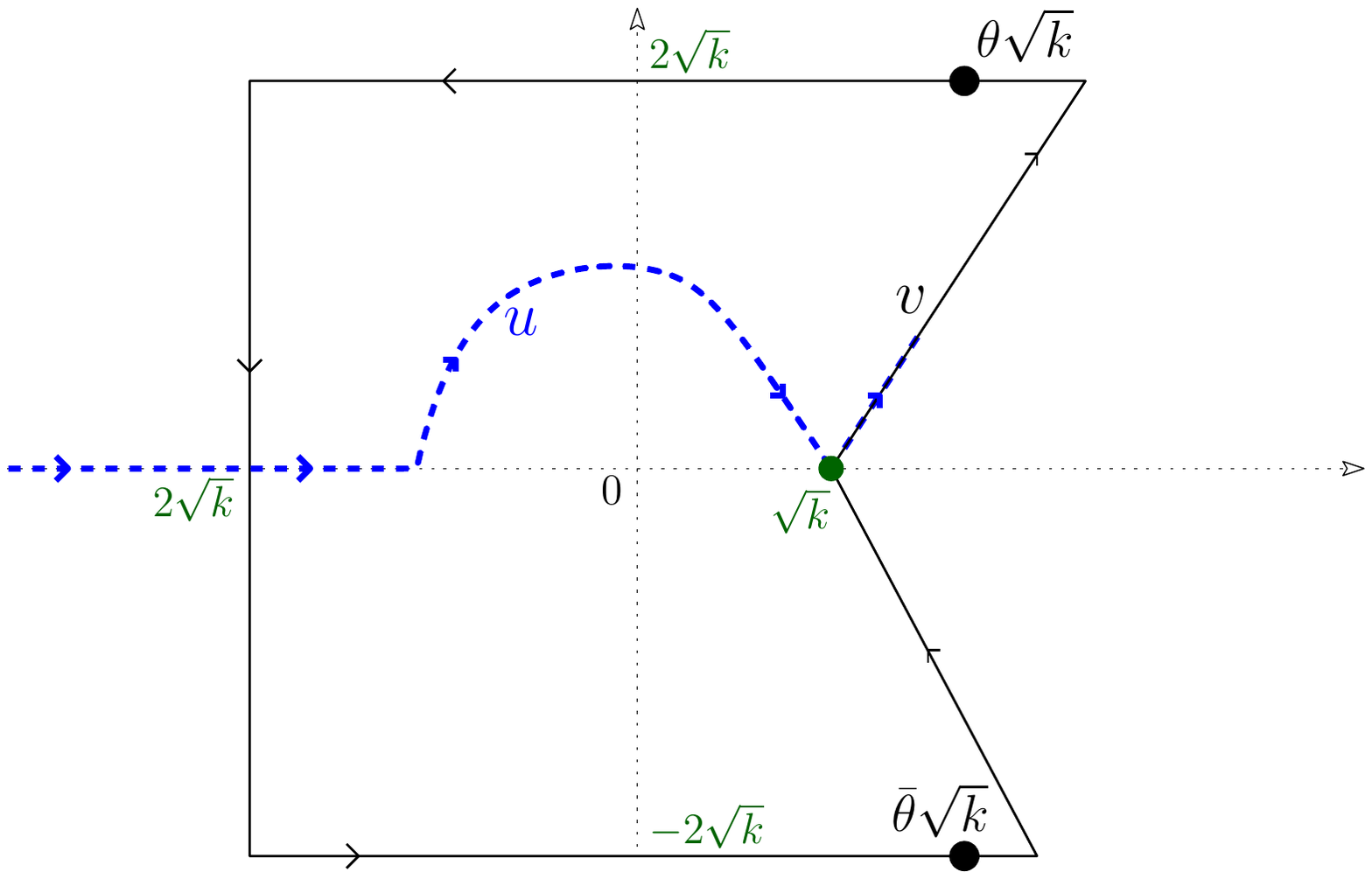}}}
 \caption{The $v$--contour is shown in solid black. The $u$--contour for points $v$ close to $\sqrt{k}$ is shown in dashed blue. The points $\theta \sqrt{k}$ and $\bar \theta\sqrt{k}$ give the minima of $\Re F(v)$ on the $v$--contour. \label{Fig_contours}}
\end{center}
\end{figure}

 The desired integration contours are shown in Figure \ref{Fig_contours}. The $v$--contour in the upper half-plane is chosen so that it starts from $\sqrt{k}$ under the angle $\tfrac{\pi}{3}$ and has growing $|v|$ as we move away from $\sqrt{k}$ until we reach the line $\Im(v)=2\sqrt{k}$, at which point the contour follows this line to the left until the point $v=-2+2\ii$ and then proceeds vertically till the real axis. In the lower-half plane the $v$--contour is given by the mirror image. Figure \ref{Fig_real_on_contours} shows the graph of $\Re F(v)$ (in the changed coordinates of \eqref{eq_G_transformed}) along the $v$--contour: the real part is minimized at points $\theta\sqrt{k}$, $\bar \theta\sqrt{k}$ and maximized at the intersections of the contour with the real axis.

\begin{figure}[t]
\begin{center}
\includegraphics[width=0.32\linewidth]{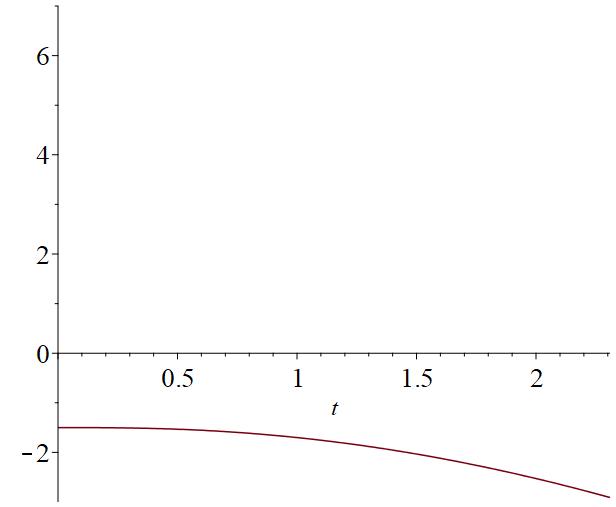}
\includegraphics[width=0.32\linewidth]{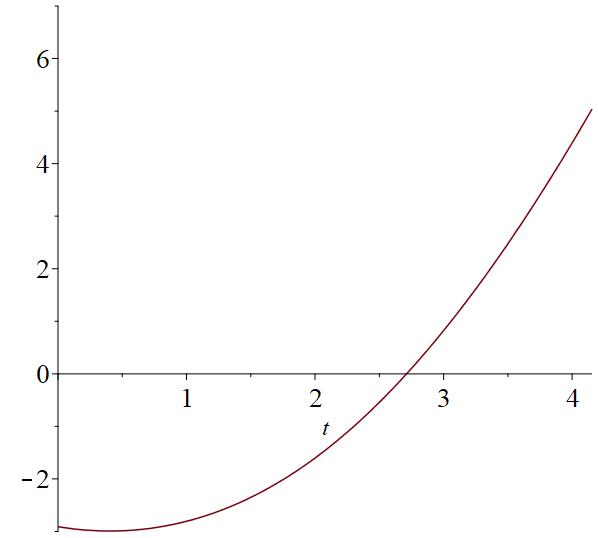}
\includegraphics[width=0.32\linewidth]{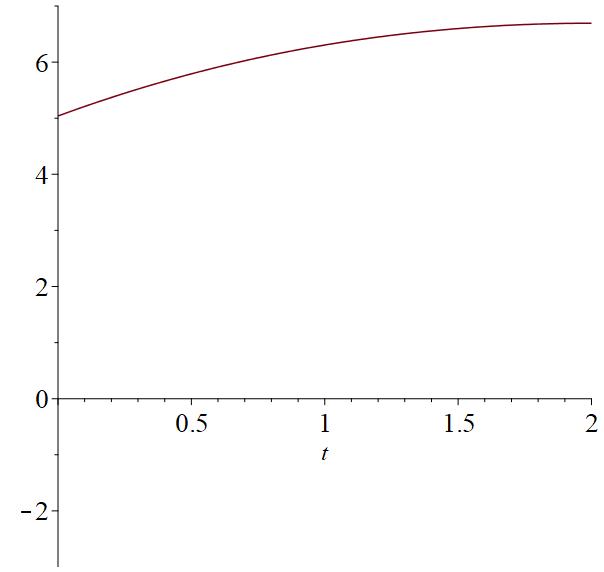}
 \caption{Three panels show the graph of $\Re\hat F(\hat v)$ with $\hat v = \frac{v}{\sqrt{k}}$ and $v$ as in Figure \ref{Fig_contours}. Left panel: $\Re \hat F(1+t\exp(\ii \pi/3))$. Middle panel: $\Re \hat F(1+\frac{2}{\sqrt{3}}-t+ 2\ii)$. Right panel: $\Re \hat F(-2+2\ii -t \ii)$. The minimum in the middle graph is attained at the point $\hat v= \theta$. \label{Fig_real_on_contours}}
\end{center}
\end{figure}

Further, when $v$ is on the part of the right part of the contour between $\bar \theta$ and $\theta$ (in particular, when it is close to $\sqrt{k}$), the $u$--contour (which we explain here in the reverse direction from $v$ to $-\infty$) starts from $v$ and first follows the $v$--contour to the point $\sqrt{k}$, then it continues from $\sqrt{k}$ under the angle $\tfrac{2\pi} 3$ by another level line $\Im(F(z))=0$ until it gets back to the real axis far left from the origin, at which point it proceeds to $-\infty$ via a horizontal line. When $v$ is on the left part of the contour, we instead follow the $v$--contour to the point $-2\sqrt{k}$ and then continue to $-\infty$.

The choice of the contours achieves the following goal: the absolute value of the $u$--integrand, which is
$$\left|\tfrac{1}{u}\cdot \exp(-F(u)) \cdot \exp(k^{-1/6}(\a_i+o(1))u)\right|,$$
 starts from being very close to $0$ when $u=\infty$, and then grows as we approach $v$ and has a sharp extremum near $v$. Hence, the absolute value of the $v$--integral can be upper bounded by  $\left|\tfrac{1}{v}\exp(-F(v))\exp(k^{-1/6}(\a_i+o(1))v)\right|$. This implies that the $v$--integrand is upper-bounded by $|v|^{-m}$ and, therefore, since $|v|$ is minimized near $\sqrt{k}$, the integrand is sharply decaying as $v$ moves away from $\sqrt{k}$. In more details, the part of the $v$--integral outside $\eps\sqrt{k}$--neighborhood of $\sqrt{k}$ is upper bounded by
 \begin{equation}
 \label{eq_x23}
  {\rm const}\cdot \sqrt{k} \cdot k^{-m/2} \left(1+\frac{\eps}{2}\right)^{-m},
 \end{equation}
 where $\sqrt{k}$ factor arises from the length of the integration contour. Since we are interested in the regime when $m$ is proportional to $k^{1/3}$, \eqref{eq_x23} is exponentially small compared to the leading contribution which comes next.

 The overall conclusion is that the integral is dominated by the contribution of a small neighborhood of the point $\sqrt{k}$. We can Taylor expand $F(v)$ in such a neighborhood:
$$
 F(v)=F(\sqrt{k})+\frac{1}{3 \sqrt{k}}(v-\sqrt{k})^3+ O\left(\frac{(v-\sqrt{k})^4}{k} \right).
$$
We further introduce the new variables:
$$
 \tilde v=k^{-1/6} (v-\sqrt{k}), \qquad \tilde u=k^{-1/6}(u-\sqrt{k}).
$$
The contour integral \eqref{eq_Q_contour_final} then asymptotically behaves as
\begin{multline}
\label{eq_Q_contour_local}
\frac{(k-m)_{m+1}}{2\pi \ii} \int_{e^{-\frac{\pi\ii}3}\infty}^{e^{\frac{\pi\ii}3}\infty} \Biggl[\int_{e^{\frac{2\pi \ii }{3}} \infty}^{\tilde v} \exp\left(-\tfrac{1}{3} \tilde u^3+ \a_i \tilde u + \a_i \sqrt{k}\right) \frac{k^{1/6}  \d \tilde u}{\sqrt{k}+k^{1/6}\tilde u} \Biggr] \\ \times
\exp\left(\tfrac{1}{3} \tilde v^3- \a_i \tilde v -\a_i \sqrt{k} \right) (\sqrt{k}+k^{1/6} \tilde v)^{-y k^{1/3}-1}  \, k^{1/6}  \d \tilde v.
\end{multline}
Equivalently, this is
\begin{equation}
\label{eq_Q_contour_local_final}
\frac{(k-m)_{m+1}}{k^{1/6} \cdot k^{\frac{m+1}{2}}} \frac{1}{2\pi \ii} \int_{e^{-\frac{\pi\ii}3}\infty}^{e^{\frac{\pi\ii}3}\infty}\Biggl[ \int_{e^{\frac{2\pi \ii }{3}} \infty}^{\tilde v} \exp\left(-\tfrac{1}{3} \tilde u^3+ \a_i \tilde u \right) \d \tilde u \Biggr]
\exp\left(\tfrac{1}{3} \tilde v^3- \a_i \tilde v - y\tilde v \right)   \d \tilde v.
\end{equation}
In the last integral the $\tilde v$--contour is an upwards-directed contour which is the union of the lines $\{ e^{-\ii\pi/3}t : t \geq 0\}$ and $\{ e^{\ii\pi/3}t : t \geq 0\}$. The internal $\hat u$--integral has quickly growing integrand and, therefore, it is dominated by the end-point $\tilde v$ giving the value $\approx \exp\left(-\tfrac{1}{3} \tilde v^3+ \a_i \tilde v\right)$, which cancels. As a result, the integrand is exponentially decaying in $\tilde v$ for each $y>0$.
 Combining with an explicit expression for $\langle Q^{(k)}_m, Q^{(k)}_m\rangle_k$ of \eqref{eq_Q_norm} we conclude that the left-hand side of \eqref{eq_Q_to_Airy} converges as $k\to\infty$ to
 \begin{equation}
\label{eq_Airy_funny_form_alt_final}
 -\frac{1}{2\pi \ii} \int_{e^{-\ii\pi/3}\infty}^{e^{\ii\pi/3} \infty} \Biggl[ \int_{e^{2\ii \pi/3}\infty}^{\tilde v}\exp\left(-\tfrac{1}{3} \tilde u^3+ \a_i \tilde u \right) \d \tilde u \Biggr]
\exp\left(\tfrac{1}{3} \tilde v^3- \a_i \tilde v - y\tilde v \right)    \d \tilde v.
\end{equation}
It remains to identify the last double integral with $\frac{\Ai(\a_i+ y)}{\Ai'(\a_i)}$. For that we analyze the double contour integral as a function of $y$. Let us denote this function through $\mathcal A(y+\a_i)$.

Let us apply the Airy operator to \eqref{eq_Airy_funny_form_alt_final}, i.e.\ we compute $\frac{\partial^2}{\partial y^2}\mathcal A(y+\a_i)-(\a_i+y)\mathcal A(y+\a_i)$, getting
 \begin{equation}
\label{eq_Airy_operator_applied}
 -\frac{1}{2\pi \ii} \int_{e^{-\ii\pi/3}\infty}^{e^{\ii\pi/3} \infty} \Biggl[ \int_{e^{2\ii \pi/3}\infty}^{\tilde v}\exp\left(-\tfrac{1}{3} \tilde u^3+ \a_i \tilde u \right) \d \tilde u \Biggr]
(\tilde v^2 - \a_i - y)\exp\left(\tfrac{1}{3} \tilde v^3- \a_i \tilde v - y\tilde v \right)    \d \tilde v.
\end{equation}
We now recognize $\frac{\partial}{\partial \tilde v} \exp\left(\tfrac{1}{3} \tilde v^3- \a_i \tilde v - y\tilde v \right)$ and can integrate by parts, noticing that
$$
 \Biggl[ \int_{e^{2\ii \pi/3}\infty}^{\tilde v}\exp\left(-\tfrac{1}{3} \tilde u^3+ \a_i \tilde u \right) \d \tilde u \Biggr]
\exp\left(\tfrac{1}{3} \tilde v^3- \a_i \tilde v - y\tilde v \right)
$$
vanishes at both infinities by the previous arguments. We get
 \begin{equation}
\label{eq_Airy_operator_applied_by_parts}
 \frac{1}{2\pi \ii} \int_{e^{-\ii\pi/3}\infty}^{e^{\ii\pi/3} \infty} \exp\left(-\tfrac{1}{3} \tilde v^3+ \a_i \tilde v \right)
\exp\left(\tfrac{1}{3} \tilde v^3- \a_i \tilde v - y\tilde v \right)    \d \tilde v,
\end{equation}
which is $0$. In addition it is clear from \eqref{eq_Airy_funny_form_alt_final} that $\lim_{y\to+\infty}\mathcal A(y)=0$, since the integrand is fast converging to zero.

We conclude that $\mathcal A(y)$ is a solution to the Airy equation, which vanishes at $+\infty$. This implies (see, e.g., \cite{VS})
$$
 \mathcal A(y+\a_i)= c\cdot \Ai(y+\a_i)
$$
for some constant $c\in\mathbb R$. This constant is fixed by the argument of Remark \ref{Remark_Airy_integral}, as soon as we have uniformity of convergence in $m$ and the tail bound \eqref{eq_Q_uniform_bound} justifying the convergence of the sum \eqref{eq_Q_dual_norm} to the integral \eqref{eq_Airy_norm}. This finishes part 1 of the proof.

\bigskip

\noindent {\bf Part 2.} We now explain an extension of the argument of the first part giving the uniform convergence over $y\in[0,+\infty)$ and the tail bound \eqref{eq_Q_uniform_bound}.
Notice that in the previous arguments uniformity of the asymptotics for $y$ in compact subsets of $(0,+\infty)$ is obtained for free. Thus, we only need to investigate $y\to 0$ and $y\to \infty$ boundary points. We start from the latter.

For large $y=\frac{m}{k^{1/3}}$ we need to establish the uniform bound \eqref{eq_Q_uniform_bound}. For that the first step is to figure out a similar bound for the asymptotic expression \eqref{eq_Airy_funny_form_alt_final}\footnote{A much faster decay is known for the Airy function as its argument tends to $+\infty$, but it is harder to see from our formulas.}. Take a radius $1$ neighborhood around $0$. The part of the $\tilde v$ integral in \eqref{eq_Airy_funny_form_alt_final} outside this neighborhood decays exponentially fast as $y$ grows. Inside the neighborhood we can upper-bound the magnitude of the integral by
$$
     {\rm const} \cdot \left| \int_0^{1} \exp(-y \exp(\ii \pi/3) t)dt  \right|=O\left(\frac{1}{y}\right), \quad y\to \infty.
$$
Switching to the prelimit asymptotic expression given by \eqref{eq_Q_contour_local} and \eqref{eq_Q_contour_local_final}, notice that the prefactor (after dividing by $k^{1/3} \langle Q^{(k)}_m, Q^{(k)}_m\rangle_k^{1/2}$) is decreasing monotonously in $m$ and, therefore, we can ignore it for the large $m$ asymptotic upper bound. After getting rid of the prefactor, the only $m$--dependent factor in the integrand is
$$
      (1+k^{-1/3}\tilde v)^{-m}=
 (1+k^{-1/3}\tilde v)^{-y k^{1/3}}.
 $$
Hence, the prelimit expression is upper-bounded for large $y$ exactly in the same way as the limiting expression \eqref{eq_Airy_funny_form_alt_final}.

Proceeding to $y$ close to $0$ we need to explain that the expression \eqref{eq_Airy_funny_form_alt_final} is a convergent integral, i.e., that the $\tilde v$--integrand decays fast enough as $\tilde v$ goes to infinity along the integration contour. For that we use the following transformation (obtained integrating by parts) of the integral over a part of the real axis:
     \begin{multline}\label{eq_x22}
      \int_{\gamma}^{q} \exp(-\alpha u^3 - \beta u) du =  \int_{\gamma}^q \frac{1}{-3 \alpha u^2 -\beta} \cdot \frac{\partial}{\partial u} \left[\exp(-\alpha u^3 - \beta u)\right] dy\\ = \frac{\exp(-\alpha q^3 - \beta q)}{-3 \alpha q^2 - \beta}-\frac{\exp(-\alpha \gamma^3 - \beta \gamma)}{-3 \alpha q^2 - \beta}
      -\int_{\gamma}^q \frac{6 \alpha u}{ (3 \alpha u^2 +\beta)^2}  \exp(-\alpha u^3 - \beta u) dy.
     \end{multline}
In the part of $\tilde u$ integral in \eqref{eq_Airy_funny_form_alt_final} from $0$ to $\tilde v$, the direction of the integration is $\exp(\pm \ii \pi/3)$, and, therefore, the parameter $\alpha$ in the last formula is $\alpha=\frac{1}{3} \exp(\pm \ii \pi)=-\frac{1}{3}$. Hence, the integrands in \eqref{eq_x22} are fast growing in $u$ and the formula implies an upper bound on the integral of the form $O\left( \frac{1}{1+q^2} \exp(-\alpha q^3 - \beta q) \right)$.
The large positive real number $q$ corresponds to $|\tilde v|$ in \eqref{eq_Airy_funny_form_alt_final} and we conclude that the integrand in $\tilde v$--integral decays as $O(\frac{1}{1+|\tilde v|^2})$ or faster for any value of $y\ge 0$. Therefore, the integral is uniformly convergent in $y\ge 0$.

The next problem is that for small $y$ (or small $m$), we can no longer guarantee the exponential decay of \eqref{eq_x23}. Note that if $m>k^{\delta}$ for some small $\delta>0$, then $(1+\frac{\eps}{2})^{-m}$ is fast decaying and our arguments go through. Thus, it remains to study the case $m<k^{\delta}$, corresponding to very small positive values of $y$. Note that according to \eqref{eq_Q_to_Airy} we expect to see the Airy function at a point close to its zero $\a_i$ in the limit. Hence, we need to show that for $m<k^{\delta}$ the left-hand side of \eqref{eq_Q_to_Airy} converges to zero. Let us denote this left-hand side through $\mathfrak Q^{(k)}_m$. We now reexamine the equations which we developed in the first proof of Theorem \ref{Theorem_Q_to_Airy}. In particular, \eqref{eq_Q_boundary} and \eqref{eq_Q_first_diff} yield that
\begin{equation}
 \mathfrak Q^{(k)}_0=k^{-1/3}(1+o(1)), \qquad \mathfrak Q^{(k)}_1-\mathfrak Q^{(k)}_0=k^{-1/3}(1+o(1)),\qquad k\to\infty.
\end{equation}
The recurrence \eqref{eq_Q_recurrence_normalized} in the asymptotic form \eqref{eq_Q_recurrence_approximating} then implies the following bound valid for all $0< m < k^{1/3}$, in which $C>0$ is a constant that can be made explicit:
\begin{equation}
\label{eq_x25}
 k^{2/3} \left| \left(\mathfrak Q^{(k)}_{m+1}-\mathfrak Q^{(k)}_m\right)-\left(\mathfrak Q^{(k)}_{m}-\mathfrak Q^{(k)}_{m-1}\right) \right|\le  C\cdot \left(|\mathfrak Q^{(k)}_{m+1}|+ |\mathfrak Q^{(k)}_{m}| +|\mathfrak Q^{(k)}_{m-1}|\right).
\end{equation}
We now show that the following two inequalities hold for all large enough $k$ and all $0<m<k^{1/6}$.
\begin{equation}
\label{eq_x24}
 |\mathfrak Q^{(k)}_m|< 2\cdot (m^2+1) \cdot k^{-1/3}, \quad |\mathfrak Q^{(k)}_m-\mathfrak Q^{(k)}_{m-1}|< (m+1) \cdot k^{-1/3}.
\end{equation}
We prove \eqref{eq_x24} by induction in $m$. For $m=1$ this is implied by \eqref{eq_Q_boundary} and \eqref{eq_Q_first_diff}. Suppose that the statement holds up to some value of $m$ and let us prove it for $m+1$.  Using \eqref{eq_x25} we write
\begin{multline*}
 |\mathfrak Q^{(k)}_{m+1}-\mathfrak Q^{(k)}_m|\le |\mathfrak Q^{(k)}_{m}-\mathfrak Q^{(k)}_{m-1}|+ C k^{-2/3} \left(|\mathfrak Q^{(k)}_{m+1}|+ |\mathfrak Q^{(k)}_{m}| +|\mathfrak Q^{(k)}_{m-1}|\right)
 \\ \le  |\mathfrak Q^{(k)}_{m}-\mathfrak Q^{(k)}_{m-1}|+ C k^{-2/3} \left(|\mathfrak Q^{(k)}_{m-1}|+ 2 |\mathfrak Q^{(k)}_{m}|\right) + C k^{-2/3} |\mathfrak Q^{(k)}_{m+1}-\mathfrak Q^{(k)}_m|.
\end{multline*}
Hence, for large $k$
\begin{multline*}
 |\mathfrak Q^{(k)}_{m+1}-\mathfrak Q^{(k)}_m|\le (1-C k^{-2/3})^{-1} \left[|\mathfrak Q^{(k)}_{m}-\mathfrak Q^{(k)}_{m-1}|+ C k^{-2/3} \left(|\mathfrak Q^{(k)}_{m-1}|+ 2 |\mathfrak Q^{(k)}_{m}|\right) \right]
 \\ \le (1-C k^{-2/3})^{-1} \left[(m+1)k^{-1/3}+ \frac{1}{2} k^{-1/3} \right]\le (m+2) k^{-1/3}.
\end{multline*}
Simultaneously,
$$
 |\mathfrak Q^{(k)}_{m+1}|\le |\mathfrak Q^{(k)}_{m+1}|+  |\mathfrak Q^{(k)}_{m+1}-\mathfrak Q^{(k)}_m| \le 2\cdot (m^2+1) \cdot k^{-1/3}+ (m+2)\cdot  k^{-1/3} \le 2\cdot ( (m+1)^2+1)\cdot  k^{-1/3},
$$
which finishes the proof of \eqref{eq_x24}. Since \eqref{eq_x24} implies that $\lim_{k\to\infty} |\mathfrak Q^{(k)}_m|=0$ uniformly over $0 \le m \le k^{1/6-\gamma}$ for any $\gamma>0$, we are done.
\end{proof}

\subsection{Proof of Theorem \ref{Theorem_Gcorners_limit_intro}}
We deal with the consecutive $N\to\infty$, $\beta\to\infty$ limit and compute the latter first, as in Section \ref{Section_GinftyE}. The $\beta\to\infty$ limit is already a Gaussian process, hence, it remains to study the behavior of its covariance as $N\to\infty$. For that we are going to pass to the limit in the formula for the covariance of \eqref{eq_covariance_zeta}. Let us first simplify it by plugging in the expressions for the weight and norm from Sections \ref{Section_Hermite_Laguerre_Jacobi} and \ref{Section_duality}. We have:
\begin{multline}
\label{eq_covariance_zeta_simp}
\Cov (\zeta_{a_1}^{k_1}, \zeta_{a_2}^{k_2})=  \frac{2}{\sqrt{k_1+1} \sqrt{k_2+1}} \sum_{\ell=\max(k_1,k_2)}^{\infty} \sum_{m=0}^{\min(k_1,k_2)-1}  \frac{Q^{(k_1)}_{m}(x^{k_1}_{a_1})}{\sqrt{\langle Q^{(k_1)}_{m}, Q^{(k_1)}_{m}\rangle_{k_1}}} \frac {Q^{(k_2)}_{m}(x^{k_2}_{a_2})}{\sqrt{\langle Q^{(k_2)}_{m}, Q^{(k_2)}_{m}\rangle_{k_2}}}
 \\ \times  \frac{(\ell-m)_{m+1}}{ (\ell+1)  \cdot  \sqrt{(k_1-m)_{m+1}} \sqrt{(k_2-m)_{m+1}}}
  \prod_{j=k_1}^{\ell-1} \left(1-\tfrac{m+1}{j+1}\right) \prod_{j=k_2}^{\ell-1} \left(1-\tfrac{m+1}{j+1}\right).
\end{multline}
Next we would like to study the asymptotics of the last line in \eqref{eq_covariance_zeta_simp} in the regime\footnote{We omit integer parts in order to shorten the notations.}:
\begin{equation}
\label{eq_edge_scaling_proof}
 k_1 = N+ 2 N^{2/3} t_1, \quad k_2=N+ 2 N^{2/3} t_2, \quad \ell= N+ 2 N^{2/3} \lambda, \quad m= y N^{1/3}, \quad N\to \infty.
\end{equation}
We write using $\ln(1+u)=u+O(u^2)$ and the notation $f\approx g$ whenever the ratio $f/g$ tends to $1$:
\begin{multline*}
 (\ell-m)_{m+1}=\ell^{m+1} \prod_{i=1}^{m} \left(1-\frac{i}{\ell}\right)=\ell^{m+1} \exp\left(-\sum_{i=1}^m  \frac{i}{\ell} + O\left(\frac{m^3}{\ell^2}\right) \right)\\=\ell^{m+1}\exp\left( O\left(\frac{m^2}{\ell}\right)+ O\left(\frac{m^3}{\ell^2}\right)\right)\approx \ell^{m+1}
 =N^{m+1}\left(1+\frac{2 \lambda}{N^{1/3}}\right)^{ y N^{1/3}+1}\\  \approx N^{m+1} \exp(2 y \lambda).
\end{multline*}
Similarly, we have
$$
 \sqrt{(k_1-m)_{m+1}} \approx N^{\frac{m+1}{2}} \exp\left(y t_1\right), \qquad   \sqrt{(k_2-m)_{m+1}} \approx N^{\frac{m+1}{2}} \exp\left(y t_2\right).
$$
Further,
$$
  \prod_{j=k_1}^{\ell-1} \left(1-\tfrac{m+1}{j+1}\right)=\exp\left(-\sum_{j=k_1}^{\ell-1} \frac{m+1}{j+1}  +O\left(\frac{(\ell-k_1) m^2}{k_1^2} \right)\right)\approx \exp\bigl( 2 y(t_1-\lambda) \bigr).
$$
And similarly
$$
  \prod_{j=k_2}^{\ell-1} \left(1-\tfrac{m+1}{j+1}\right) \approx \exp\bigl( 2 y(t_2-\lambda) \bigr).
$$
Altogether, the second line in \eqref{eq_covariance_zeta_simp} behaves as $N\to\infty$ as
$$
 \frac{1}{N} \exp\bigl( y(2 \lambda-t_1-t_2+2t_1-2\lambda+2t_2-2 \lambda) \bigr)=
 \frac{1}{N} \exp\bigl( y(t_1+t_2-2\lambda) \bigr).
$$
Summing over $\ell$ the second line in \eqref{eq_covariance_zeta_simp}, we see an approximation of a computable integral:
\begin{multline*}
\sum_{\ell=\max(k_1,k_2)}^{\infty}  \frac{(\ell-m)_{m+1}}{ (\ell+1)  \cdot  \sqrt{(k_1-m)_{m+1}} \sqrt{(k_2-m)_{m+1}}}
  \prod_{j=k_1}^{\ell-1} \left(1-\tfrac{m+1}{j+1}\right) \prod_{j=k_2}^{\ell-1} \left(1-\tfrac{m+1}{j+1}\right)
  \\ \approx 2 N^{-1/3} \int_{\max(t_1,t_2)}^{\infty}  \exp\bigl( y(t_1+t_2-2 \lambda) \bigr) d\lambda=
  N^{-1/3} \frac{1}{y}  \exp\bigl( -y|t_1-t_2| \bigr),
\end{multline*}
where the prefactor $2$ appears because of $2$ in \eqref{eq_edge_scaling_proof}.
 Further, the $m$--sum in \eqref{eq_covariance_zeta_simp} becomes as $N\to\infty$:
\begin{multline}
\label{eq_covariance_zeta_simp_final}
N^{1/3} \Cov (\zeta_{a_1}^{k_1}, \zeta_{a_2}^{k_2})\\ =  2 N^{-1/3} \sum_{m=0}^{\min(k_1,k_2)-1}  k_1^{-1/3} \frac{Q^{(k_1)}_{m}(x^{k_1}_{a_1})}{\sqrt{\langle Q^{(k_1)}_{m}, Q^{(k_1)}_{m}\rangle_{k_1}}} k_2^{-1/3} \frac {Q^{(k_2)}_{m}(x^{k_2}_{a_2})}{\sqrt{\langle Q^{(k_2)}_{m}, Q^{(k_2)}_{m}\rangle_{k_2}}} \frac{1}{y}  \exp\bigl( -y|t_1-t_2| \bigr).
 \end{multline}
 Plugging in $k_1=\kappa(t_1)$, $k_2=\kappa(t_2)$, $a_1=i$, $a_2=j$ and using Theorem \ref{Theorem_Q_to_Airy} we recognize a Riemann sum approximating as $N\to\infty$ the integral in the right-hand side of \eqref{eq_Edge_limit_covariance}. (The tail part corresponding to the large values of $m$ is being controlled by the uniform bound \eqref{eq_Q_uniform_bound}.) This finishes the proof of Theorem \ref{Theorem_Gcorners_limit_intro}.

\subsection{Random walk representation}
\label{Section_random_walk_representation}

Consider the matrix
$$
 P_t(i\to j)=\int_0^{\infty} \frac{\Ai(\a_i+y) \Ai(\a_j+y)}{\Ai'(\a_i) \Ai'(\a_j)} \exp(-ty)\, \d y, \qquad \a_i,\a_j \text{ are zeros of } \Ai(z).
$$

\begin{theorem} \label{Theorem_edge_semigroup}
 The matrices $P_t(i \to j)$, $t\ge 0$, $i,j\in\mathbb Z_{>0}$ form a stochastic semigroup, which means that:
 \begin{enumerate}
  \item $P_t(i\to j)\ge 0$ for each $t>0$ and $P_0(i,j)=\1_{i=j}$;
  \item For each $t\ge 0$
   \begin{equation}
   \label{eq_x26}
    \sum_{j=1}^\infty P_t(i\to j) =1;
   \end{equation}
   \item For each $t,s\ge 0$ and each $i,j\in\mathbb Z_{>0}$
   \begin{equation}
   \label{eq_semigroup_Pt}
    \sum_{q=1}^{\infty} P_t(i\to q) P_s(q\to j)=P_{t+s}(i\to j).
   \end{equation}
 \end{enumerate}
\end{theorem}
\begin{proof}
The proof is based on the combination of two ideas. First, $P_t(i\to j)$ is a limit of the diffusion kernels $K^{k,\ell}(a\to b)$ of Section \ref{Section_innovations}, which shows that it is non-negative. In principle, stochasticity and semigroup property might have been lost in the limit transition: the equalities \eqref{eq_x26} and \eqref{eq_semigroup_Pt} might have turned into inequalities. In order to rule out this possibility we find explicit eigenfunctions of $P_t(i\to j)$ with eigenvalues arbitrarily close to $1$.

{\bf Step 1.} Consider the Gaussian $\infty$--corners process with $x_i^k$ being the roots of the Hermite polynomials. Then \eqref{eq_Diffusion_spectral} yields an expression for the corresponding diffusion kernels:
$$
 K^{k,\ell}(a\to b)=
  \frac{1}{k+1} \sum_{m=0}^{k-1} \frac{Q^{(k)}_m(x^k_a) Q^{(\ell)}_m(x^\ell_b)}{\langle Q^{(k)}_m, Q^{(k)}_m \rangle_k} \prod_{j=k}^{\ell-1} \left(1-\tfrac{m+1}{j+1}\right).
$$
Set $\ell=k+\lfloor 2 t k^{2/3}\rfloor$, $a=k+1-i$, $b=\ell+1-j$ and send $k\to\infty$ in the last formula using Theorem \ref{Theorem_Q_to_Airy}, formula \eqref{eq_Q_norm} and
computation for $m\approx y k^{1/3}$
\begin{multline*}
 \sqrt{\frac{\langle Q^{(k)}_m, Q^{(k)}_m\rangle_\ell}{\langle Q^{(k)}_m, Q^{(k)}_m\rangle_k}}  \prod_{j=k}^{\ell-1} \left(1-\tfrac{m+1}{j+1}\right)=
 \sqrt{\frac{k+1}{m+1} \prod_{j=0}^m \left(1+\frac{\ell-k}{k-j}\right)} \prod_{j=k}^{\ell-1} \left(1-\tfrac{m+1}{j+1}\right)\\ \approx \exp(ty)\exp(-2ty)=\exp(-ty).
\end{multline*}
We get
\begin{equation}
\label{eq_semigroup_convergence}
 \lim_{k\to\infty} K^{k,k+\lfloor 2 t k^{2/3}\rfloor}(k+1-i\to \ell+1-j)= P_t(i\to j).
\end{equation}
Since the matrices $K^{k,\ell}(a\to b)$  are stochastic, we conclude that $P_t(i\to j)\ge 0$ and ${\sum_{j=1}^{\infty} P_t(i\to j)\le 1}$.

{\bf Step 2.} For $y\ge 0$ denote
$$
 \mathcal A_i(y)= \frac{\Ai(\a_i+y)}{\Ai'(\a_i)}.
$$
As functions of $y$ these are  eigenfunctions of the Sturm--Liouville operator corresponding to the Airy differential operator on $[0,+\infty)$ with Dirichlet boundary condition at $y=0$:
$$
\frac{\partial^2}{\partial y^2} \mathcal A_i(y)+ y \mathcal A_i(y)= \a_i \mathcal A_i(y), \quad y\ge 0; \qquad \mathcal A_i(0)=0.
$$
We also know that they are orthonormal (see Remark \ref{Remark_Airy_integral}):
 $$
 \int_0^{\infty} \mathcal A_i(y) \mathcal A_j(y) dy=\delta_{i=j}.
$$
The general theory of Sturm--Liouville expansions (see \cite[Section 2.7 and Section 4.12]{Tit} or \cite[Section 4.4]{VS}) yields that the functions $\mathcal A_i(y)$, $i=1,2,\dots,$ form a \emph{complete} orthonormal basis. In particular, we can expand function $\mathcal A_i\exp(-ty)$ in this basis, yielding
\begin{equation}
\label{eq_x27}
 \mathcal A_i(y) \exp(-ty)= \sum_{j=1}^{\infty} \mathcal A_j(y) \int_0^{\infty} \mathcal A_j (y) (\mathcal A_i(y) \exp(-ty)) \d y.
\end{equation}
Let us now change the point of view, fix some $y>0$ and treat $\mathcal A_i(y)$ as a function of $i$. Then \eqref{eq_x27} means that this is an eigenvector of the matrix $P_t(i\to j)$ with eigenvalue $\exp(-ty)$. Note that we can not take $y=0$ here, since $A_i(0)$ vanishes.

The definition of $\mathcal A_i$ implies that for each $i=1,2,\dots$
\begin{equation}
\label{eq_x29}
\lim_{y\to 0} \frac{1}{y}\mathcal A_i(y)=1.
\end{equation}
 In addition, there is a uniform bound:
\begin{equation}
\label{eq_x30}
  \lim_{y\to 0} \sup_{i\ge 1} \left| \frac{1}{y} \mathcal A_i(y)\right|=1,
\end{equation}
which follows from the known asymptotic expansions for $\Ai(x)$, $x\to-\infty$, and for $\Ai'(\a_i)$, $i\to\infty$, see, e.g., \cite[(2.48) and (2.58)]{VS}.

We can now apply \eqref{eq_x27} to get an inequality
\begin{equation}
\label{eq_x28}
 \left|\frac{1}{y}\mathcal A_i(y) \exp(-ty)\right|=\left|\frac{1}{y}\sum_{j=1}^\infty \mathcal A_j(y) P_t(i\to j)\right| \le \sup_{j\ge 1} \left|\frac{1}{y}\mathcal A_j(y)\right| \sum_{j=1}^{\infty} P_t(i\to j).
\end{equation}
Sending $y\to 0$ using \eqref{eq_x29} and \eqref{eq_x30} we conclude that
$ \sum_{j=1}^{\infty} P_t(i\to j)\ge 1$. Combining with the opposite inequality established on the first step we conclude that $\sum_{j=1}^{\infty} P_t(i\to j)=1$.

{\bf Step 3.} It remains to prove the semigroup property. By definition, it is satisfied by the matrices $K^{k,\ell}(a\to b)$ and we have
\begin{equation}
\label{eq_x31}
 \sum_{c=1}^{\ell} K^{k,\ell}(a\to c) K^{\ell, r}(c\to b)= K^{k,r}(a\to b).
\end{equation}
We now set  $\ell=k+\lfloor 2 t k^{2/3}\rfloor$, $r=\ell+\lfloor 2 s k^{2/3}\rfloor$, $a=k+1-i$, $c=\ell+1-q$, $b=r+1-j$ and send $k\to\infty$. Using \eqref{eq_semigroup_convergence} we see that the terms of the series \eqref{eq_x31} converge towards those of \eqref{eq_semigroup_Pt}. It remains to notice an asymptotic tail-bound: for any fixed $M$ we have
\begin{equation}
\label{eq_x32}
 \sum_{c=1}^{\ell-M} K^{k,\ell}(a\to c) K^{\ell, r}(c\to b)\le  \sum_{c=1}^{\ell-M} K^{k,\ell}(a\to c)=1-\sum_{c=\ell-M+1}^{\ell} K^{k,\ell}(a\to c)\to 1- \sum_{q=1}^{M} P_t(i\to q).
\end{equation}
Since $\sum_{q=1}^{\infty} P_t(i\to q)=1$, by choosing large enough $M$ we can make \eqref{eq_x32} arbitrarily small. Hence, $k\to\infty$ limit of \eqref{eq_x31} gives \eqref{eq_semigroup_Pt}.
\end{proof}

Let us consider a continuous time homogeneous Markov chain $\mathcal X_{(x_0)}(t)$, $t\ge 0$, taking values in the state space $\mathbb Z_{>0}$. The initial value is $x_0$, i.e.\ $\mathcal X_{(x_0)}(0)=x_0$. The transitional probabilities are given by $P_t$:
$$
 {\rm Prob }(X_{(x_0)}(t)=a)= P_t(x_0\to a).
$$

 Next, we take a countable collection of standard Brownian motions $W^{(i)}(t)$, $i\in\mathbb Z_{>0}$. For each $x\in\mathbb Z_{>0}$ and $t\in\mathbb R$ define a random variable
 $$
   \mathfrak Z(i,t)=2 \sum_{j=1}^{\infty} \int_t^{\infty} P_{r-t}(i\to j) \d W^{(j)}(r).
 $$
  An alternative expression for $\mathfrak Z(i,t)$ was given in \eqref{eq_Z_intro}. In words, we start the Markov chain $\mathcal X$ from $i$ at time $t$ and add the white noises $\dot{W}^{(i)}$ along its trajectory. $\mathfrak Z(i,t)$ is the expectation of the sum over the randomness coming from $\mathcal X$; it is still random variable with randomness coming from the Brownian motions. We can also view $\mathfrak Z(i,t)$ as the partition function of a directed polymer in additive Gaussian noise.

\begin{theorem} \label{Theorem_Z_as_polymer}
 The finite-dimensional distributions of $\mathfrak Z(i,t)$ are the same as ones of the limit in Theorems \ref{Theorem_Gcorners_limit_intro}, \ref{Theorem_DBM_limit_intro}, i.e.\ the covariance
 $\E  \mathfrak Z(i,t)  \mathfrak Z(j,s)$ matches the right-hand side of \eqref{eq_Edge_limit_covariance}.
\end{theorem}
\begin{proof} Since Ito integral is a $L_2$--isometry, we have
  \begin{multline}
   \E  \mathfrak Z(i,t)  \mathfrak Z(j,s)=
   4 \E \sum_{a=1}^{\infty} \int_t^{\infty} P_{r-t}(i\to a) \d W^{(a)}(r) \sum_{b=1}^{\infty} \int_s^{\infty} P_{r'-s}(i\to b) \d W^{(b)}(r')
   \\= 4 \int_{\max(t,s)}^{\infty}  \sum_{\ell=1}^{\infty} P_{r-t}(i\to  \ell) P_{r-s}(j\to \ell) \d r.
\end{multline}
Using the symmetry $P_t(x,y)=P_t(y,x)$ and the semigroup property  \eqref{eq_semigroup_Pt}, we compute the sum over $\ell$ and get
$$
 4 \int_{\max(t,s)}^{\infty} P_{2r-t-s}(i\to j) \d r  = 4 \int_{\max(t,s)}^{\infty} \d r \int_{0}^{\infty} \frac{\Ai(\a_i+y) \Ai(\a_j+y)}{\Ai'(a_i) \Ai'(a_j)} \exp(-(2r-t-s)y)\, \d y.
$$
Changing the order of integration and computing the $\d r$ integral we finally get
$$
2 \int_{0}^{\infty} \frac{\Ai(\a_i+y) \Ai(\a_j+y)}{\Ai'(\a_i) \Ai'(\a_j)} \exp\bigl(-(2\max(t,s)-t-s)y\bigr)\, \frac{\d y}{y}.\qedhere
$$
\end{proof}

Our next aim is to compute the intensities of the Markov chain $\mathcal X^{(x)}(t)$, matching its description at the end of Section \ref{Section_intro_as_results}.

\begin{proposition} \label{Proposition_semigroup_intensity} We have
\begin{equation}
 \frac{\partial}{\partial t} P_t(i\to j)\Bigr|_{t=0}=\begin{cases} \frac{2}{(\a_i-\a_j)^2}, & i\ne j,\\ \frac{2}{3}\a_i, & i=j. \end{cases}
\end{equation}
\end{proposition}
For the proof we need two computations of indefinite integrals.

\begin{lemma} Fix any $a\in\mathbb R$ and introduce the notation
$$
 \Ai_a=\Ai(y+a).
$$
Then we have
\begin{equation}
\label{eq_Airy_antider_1}
 \frac{\partial}{\partial y} \left( \frac{2a-y}{3} \Ai'_a\Ai'_a+\frac{1}{3} \Ai'_a \Ai_a  +\frac{(y+a)(y-2a)}{3}\Ai_a\Ai_a \right)= y \Ai_a \Ai_a,
\end{equation}
 Also for any $a,b\in \mathbb R$
\begin{multline}
\label{eq_Airy_antider_2}
 \frac{\partial}{\partial y} \Biggl( 2\Ai'_a \Ai'_b+ (a-b) (y \Ai'_a \Ai_b-y \Ai_a \Ai'_b)-2y\Ai_a\,\Ai_b-(a+b)\Ai_a\Ai_b + 2\frac{\Ai_a\Ai'_b-\Ai'_a\Ai_b}{b-a}\Biggr)\\=(a-b)^2 y \Ai_a \Ai_b.
\end{multline}
\end{lemma}
\begin{proof}
 The method for finding such identities is suggested in \cite{Albright}. The identities themselves are checked by direct differentiation using \eqref{eq_Airy_DE}.
 The left-hand side of \eqref{eq_Airy_antider_1} is transformed as follows:
\begin{multline*}
 \left(\frac{-1}{3} \Ai'_a\Ai'_a + 2(a+y) \frac{2a-y}{3} \Ai'_a\Ai_a\right) +\left(\frac{1}{3} (a+y) \Ai_a \Ai_a +\frac{1}{3} \Ai'_a \Ai'_a\right) \\ +\left(2 \frac{(y+a)(y-2a)}{3}\Ai'_a\Ai_a +\frac{y+a+y-2a}{3}\Ai_a\Ai_a \right)=y\Ai_a\Ai_a.
\end{multline*}
The left-hand side of \eqref{eq_Airy_antider_2} is transformed as follows:
\begin{multline*}
  \bigl(2(y+a)\Ai_a\Ai'_b+2(y+b)\Ai'_a\Ai_b \bigr)+ (a-b)\bigl(  (\Ai'_a \Ai_b-\Ai_a \Ai'_b) +y ( (y+a)\Ai_a \Ai_b- (y+b)\Ai_a \Ai_b)   \bigr)\\ -\bigl(2\Ai_a\Ai_b+2y\Ai'_a\Ai_b+2y\Ai_a\Ai'_b \bigr) -(a+b)\bigl(\Ai'_a\Ai_b+\Ai_a \Ai'_b\bigr) \\+\left( \frac{2}{b-a} \bigl( (y+b)\Ai_a\Ai_b-(y+a)\Ai_a\Ai_b\bigr) \right) =(a-b)^2 y \Ai_a \Ai_b.\qedhere
\end{multline*}
\end{proof}

\begin{proof}[Proof of Proposition \ref{Proposition_semigroup_intensity}]
 Differentiating under the integral sign,  we get
 \begin{equation}
  \frac{\partial}{\partial t} P_t(i\to j)\Bigr|_{t=0}= - \int_0^{\infty}y  \frac{\Ai(\a_i+y) \Ai(\a_j+y)}{\Ai'(\a_i) \Ai'(\a_j)} \, \d y.
 \end{equation}
 For the case $i=j$ we apply \eqref{eq_Airy_antider_1} converting the last expression into
 \begin{equation}
 \frac{1}{{\Ai'(\a_i) \Ai'(\a_i)}} \left( \frac{2\a_i-y}{3} \Ai'_{\a_i}\Ai'_{\a_i}+\frac{1}{3} \Ai'_{\a_i} \Ai_{\a_i}  +\frac{(y+\a_i)(y-2\a_i)}{3}\Ai_{\a_i}\Ai_{\a_i} \right)\Biggr|_{y=\infty}^{y=0}= \frac{2}{3}\a_i.
\end{equation}
When $i\ne j$, we apply \eqref{eq_Airy_antider_2} instead.
\end{proof}

We end this section by noting conservativity of the semigroup $P_t(i\to j)$, i.e., that the sum of its intensities over $j$ vanishes.

\begin{lemma} We have
$$
 \sum_{j\ge 1\, \mid\,  j\ne i} \frac{1}{(\a_i-\a_j)^2}= -\frac{1}{3} \a_i.
$$
\end{lemma}
\begin{proof}
 This is just one of many similar identities found in \cite{BR}. Alternatively, it can be proven as $k\to\infty$ limit of the identity of Lemma \ref{Lemma_sum_squares_as_second_derivative} specialized by \eqref{eq_weight_def} and \eqref{eq_Hermite_ortho_weight}
 $$
  \sum_{j=1}^k \frac{1}{(x^k_{k+1-j}-x^{k-1}_{k-i})^2}=k,
 $$
 where $x^k_a$ are the roots of the Hermite polynomials.
\end{proof}

\subsection{Holder-continuity of $\mathfrak Z(i,t)$}
\label{Section_continuity}

\begin{theorem}
 The process $\mathfrak Z(i,t)$ has a continuous modification, such that for each $i=1,2,\dots$ the process $\mathfrak Z(i,t)$ is almost surely a locally $\gamma$--Holder continuous function of $t$ for all $0<\gamma<\tfrac{1}{2}$.
\end{theorem}
\begin{proof}
 By the Kolmogorov continuity theorem (see e.g.\ \cite[Theorem 3.23]{Kal}) it suffices to check that for each $i=1,2,\dots$ and each $d=1,2,\dots$ there exists a constanct $C(i,d)$, such that
 \begin{equation}
 \label{eq_Kolm_crit}
  \E (\mathfrak Z(i,t)- \mathfrak Z(i,s))^{2d} \le C(i,d) |t-s|^{d}, \qquad t,s\in\mathbb R.
 \end{equation}
Because  $\bigl(\mathfrak Z(i,t)- \mathfrak Z(i,s)\bigr)$ is a mean $0$ Gaussian random variable, \eqref{eq_Kolm_crit} for $d=1$ implies it for all $d=2,3,\dots$. For $d=1$,
we recall the formula for the covariance obtained by substituting $i=j$ in \eqref{eq_Edge_limit_covariance}:
$$
\E \mathfrak Z(i,t) \mathfrak Z(i,s)= \frac{2}{[\Ai'(\a_i)]^2} \int_0^{\infty} [\Ai(\a_i+y)]^2 \exp\left(-|t-s| y\right) \frac{\d y}{y}.
$$
Using the inequality $\exp(-a)\ge 1-a$, valid for $a\ge 0$,
we get
$$
 \E \mathfrak Z(i,t) \mathfrak Z(i,s)\ge  \frac{2}{[\Ai'(\a_i)]^2} \int_0^{\infty} [\Ai(\a_i+y)]^2  \frac{\d y}{y} -  \frac{2 |t-s|}{[\Ai'(\a_i)]^2} \int_0^{\infty} [\Ai(\a_i+y)]^2  \d y.
$$
The first integral in the last formula is $\E  \mathfrak Z^2(i,t)=\E  \mathfrak Z^2(i,s)$ and the second integral is computed by \eqref{eq_squared_integral}. We conclude that
$$
  \E \mathfrak Z(i,t) \mathfrak Z(i,s)\ge  \E  \mathfrak Z^2(i,t)- 2 |t-s|.
$$
Hence,
$$
 \E (\mathfrak Z(i,t)-\mathfrak Z(i,s))^2\le 4 |t-s|,
$$
which implies \eqref{eq_Kolm_crit} for $d=1$. 
\end{proof}

\section{The $\beta=\infty$ Dyson Brownian Motion: proof of Theorem \ref{Theorem_DBM_limit_intro}}

The proof is split into two parts. First, we express the covariance of the $\beta\to\infty$ limit of the Dyson Brownian Motion (as in Theorem \ref{Theorem_DBM_beta_limit}) through the orthogonal polynomials $Q^k_i(x)$. Then we use the asymptotics of these polynomials established in Theorem \ref{Theorem_Q_to_Airy} to finish the proof. This section also contains the proofs of Lemma \ref{Lemma_zeta_SDE} and identity \eqref{eq_x38} (see Remark \ref{Remark_match_proof}).

\subsection{Covariance of the $\beta=\infty$ Dyson Brownian Motion} \label{Section_Inf_DBM_sol}

The aim of this section is to solve inhomogeneous linear equations \eqref{eq_DBM_infinity}. By the well-known algorithm for finding the solutions to inhomogeneous differential equations, we need to start by identifying $N$ linearly independent solutions to the homogeneous version of \eqref{eq_DBM_infinity}.

\begin{theorem} \label{Theorem_digonalization_DBM} Consider a linear $N$--dimensional system of differential equations
\begin{equation}
\label{eq_DBM_homo}
  \d z_i(t) = - \sum_{j \ne i} \frac{ z_i(t)- z_j(t)}{t(x_i^N-x_j^N)^2} \d t, \qquad t\ge 0,\quad i=1,2,\dots,N,
\end{equation}
where $x_i^N$ is the $i$th zero of the Hermite polynomial $H_N(x)$. Let $Q^{(m)}_N$ be the $m$--th orthogonal polynomial with respect to the uniform measure on $\{x_1^N,\dots,x_N^N\}$, as in Definition \ref{Definition_new_ortho}.
Then for each $m=0,1,2\dots,N-1$, the $N$--dimensional vector
\begin{equation}
 z_i(t)= t^{-m/2} \, Q^{(m)}_N\left( x^i_N\right), \quad i=1,2,\dots,N,
\end{equation}
is a solution to \eqref{eq_DBM_homo}.
\end{theorem}
\begin{remark}
 The statement of Theorem \ref{Theorem_digonalization_DBM} is closely related to that of Theorem \ref{Theorem_diagonalization_of_transition}.  In random matrix terminology, Theorem \ref{Theorem_diagonalization_of_transition} corresponds to the changing matrix size, while Theorem \ref{Theorem_digonalization_DBM} is about time evolution of a matrix of a fixed size. Our proofs of these theorems follow similar schemes: essentially we are showing that the dynamics \eqref{eq_DBM_homo} preserves both polynomiality and orthogonality with respect to the counting measure on the set $\{\sqrt{t} x_1^N, \sqrt{t} x_2^N, \dots, \sqrt{t} x_N^N\}$.
\end{remark}
\begin{proof}[Proof of Theorem \ref{Theorem_digonalization_DBM}] The statement would follow as soon as we show that
\begin{equation}
\label{eq_Q_eigen}
\frac{m}{2} Q^{(m)}_N\left(x_i^N\right)=\sum_{j \ne i} \frac{ Q^{(m)}_N\bigl(x_i^N\bigr)- Q^{(m)}_N\bigl( x_j^N\bigr)}{(x_i^N-x_j^N)^2}, \qquad  0 \le m \le N-1.
\end{equation}
In order to prove \eqref{eq_Q_eigen} we set $\mathbf L$ to be the linear operator in $N$--dimensional Euclidean (with respect to counting measure) space $\ell_2(x_1^N, x_2^N, \dots, x_N^N)$ with matrix $\frac{1}{(x_i^N-x_j^N)^2}$, $i,j=1,\dots,N$, in the standard coordinate basis. Let $\mathbf L_Q$ be the matrix of the same operator $\mathbf L$ in the orthonormal basis of functions $$
\frac{Q^{(0)}_N(x)}{\sqrt{(N+1)\langle Q^{(0)}_N,Q^{(0)}_N\rangle_N}},\, \frac{Q^{(1)}_N(x)}{\sqrt{(N+1)\langle Q^{(1)}_N,Q^{(1)}_N\rangle_N}},\,\dots, \frac{Q^{(N-1)}_N(x)}{{\sqrt{(N+1)\langle Q^{(N-1)}_N,Q^{(N-1)}_N\rangle}}}.
$$
 The relation \eqref{eq_Q_eigen} is readily implied by the following three properties that we will prove:
\begin{enumerate}
 \item The matrix $\mathbf L_Q$ is symmetric.
 \item The matrix $\mathbf L_Q$ is triangular.
 \item The diagonal elements of the matrix $\mathbf L_Q$ are $0,\frac{1}{2},\frac{2}{2},\frac{3}{2},\dots,\frac{N-1}{2}$.
\end{enumerate}
For the first property note that $\mathbf L$ is symmetric in standard coordinate basis. Hence, its matrix in any orthonormal basis is also symmetric and so is $\mathbf L_Q$. For the remaining two properties we fix $0\le m \le N-1$ and consider a function $R^{(m)}: \{x_1^N, x_2^N, \dots, x_N^N\}\to \mathbb R$ given by
 \begin{equation}
  R^{(m)}\left(x_i^N\right):= -\frac{m}{2} Q^{(m)}_N\left(x_i^N\right)+ \sum_{j \ne i} \frac{ Q^{(m)}_N\bigl(x_i^N\bigr)- Q^{(m)}_N\bigl( x_j^N\bigr)}{(x_i^N-x_j^N)^2}.
 \end{equation}
 The desired two properties of $\mathbf L_Q$ would follow immediately, if we manage to prove that $R^{(m)}$ is a polynomial of degree at most $m-1$ of real argument $x_i^N$, $i=1,2,\dots,N$. In fact, the exact nature of the polynomial $Q^{(m)}_N$ is irrelevant here. Expanding $Q^{(m)}_N$ into monomials, it suffices to check that the function
 \begin{equation}
 \label{eq_x17}
  x_i^N \mapsto -\frac{m}{2} \left( x_i^N\right)^m+ \sum_{j \ne i} \frac{ \bigl(x_i^N\bigr)^m - \bigl( x_j^N\bigr)^m}{(x_i^N-x_j^N)^2}, \quad i=1,2,\dots,N,
 \end{equation}
 is a polynomial of degree at most $m-1$. The last expression transforms into
 \begin{equation}
 \label{eq_x15}
   -\frac{m}{2} (x_i^N)^{m}+ \sum_{j\ne i} \frac{ (x_i^N)^{m-1}+(x_i^N)^{m-2} (x_j^N)+\dots+ (x_j^N)^{m-1}} {x_i^N-x_j^N}, \quad i=1,2,\dots,N.
 \end{equation}
 Let us use an identity which is implied by \eqref{eq_Hermite_var_eq}:
\begin{equation}
\label{eq_x33}
  \frac{m}{2} (x^N_i)^m-\sum_{j\ne i} \frac{m (x^N_i)^{m-1}}{x^N_i-x^N_j}=0, \quad i=1,2,\dots,N.
\end{equation}
Subtracting \eqref{eq_x33} from \eqref{eq_x15}, we convert the latter into
\begin{multline}
 \sum_{j\ne i} \frac{ [(x_i^N)^{m-1}-(x_i^N)^{m-1}]+[(x_i^N)^{m-2} (x_j^N)-(x_i^N)^{m-1}]+\dots+ [(x_j^N)^{m-1}-(x_i^N)^{m-1}]} {x_i^N-x_j^N}
 \\= -\sum_{j\ne i} \left(0+(x_i^N)^{m-2} + (x_i^N)^{m-3}  \left[ x_i^N + x_j^N\right]  + \dots + \left[(x_j^N)^{m-2}+ (x_j^N)^{m-3} (x_i^N)+\dots (x_i^N)^{m-2}\right]  \right),
\end{multline}
which is a (minus) sum of the expressions of the form
\begin{equation}
\label{eq_x16}
 (x_i^N)^{\ell} \sum_{j\ne i} (x_j^N)^{m-2-\ell}= (x_i^N)^{\ell} (p_{m-2-\ell}- x_i^{m-2-\ell}) ,
\end{equation}
where $0\le \ell \le m-2$ and $p_k=\sum_{j=1}^N (x_j^N)^k$. The expression \eqref{eq_x16} is a polynomial in $x_i^N$ of degree $m-2$, whose coefficients do not depend on $i$. Hence, \eqref{eq_x17} is a polynomial in $x_i^N$ of degree at most $m-2$ (which is even better than the degree at most $m-1$ that we wanted to have).
\end{proof}

We can now write down an explicit formula for the solution to \eqref{eq_DBM_infinity}.

\begin{theorem}
\label{Theorem_DBM_infty_solution}
 The system of SDEs \eqref{eq_DBM_infinity} is solved by
 \begin{equation}
 \label{eq_DBM_infinity_solution}
   \zeta_i^N(t)= \sqrt{2} \sum_{m=0}^{N-1} Q^{(m)}_N(x_i^N)  \sum_{j=1}^N \frac{Q^{(m)}_N (x_j^N)}{{(N+1)\bigl\langle Q^{(m)}_N, Q^{(m)}_N \bigr\rangle_N}}  \int_0^t \left(\frac{s}{t}\right)^{m/2} \d W_j(s),
 \end{equation}
 where the scalar product $\bigl\langle Q^{(m)}_N, Q^{(m)}_N \bigr\rangle_N$ is as in Definition \ref{Definition_new_ortho} and Corollary \ref{Corollary_Q_norm}, so that
 $$
  (N+1)\bigl\langle f,g\bigr\rangle_N = \sum_{a=1}^N f(x_a^N) g(x_a^N).
 $$
\end{theorem}
\begin{proof} Using the result of Theorem \ref{Theorem_digonalization_DBM} we have:
\begin{multline}
 \d \zeta_i^N(t)= \sqrt{2}\, \d \left[ \sum_{m=0}^{N-1} t^{-m/2} Q^{(m)}_N(x^i_N) \sum_{j=1}^N \frac{Q^{(m)}_N (x_j^N)}{{(N+1)\bigl\langle Q^{(m)}_N, Q^{(m)}_N \bigr\rangle_N}}  \int_0^t s^{m/2} \d W_j(s)\right]\\=
   \sqrt{2}\sum_{m=0}^{N-1} \d\left[t^{-m/2} Q^{(m)}_N(x_i^N)\right] \sum_{j=1}^N \frac{Q^{(m)}_N (x_j^N)}{{(N+1)\bigl\langle Q^{(m)}_N, Q^{(m)}_N \bigr\rangle_N}}  \int_0^t s^{m/2} \d W_j(s)\\+    \sqrt{2} \sum_{m=0}^{N-1} t^{-m/2} Q^{(m)}_N(x_i^N) \sum_{j=1}^N \frac{Q^{(m)}_N (x^j_N)}{{(N+1)\bigl\langle Q^{(m)}_N, Q^{(m)}_N \bigr\rangle_N}} \, \d \left[ \int_0^t s^{m/2} \d W_j(s)\right]
   \\= -\sum_{j\ne i} \frac{\zeta_i^N(t)-\zeta_j^N(t)}{t(x_i^N-x_j^N)^2} +\sqrt{2} \sum_{m=0}^{N-1} \sum_{j=1}^N   \frac{Q^{(m)}_N(x_i^N) Q^{(m)}_N (x_j^N)}{{(N+1)\bigl\langle Q^{(m)}_N, Q^{(m)}_N \bigr\rangle_N}}  W_j(t)
   \\= -\sum_{j\ne i} \frac{\zeta_i^N(t)-\zeta_j^N(t)}{t(x_i^N-x_j^N)^2} +\sqrt{2} \, W_j(t),
\end{multline}
where the last identity is obtained by changing the order of summation and using the fact that the matrix $(i,m)\mapsto \frac{Q^{(m)}_N(x_i^N)}{\sqrt{(N+1)\bigl\langle Q^{(m)}_N, Q^{(m)}_N \bigr\rangle_N}}$ is orthogonal.
\end{proof}

We further show that \eqref{eq_DBM_infinity_solution} is the unique solution of Lemma \ref{Lemma_zeta_SDE}.

\begin{proof}[Proof of Lemma \ref{Lemma_zeta_SDE}]
 In Theorem \ref{Theorem_DBM_infty_solution} we checked that \eqref{eq_DBM_infinity_solution} solves the SDE \eqref{eq_DBM_infinity}. It is also clear that this solution satisfies the initial condition: $\lim_{t\to 0} \zeta_i^N(t)=0$. Thus, it remains to check the uniqueness. Let $(\zeta_i^N(t))_{i=1}^N$ and $(\tilde \zeta_i^N(t))_{i=1}^N$ be two stochastic processes satisfying conditions of Lemma \ref{Lemma_zeta_SDE} with the same Brownian motions $(W_i(t))_{i=1}^N$. Then their difference solves a deterministic homogeneous linear differential equation:
 $$
 \d [\zeta_i^N-\tilde \zeta_i^N](t) = - \sum_{j \ne i} \frac{ [\zeta_i^N-\tilde \zeta_i^N](t)- [\zeta_j^N-\zeta_j^N](t)}{t(x_i^N-x_j^N)^2} \d t, \qquad t>0,\quad i=1,2,\dots,N.
 $$
 A complete basis of solutions of this equation was found in Theorem \ref{Theorem_digonalization_DBM}. None of the non-zero solutions tends to $(0,\dots,0)$ at $t\to 0$. Hence, $\zeta_i^N(t)-\tilde \zeta_i^N(t)$ must be almost surely equal to zero for all $i=1,\dots,N$ and all $t\ge 0$.
\end{proof}

\begin{lemma} \label{Lemma_DBM_infty_covariance} $\bigl(\zeta_i^N(t)\bigr)_{i=1}^N$, $t\ge 0$, of Theorem \ref{Theorem_DBM_infty_solution} is a mean $0$ Gaussian process with covariance
\begin{equation}
\label{eq_DBM_infty_covariance}
  \Cov( \zeta_i^N(t), \zeta_j^N(s))=2    \sum_{m=0}^{N-1} \frac{Q^{(m)}_N(x_i^N)  Q^{(m)}_N(x_j^N) }{(N+1) \bigl\langle Q^{(m)}_N, Q^{(m)}_N \bigr\rangle_N}\cdot  \frac{(\min(t,s))^{m+1}}{(m+1) (ts)^{m/2}}
\end{equation}
\end{lemma}
\begin{proof} Using the isometry property of stochastic integrals
$$
 \E  \int_0^t f(\tau) \d W_a(\tau) \int_0^s g(\sigma) \d W_b(\sigma)= \delta_{a=b} \int_0^{\min(t,s)} f(\tau) g(\tau) d\tau
$$
and \eqref{eq_DBM_infinity_solution}, we have
\begin{multline}
  \E \zeta_i^N(t) \zeta_j^N(s)=2 \E \Biggl[ \sum_{m=0}^{N-1} Q^{(m)}_N(x_i^N)  \sum_{a=1}^N \frac{Q^{(m)}_N (x_a^N)}{{(N+1)\bigl\langle Q^{(m)}_N, Q^{(m)}_N \bigr\rangle_N}}  \int_0^t \left(\frac{\tau }{t}\right)^{m/2} \d W_a(\tau) \\ \times \sum_{\ell=0}^{N-1} Q^{(\ell)}_N(x_j^N)  \sum_{b=1}^N \frac{Q^{(\ell)}_N (x_b^N)}{{(N+1)\bigl\langle Q^{(\ell)}_N, Q^{(\ell)}_N \bigr\rangle_N}}  \int_0^s \left(\frac{\sigma }{s}\right)^{\ell/2} \d W_b(\tau)  \Biggr]= \frac{2}{(N+1)^2}\\ \times \E \Biggl[ \sum_{m=0}^{N-1} \sum_{\ell=0}^{N-1} Q^{(m)}_N(x_i^N)  Q^{(\ell)}_N(x_j^N)    \sum_{a=1}^N \frac{Q^{(m)}_N (x_a^N)}{{\bigl\langle Q^{(m)}_N, Q^{(m)}_N \bigr\rangle_N}}  \frac{Q^{(\ell)}_N (x_a^N)}{{\bigl\langle Q^{(\ell)}_N, Q^{(\ell)}_N \bigr\rangle_N}}  \int_0^{\min(t,s)} \left(\frac{\tau}{t}\right)^{m/2} \left(\frac{\tau}{s}\right)^{\ell/2} \d \tau  \Biggr].
\end{multline}
It remains to compute the $\tau$--integral and to use the orthogonality relation
$$
  \frac{1}{N+1}\sum_{a=1}^N Q^{(m)}_N (x_a^N) Q^{(\ell)}_N (x_a^N) = \delta_{m=\ell} \cdot \bigl\langle Q^{(m)}_N, Q^{(m)}_N \bigr\rangle_N. \qedhere
$$
\end{proof}

\begin{corollary} \label{Corollary_two_forms_of_covariance} The fixed $t$ covariance of the process $\bigl(\zeta_i^N(t)\bigr)_{i=1}^N$ of Theorem \ref{Theorem_DBM_beta_limit} (equivalently, of the Gaussian vector of \eqref{eq_DBM_one_point_limit}) is given by
\begin{equation}
\label{eq_GinftyE_covariance_final}
  \Cov( \zeta_i^N(t), \zeta_j^N(t))=\frac{2  t}{N+1}   \sum_{m=0}^{N-1} \frac{Q^{(m)}_N(x_i^N)  Q^{(m)}_N(x_j^N) }{(m+1) \bigl\langle Q^{(m)}_N, Q^{(m)}_N \bigr\rangle_N}.
\end{equation}
At $t=1$ the same formula also computes the covariance $ \Cov( \zeta_i^N, \zeta_j^N)$ for the double infinite sum \eqref{eq:zeta-K} of Theorem \ref{Theorem_GinftyE_RW}.
\end{corollary}
\begin{remark}
\label{Remark_match_proof}
Comparing \eqref{eq:zeta_projection} with \eqref{eq_DBM_one_point_limit}, we conclude that the left-hand side of \eqref{eq_x38} coincides with the variance of $\zeta_i^N(1)$. Hence, the last statement of Corollary \ref{Corollary_two_forms_of_covariance} implies \eqref{eq_x38}.
\end{remark}
\begin{remark}
  The formula \eqref{eq_GinftyE_covariance_final} was also proven in \cite[Theorem 3.1]{AV}: the proof there is based on an explicit diagonalization of the quadratic form in the exponent of \eqref{eq_x38}.
\end{remark}

\begin{proof}[Proof of Corollary \ref{Corollary_two_forms_of_covariance}]
The formula \eqref{eq_GinftyE_covariance_final} is obtained by substituting $t=s$ into \eqref{eq_DBM_infty_covariance}.

On the other hand, the covariance of the infinite sum \eqref{eq:zeta-K} is computed by setting $k_1=k_1=N$ in \eqref{eq_covariance_zeta}. Using \eqref{eq_Hermite_ortho_weight}, it becomes:
  \begin{multline}
\label{eq_covariance_zeta_repeat}
\Cov (\zeta_{a_1}^{N}, \zeta_{a_2}^{N})\\= \frac{2}{(N+1)^2} \sum_{\ell=N}^{\infty} \sum_{m=0}^{N-1}  Q^{(N)}_{m}(x^{N}_{a_1}) \, Q^{(N)}_{m}(x^{N}_{a_2})   \frac{\langle Q^{(\ell)}_{m}, Q^{(\ell)}_{m} \rangle_\ell}{\langle Q^{(N)}_{m}, Q^{(N)}_{m} \rangle_{N} \langle Q^{(N)}_{m}, Q^{(N)}_{m} \rangle_{N}}
 \prod_{j=N}^{\ell-1} \left(1-\tfrac{m+1}{j+1}\right)^2.
\end{multline}
 In order to match  \eqref{eq_covariance_zeta_repeat} with \eqref{eq_GinftyE_covariance_final} at $t=1$, we interchange the order of the summations in the former and compute the sum  $\sum_{\ell=N}^{\infty}$ for each $0\le m \le N-1$, using the explicit formula for $\langle Q^{(\ell)}_{m}, Q^{(\ell)}_{m} \rangle_\ell$ from Corollary \ref{Corollary_Q_norm} and the Pochammer symbol notation ${(x)_n=x(x+1)\cdots (x+n-1)}$:
 \begin{multline}
 \label{eq_x36}
   \sum_{\ell=N}^{\infty}  \langle Q^{(\ell)}_{m}, Q^{(\ell)}_{m} \rangle_\ell
 \prod_{j=N}^{\ell-1} \left(1-\tfrac{m+1}{j+1}\right)^2\\ =  \sum_{\ell=N}^{\infty}  \frac{\ell(\ell-1)\cdots (\ell-m)}{\ell+1} \cdot \left(\frac{(N-m)(N+1-m)\cdots (\ell-m-1)}{(N+1)(N+2)\cdots \ell} \right)^2\\
=  \sum_{\ell=N}^{\infty}  \frac{(N-m)(N-m+1)\cdots (\ell-1)\ell}{(N+1)(N+2)\cdots(\ell+1)}
 \cdot   \frac{(N-m)(N+1-m)\cdots (\ell-m-1)}{(N+1)(N+2)\cdots \ell}
\\ = \frac{(N-m)\cdots N}{N+1} \sum_{\ell=N}^{\infty}   \frac{(N-m)_{\ell-N}}{(N+2)_{\ell-N}}= \frac{(N-m)(N-m+1)\cdots N}{N+1}\,_2F_1(1, N-m; N+2; 1),
 \end{multline}
 where $_2F_1$ is the Gauss hypergeometric function. Its value can be computed using the Gauss's summation theorem:
 $$
  _2F_1(a,b;c;1)=\frac{\Gamma(c)\Gamma(c-a-b)}{\Gamma(c-a)\Gamma(c-b)}.
 $$
Hence, we further transform \eqref{eq_x36} into
$$
 \frac{(N-m)(N-m+1)\cdots N}{N+1} \cdot \frac{\Gamma(N+2)\Gamma(m+1)}{\Gamma(N+1)\Gamma(m+2)}=  \frac{(N-m)(N-m+1)\cdots N}{m+1}.
$$
Plugging the result back into \eqref{eq_covariance_zeta_repeat} and using Corollary \ref{Corollary_Q_norm} again, we arrive at \eqref{eq_GinftyE_covariance_final} with $t=1$, as desired.
\end{proof}

\subsection{Proof of Theorem \ref{Theorem_DBM_limit_intro}}

The theorem deals with the iterative limit $N\to\infty$, $\beta\to\infty$. The latter limit is  computed in Theorem \ref{Theorem_DBM_beta_limit}, it is a Gaussian process and we use the result of Lemma \ref{Lemma_DBM_infty_covariance}  for its covariance. It remains to send $N\to\infty$ in \eqref{eq_DBM_infty_covariance}, i.e.\ to compute the limit
\begin{multline}
\label{eq_x20}
 \lim_{N\to\infty} \E\bigl[ N^{1/3}  \zeta_{N+1-i}^N(1+  2t N^{-1/3} ) \zeta_{N+1-j}^N(1+   2s N^{-1/3})\bigr]=
 \\ \lim_{N\to\infty} 2 N^{1/3}   \sum_{m=0}^{N-1} \frac{Q^{(m)}_N(x^{N+1-i}_N)  Q^{(m)}_N(x^{N+1-j}_N) }{(N+1) \bigl\langle Q^{(m)}_N, Q^{(m)}_N \bigr\rangle_N}\cdot  \frac{(1+2N^{-1/3}\min(t,s))^{m+1}}{(m+1) (1+2N^{-1/3}t)^{m/2}(1+2N^{-1/3}s)^{m/2}}.
\end{multline}
We use Theorem \ref{Theorem_Q_to_Airy} to compute the asymptotic behavior of $Q^{(m)}_N(x^{N+1-i}_N)$ and  $Q^{(m)}_N(x^{N+1-j}_N)$, transforming \eqref{eq_x20} into
\begin{equation}
\label{eq_x21}
\lim_{N\to\infty} 2 \sum_{m=0}^{N-1} \frac{\Ai\bigl(\a_i+\frac{m}{N^{1/3}}\bigr) \Ai\bigl(\a_j+\frac{m}{N^{1/3}}\bigr)}{\Ai'(\a_i)\Ai'(\a_j)}\cdot  \frac{(1+2 N^{-1/3}\min(t,s))^{m+1}}{(m+1) (1+2 N^{-1/3}t)^{m/2}(1+2 N^{-1/3}s)^{m/2}}.
\end{equation}
The terms in the last sum rapidly decay as $\frac{m}{N^{1/3}}\to+\infty$. Hence, denoting $y=\frac{m}{N^{1/3}}$ and using the $N\to\infty$ asymptotic approximation
\begin{multline*}
\frac{(1+2 N^{-1/3}\min(t,s))^{m+1}}{(m+1) (1+2 N^{-1/3}t)^{m/2}(1+2 N^{-1/3}s)^{m/2}}\approx N^{-1/3} \frac{1}{y} \exp\left( 2y\min(t,s)-yt-ys\right)\\=
N^{-1/3} \frac{1}{y} \exp( -y|t-s|),
\end{multline*}
\eqref{eq_x21} becomes a Riemann sum approximating as $N\to\infty$ the integral
$$
 2 \int_0^{\infty} \frac{\Ai\bigl(\a_i+y\bigr) \Ai\bigl(\a_j+y\bigr)}{\Ai'(\a_i)\Ai'(\a_j)} \exp( -y|t-s|) \frac{\d y}{y},
$$
thus, matching \eqref{eq_Edge_limit_covariance} and finishing the proof.

\section{Appendix: steepest descent analysis}

\label{Section_steepest_descent}

\begin{proof}[Proof of Theorem \ref{Theorem_Hermite_limit}] Rescaling and shifing the variables $y_i$, we can (and will) assume without loss of generality that $\mu_N=0$ and $\sigma_N=1$.

We use the contour integral representation of the derivative to write
\begin{equation}
\label{eq_x8}
   P_k(y
)= {N\choose k}^{-1} \cdot \frac{1}{2\pi \ii }  \oint_0 \frac{P_N(z+y)}{z^{N-k+1}} dz,
 \end{equation}
 where the integral goes over a positively oriented loop enclosing $0$. We further set
 $$
  y=\frac{x}{\sqrt{N}} .
 $$
Our aim is to show that up to certain factors, which have no zeros, $P_k(\frac{x}{\sqrt{N}} )$ becomes the degree $k$ Hermite polynomial as $N\to\infty$. By the Hurwitz theorem, this would imply the desired convergence of zeros.

 Using \eqref{eq_x8} and adopting the notation $\sim$ to indicate an equality up to factors independent of $x$, we have
$$
  P_k\left( \frac{x}{\sqrt{N}} \right) \sim \oint_{0}   \prod_{i=1}^N \left(1+\frac{-y_i+  \frac{x}{\sqrt{N}} }{z} \right)  \frac{dz}{z^{1-k}}.
$$
Note that $|y_i|/\sqrt{N}\to 0$ uniformly in $i$ as $N\to\infty$ due to Assumption \ref{Assumption_Hermite}. Hence, we can
 change the variable $z=\frac{\sqrt{N}}{ w}$ and use the Taylor series expansion $\ln(1+q)=q-\frac{q^2}{2}+O(q^3)$ to get
\begin{multline*}
  P_k\left(\frac{x}{\sqrt{N}} \right) \sim \oint_{0}  \exp\left[   \sum_{i=1}^N \ln\left(1+\frac{w}{\sqrt N} \cdot \left(-y_i+  \frac{x}{\sqrt{N}} \right) \right)\right]  \frac{dw}{w^{k+1}}
  \\= \oint_{0}  \exp\Biggl[   \sum_{i=1}^N \left(\frac{w}{\sqrt N} \cdot \left(-y_i+  \frac{x}{\sqrt{N}} \right) \right)-\frac{1}{2}\sum_{i=1}^N \frac{w^2}{N} (-y_i)^2-\frac{1}{2} \sum_{i=1}^N \frac{w^2 x^2}{N^2} \\+\frac{1}{2}\frac{w^2}{N\sqrt{N}} x \sum_{i=1}^N y_i + \frac{1}{N\sqrt N}\sum_{i=1}^N O\left((-y_i)^3\right) +o(1)   \Biggr]  \frac{dw}{w^{k+1}}.
\end{multline*}
By Assumption \ref{Assumption_Hermite} and our choices of $\mu_N$ and $\sigma_N$
$$
 \sum_{i=1}^N y_i=0,\quad  \frac{1}{N}\sum_{i=1}^N (y_i)^2=\sigma_N=1,\quad \frac1{N \sqrt{N}} \sum_{i=1}^N  |y_i|^3= \frac{\kappa_N^3}{\sqrt N}=o(1).
$$
Hence, we conclude that after factoring out the $x$--independent constants $P_k(\frac{x}{\sqrt{N}})$ converges (uniformly over $x$ belonging to a compact subsets of the complex plane) to
$$
\frac{k!}{2\pi \ii} \oint_{0}  \exp\Biggl[   wx -\frac{w^2}{2}  \Biggr]  \frac{dw}{w^{k+1}},
$$
which is a known contour integral representation for the Hermite polynomial $H_k(x)$, see \cite[(9.15.10)]{KS}.
\end{proof}

\begin{proof}[Proof of Theorem \ref{Theorem_bulk_lattice}]
 Since we deal only with roots of the polynomials, but not with their values, we can and will omit
 various multiplicative constants. We would like to investigate the zeros of the function $P_k(x+
 \tfrac{1}{N}\chi)$ of a complex variable $\chi$ as $N\to\infty$. Using the contour integral representation of the derivative, we have
 $$
   P_k(x+
 \tfrac{1}{N}\chi)\sim \oint \frac{P_N(z+x)}{(z-\tfrac1N \chi)^{N-k+1}} dz,
 $$
 where the integration contour encloses the unique pole of the integrand at $z=\tfrac1N \chi$. We
 would like to apply the steepest descent method to the last integral. For that we write the
 integrand as
 \begin{equation}
\label{eq_x1}
 \exp( N G(z)) \cdot \left(1-\frac\chi{N z}\right)^{-N+k-1},
 \end{equation} where
 $$
  G(z)=\frac{1}{N}\ln\bigl( P_N(z+x)\bigr) - \frac{N-k+1}{N} \ln z.
 $$
 The second factor in \eqref{eq_x1} converges as $N\to\infty$, and we are led to study the first
 oscillating factor. The steepest descent method suggests to deform the contours of integration so that they pass through the critical
 points of $G(z)$. Thus, we arrive at the equation $G'(z)=0$, which is \eqref{eq:critical_equation}. We deform the
contours to pass through its complex critical points $z_c$ and $\overline{z_c}$. The contour itself is then the union of curves $\Im G(z)={\rm const}$ along which $\Re G(z)$ has maxima at $z=z_c$ and $z=\overline{z_c}$. The result is that the dominating contribution to the integral is given by small neighborhoods of these critical points. Near the critical point
$z_c$ we have
$$
 G(z)=G(z_c)+\frac{G''(z_c)}2 (z-z_c)^2+ o( (z-z_c)^2).
$$
Note that $G''(z_c)$ is non-zero, since its vanishing would mean a double critical point for
$G(z)$, which is impossible, as the argument of Lemma \ref{Lemma_complex_roots} explains\footnote{We also would like $G''(z_c)$ to remain bounded away from $0$ as $N\to\infty$, which follows from its convergence to a limiting value under Assumption \ref{Assumption_weak}.}. Hence,
making a change of variable $z=z_c+\frac{1}{\sqrt N \sqrt{G''(z_c)}} w$, the integral near $z_c$ becomes a Gaussian
integral and evaluates explicitly as $N\to\infty$ to
\begin{equation}
\label{eq_x34}
\frac{1}{\sqrt{N}} \sqrt{\frac{2\pi}{G''(z_c)}} \cdot \exp\bigl(N G(z_c)\bigr) \cdot \exp\left(\frac{N-k+1}{N} \cdot
\frac{\chi}{z_c}\right),
\end{equation}
where the last factor arose from the limit of the second factor in \eqref{eq_x1}. In principle, one should be careful in choosing the branch of $\sqrt{G''(z_c)}$ in  \eqref{eq_x34}, but the final asymptotic theorem is not sensitive to this aspect and we will not detail it. Similarly, the
contribution of the neighborhood of $\overline{z_c}$ is
\begin{equation}
\sqrt{\frac{2\pi}{G''(\overline{z_c})}} \cdot \exp\bigl(N G(\overline{z_c})\bigr) \cdot
\exp\left(\frac{N-k+1}{N} \cdot \frac{\chi}{\overline{z_c}}\right),
\end{equation}
Note that $G(\overline{z})=\overline{G(z)}$. Hence, we conclude that
\begin{multline} \label{eq_x2}
   P_k(x+
 \tfrac{1}{N}\chi)\sim \frac{1}{\sqrt{G''(z_c)}} \exp\bigl( \ii N \Im G(z_c)\bigr) \cdot \exp\left(\frac{N-k+1}{N} \cdot
\frac{\chi}{z_c}\right) (1+r_1(\chi))\\+\frac{1}{\sqrt{G''(z_c)}}  \exp\bigl(- \ii N \Im G(z_c)\bigr) \cdot
\exp\left(\frac{N-k+1}{N} \cdot \frac{\chi}{\overline{z_c}}\right) (1+r_2(\chi)),
\end{multline}
where $\sim$ hides $\chi$--independent factors and $r_1(\chi)$, $r_2(\chi)$ are complex remainders,
which tend to $0$ as $N\to\infty$ (uniformly over $\chi$ belonging to compact sets). By the Hurwitz's
theorem (or by the Rouche's theorem) zeros of a uniformly convergent sequence of holomorphic functions
converge to those of the limiting function. Applying this statement to $P_k(x+
 \tfrac{1}{N}\chi)$ as a function of $\chi$ (after multiplication by a proper constant to get the
 right-hand side of \eqref{eq_x2}, and noting that the exponent $\ii N \Im
 G(z_c)$ in $\exp\bigl(\ii N \Im
 G(z_c)\bigr)$ can be made bounded by using $2\pi \ii$ periodicity of $\exp(\cdot)$), we conclude
 that the zeros of $P_k(x+
 \tfrac{1}{N}\chi)$ as $N\to\infty$ are the same as those of
 \begin{equation}
 \label{eq_x3}
  \exp\bigl( \ii N \Im G(z_c)\bigr) \cdot \exp\left(\frac{N-k+1}{N} \cdot
\frac{\chi}{z_c}\right)+\exp\bigl(- \ii N \Im G(z_c)\bigr) \cdot \exp\left(\frac{N-k+1}{N} \cdot
\frac{\chi}{\overline{z_c}}\right).
 \end{equation}
For fixed ratio $\frac{N-k+1}{N}$ the latter zeros form a lattice on the real line with step
$$
 u=\pi \left(\frac{N-k+1}{N} \Im \frac{1}{z_c} \right)^{-1}.
$$
On the other hand, if we increase $k$ by $1$, then the change in $\frac{N-k+1}{N}$ is negligible,
however, the definition of $G$ changes: $\exp(N G(z))$ is multiplied by $z$. We can still use the
same critical point $z_c$ in the asymptotic computation and only change $N G(z_c)$ in \eqref{eq_x3}
by adding a new term $\ln(z_c)$. We conclude that $k\to k+1$ results in the shift of the lattice of
zeros to the left by
$$
 v=u\cdot \frac{1}{\pi} \Im \ln(z_c)=u \cdot \frac{1}{\pi} \arg (z_c).\qedhere
$$

\end{proof}

\begin{proof}[Proof of Theorem \ref{Theorem_edge_Airy}] 
We follow the same approach as in Theorem \ref{Theorem_bulk_lattice}. The only difference is that now we have a double critical point $z_c$ on the real line, instead of a pair of complex conjugate critical points. Our first task is to identify the location of this point. Here we rely on lemma, which we prove a bit later.

\begin{lemma}
\label{Lemma_critical_point_edge}
 Under assumptions of Theorem \ref{Theorem_edge_Airy}, the (unique) double critical point $z_c$ of \eqref{eq:critical_equation} satisfies $z_c>y_N-x>0$, and moreover the difference $z_c-(y_N-x)$ stays bounded away from $0$ as $N\to\infty$. The third derivative $G'''(z_c)$ is positive and stays bounded away from $0$ and $\infty$ as $N\to\infty$.
\end{lemma}
Next, we write:
 \begin{equation}
 \label{eq_x4}
   P_k(x+
 \tfrac{1}{N^{2/3}}\chi)\sim \oint \exp\bigl( N G(z)\bigr) \left(1-\frac{\chi}{ N^{2/3} z}\right)^{-N+k-1} dz,
 \end{equation}
 where
 $$
  G(z)=\frac{1}{N}\ln\bigl( P_N(z+x)\bigr) - \frac{N-k+1}{N} \ln z.
 $$
We  deform the integration contour to pass through $z_c$ and the integral becomes dominated by a small neighborhood of this point\footnote{We omit a standard justification of this fact.}. In this neighborhood we have:
$$
 G(z)=G(z_c)+ \frac{G'''(z_c)}{6} (z-z_c)^3+o(z-z_c)^3.
$$
 We make a change of variable
 $$
  z=z_c+ N^{-1/3} w.
 $$
 We need to find the asymptotic expansion of the second factor in the integrand of \eqref{eq_x4}:
\begin{multline*}
 \left(1-\frac{\chi}{ N^{2/3} z}\right)^{-N+k-1} =  \left(1-\frac{\chi}{ N^{2/3} z_c}\right)^{-N+k-1}  \cdot
 \left(\frac{ (N^{2/3} z-\chi)(N^{2/3} z_c)}{ (N^{2/3} z_c-\chi)(N^{2/3} z)}\right)^{-N+k-1}
 \\= \left(1-\frac{\chi}{ N^{2/3} z_c}\right)^{-N+k-1}  \cdot
 \left(\frac{ (z_c+ N^{-1/3} w-N^{-2/3}\chi)(z_c)}{ (z_c -N^{-2/3}\chi)(z_c+ N^{-1/3} w)}\right)^{-N+k-1}
\\= \left(1-\frac{\chi}{ N^{2/3} z_c}\right)^{-N+k-1}  \cdot
 \left(1+\frac{N^{-1}\chi w}{ z_c^2+ N^{-1/3} w z_c-N^{-2/3} z_c  -N^{-1}\chi w}\right)^{-N+k-1}.
\end{multline*}
As $N\to\infty$, the first factor is a function of (finite) $\chi$, which has no zeros and therefore can be ignored for our computations. The second factor asymptotically becomes
$$
 \exp\left(-\frac{N-k+1}{N} \cdot \frac{\chi w}{z_c^2}\right).
$$
We conclude that up to factors, which have no zeros (as functions of $\chi$),
 \begin{equation}
 \label{eq_x5}
   P_k(x+
 \tfrac{1}{N^{2/3}}\chi)\sim \int \exp\left(  \frac{G'''(z_c)}{6} w^3 -\frac{N-k+1}{N} \cdot \frac{\chi w}{z_c^2}  \right)  dw.
 \end{equation}
 We have some freedom in choosing the contour of integration, as long as it extends to infinity in both directions in such a way that the integrand decays. Our choice is to integrate over the unions of two rays $\arg(w)=\pm \frac{\pi}{3}$, which gives real negative values for $w^3$ (recall that $G'''(z_c)>0$).

 We now recall the contour integral representation of the Airy function:
 \begin{equation}
  \label{eq:Airy_contour_integral}
  \Ai(\xi)=\frac{1}{2\pi\ii} \int \exp\left(\frac{\tilde w^3}{3} -\xi \tilde w \right) d \tilde w,
 \end{equation}
 where the integration contour is the same as in \eqref{eq_x5}. Changing the integration variable in \eqref{eq_x5} by
 $$
  w=\left(\frac{2}{G'''(z_c)}\right)^{1/3} \tilde w,
 $$
 we conclude that
  \begin{multline*}
   P_k(x+
 \tfrac{1}{N^{2/3}}\chi)\sim \int \exp\left(  \frac{\tilde w^3}{3} -\left(\frac{2}{G'''(z_c)}\right)^{1/3} \frac{N-k+1}{N} \cdot \frac{\chi \tilde w}{z_c^2}  \right)  d\tilde w\\ \sim \Ai\left( \chi \cdot \left(\frac{2}{G'''(z_c)}\right)^{1/3} \frac{N-k+1}{N} \cdot \frac{1}{z_c^2} \right).\qedhere
 \end{multline*}
\end{proof}

\begin{proof}[Proof of Lemma \ref{Lemma_critical_point_edge}]
 Note that roots of a polynomial smoothly depend on the coefficients of this polynomial as long as roots do not merge together. We use this observation to deform from the $x=+\infty$ case down to the first $x$ when a double root of \eqref{eq:critical_equation} arises. Recall from Lemma \ref{Lemma_complex_roots} that \eqref{eq:critical_equation} has $N$ roots (with multiplicity). When $x$ is large positive, we can pin down all these roots on the real line: following the sign changes of  \eqref{eq_x6}, we locate $N-1$ roots inside segments $(y_{i+1}-x, y_{i}-x)$, $1\le i <N$, and another root inside the ray $(0,+\infty)$. This remains true as long as $x>y_N$. Let us investigate what happens when $x$ becomes slightly smaller, i.e.\ for $x=y_N-\eps$. We claim that we now have two distinct roots on the segment $(y_N-x,+\infty)$. Indeed, the function in the left-hand side of \eqref{eq_x6} is positive at $z=y_N-x+0$, becomes negative for slightly larger $z$ (because of the contribution of $-\frac{N-k+1}{N} \cdot \frac{1}{z}$; in this part the lower bound on the spacings $y_{N+1-i}-y_i$ in Assumption \ref{Assumption_strong} becomes important), and then it is again positive for very large $z\to+\infty$. When we further decrease $x$, all other roots continue to belong to the segments $(y_{i+1}-x, y_i-x)$ and, therefore, the first appearance of a double root is when the above two roots on $(y_{N}-x,+\infty)$ merge. Hence, $z_c>y_N-x>0$.

 Now set $\delta(N)=z_c-(y_N-x)$. Our aim is to show that $\delta(N)$ is bounded away from $0$ as $N\to\infty$. We argue by contradiction and assume that $\delta(N)$ can become arbitrary small, i.e.\ there is a growing sequence $N_m$ such that $\lim_{m\to\infty} \delta(N_m)=0$. Then one can find a constant $D>0$ such that $y_{N_m}-x>D$ for all $m$. (Indeed, otherwise, passing to a further subsequence, if necessary, we would get $\lim_{m\to\infty} (y_{N_m}-x)=0$, and consequently the left-hand side of \eqref{eq_x6} would be negative at $z_c$ due to dominating contribution of $-\frac{N-k+1}{N} \cdot \frac{1}{z}$.) But then we can upper bound $G''(z_c)$ as:
 \begin{multline}
 \label{eq_x7}
   G''(z_c)=-\frac{1}{N} \sum_{i=1}^N \frac{1}{(z_c-(y_i-x))^2}+ \frac{N-k+1}{N}\cdot \frac{1}{z_c^2}\\< -\frac{1}{N} \sum_{i=1}^N \frac{1}{(\delta(N)+y_N-y_i)^2}+  \frac{N-k+1}{N} \frac{1}{D^2}.
 \end{multline}
 Since by Assumption \ref{Assumption_strong}, the empirical measure of $\{y_i\}$ converges to a measure $\rho$ supported on $[A,B]$, $y_N$ converges to $B$, and
 $$
  \int_A^B \frac{1}{(B-x)^2} \rho(dx)=+\infty,
 $$
 the inequality \eqref{eq_x7} implies that $G''(z_c)$ goes to $-\infty$ as $N\to\infty$, which contradicts $G''(z_c)=0$. Hence, our assumption was wrong and $\delta(N)$ is indeed bounded away from $0$.

 \smallskip

 Next, $G'''(z_c)$ is non-negative, since $G'(z)$ is a non-negative function on $z\in(y_N-x,+\infty)$ with a minimum $G'(z_c)=0$. $G'''(z_c)$ is bounded away from $\infty$ immediately from the formula
 $$
  G'''(z_c)= 2\frac{1}{N} \sum_{i=1}^N \frac{1}{(\delta(N)+y_N-y_i)^3}- 2\frac{N-k+1}{N}\cdot \frac{1}{z_c^3}
 $$
 and the facts that $\delta(N)$ is bounded away from $0$ and $z_c>\delta(N)$.

 It remains to show that $G'''(z_c)$ is bounded away from $0$. Indeed, otherwise, passing to a subsequence, if necessary, we would see a triple root at $z_c$ for the function $G(z)$. But (by the Hurwitz or by the Rouchet's theorem) this is impossible, since for finite $N$ we have shown that $G(z)$ has only a double root at $z_c$ and no other roots in a neighborhood.
\end{proof}

\end{document}